\documentclass[11pt,a4paper]{article}

\usepackage{epsf,epsfig,amsfonts,amsgen,amsmath,amstext,amsbsy,amsopn,amsthm,cases,listings,color
}
\usepackage{ebezier,eepic}
\usepackage{color}
\usepackage{multirow}
\usepackage{epstopdf}
\usepackage{graphicx}
\usepackage{pgf,tikz}
\usepackage{mathrsfs}
\usepackage[marginal]{footmisc}
\usepackage{enumitem}
\usepackage{appendix}

\usepackage{pgf,tikz,pgfplots}
\usepackage{subfigure}
\usepackage{mathrsfs}
\usetikzlibrary{arrows}

\definecolor{uuuuuu}{rgb}{0.27,0.27,0.27}
\definecolor{sqsqsq}{rgb}{0.1255,0.1255,0.1255}

\newtheorem{dfn}{Definition} [section]
\newtheorem{thm}[dfn]{Theorem}
\newtheorem{lemma}[dfn]{Lemma}
\newtheorem{prop}[dfn]{Proposition}

\newtheorem{coro}[dfn]{Corollary}

\newtheorem{claim}[dfn]{Claim}
\newtheorem{prob}[dfn]{Problem}
\newtheorem{obs}[dfn]{Observation}
\newtheorem{fact}[dfn]{Fact}
\newenvironment{pf}{\noindent{\bf Proof}}

\def\lc{\left\lceil}

\def\rc{\right\rceil}

\def\lf{\left\lfloor}

\def\rf{\right\rfloor}
\begin{document}
\title{\bf\Large The feasible region of hypergraphs}

\date{\today}

\author{Xizhi Liu\thanks{Department of Mathematics, Statistics, and Computer Science, University of Illinois, Chicago, IL, 60607 USA. Email: xliu246@uic.edu}
\and
Dhruv Mubayi\thanks{Department of Mathematics, Statistics, and Computer Science, University of Illinois, Chicago, IL, 60607 USA. Email: mubayi@uic.edu.
Research partially supported by NSF award DMS-1800832.}
}
\maketitle
\footnote{2010 MSC: 05D05, 05C35, 05C65}
\begin{abstract}
Let $\mathcal{F}$ be a family of $r$-uniform hypergraphs.
The feasible region $\Omega(\mathcal{F})$ of $\mathcal{F}$ is
the set of points $(x,y)$ in the unit square such that there exists a sequence of $\mathcal{F}$-free $r$-uniform hypergraphs
whose edge density approaches $x$ and whose shadow density approaches $y$.
The feasible region provides a lot of combinatorial information, for example,
the supremum of $y$ over all $(x,y) \in \Omega(\mathcal{F})$ is the Tur\'{a}n density $\pi(\mathcal{F})$,
and $\Omega(\emptyset)$ gives the Kruskal-Katona theorem.

We undertake a systematic study of $\Omega(\mathcal{F})$,
and prove that $\Omega(\mathcal{F})$ is completely determined by a left-continuous almost everywhere differentiable function;  and
moreover, there exists an $\mathcal{F}$ for which this function is not continuous.
We also extend some old related theorems.
For example, we generalize a result of Fisher and Ryan to hypergraphs and extend  a classical result of Bollob\'as by almost completely determining the feasible region for cancellative triple systems.
\end{abstract}

\section{Introduction}
Given a set $V$ and  an integer $r > 0$, let $\binom{V}{r} = \left\{ W \subset V: |W| = r \right\}$.
An $r$-uniform hypergraph (henceforth $r$-graph) $\mathcal{H}$ with vertex set $X$ is a subset of $\binom{X}{r}$,
and we denote $X$ by $V(\mathcal{H})$.
Let $v(\mathcal{H}) = |V(\mathcal{H})|$.
The shadow of an $r$-graph $\mathcal{H}$ is
\[
\partial \mathcal{H} = \left\{ A \in \binom{V(\mathcal{H})}{r-1}: \text{$\exists B\in \mathcal{H}$ such that $A\subset B$} \right\}.
\]
The classical Kruskal-Katona theorem gives a tight upper bound for $|\mathcal{H}|$ as a function of $|\partial \mathcal{H}|$.
The following technically simpler version of the Kruskal-Katona theorem serves as a good starting point for the work in this paper.

\begin{thm}[see Lov\'{a}sz \cite{LO93}]
Let $\mathcal{H}$ be an $r$-graph, and suppose that $|\partial\mathcal{H}| = \binom{z}{r-1}$ for some real number $z \ge r$.
Then $|\mathcal{H}| \le \binom{z}{r}$.
\end{thm}

Let $\mathcal{F}$ be a family of $r$-graphs.
Then $\mathcal{H}$ is $\mathcal{F}$-free if it does not contain any member of $\mathcal{F}$ as a (not necessarily induced) subgraph.
The Tur\'{a}n number $\textrm{ex}(n,\mathcal{F})$ of $\mathcal{F}$ is the maximum number of edges in an $\mathcal{F}$-free $r$-graph on $n$ vertices.
The Tur\'{a}n density of $\mathcal{F}$ is $\pi(\mathcal{F}) = \lim_{n\to \infty} \textrm{ex}(n,\mathcal{F}) / \binom{n}{r}$.
Determining $\pi(\mathcal{F})$ for $r\ge 3$ is known to be notoriously hard in general,
and we refer the reader to a survey by Keevash $\cite{KE11}$ for results before 2011.

In this paper, we combine the Kruskal-Katona theorem and the hypergraph Tur\'{a}n problem by considering the following more general question.
\begin{align}\label{feasible-region-problem}
&\mbox{If $\mathcal{H}$ is $\mathcal{F}$-free, what are the possible values of $|\mathcal{H}|$ for fixed $|\partial\mathcal{H}|$?} \tag{$\ast$}
\end{align}
In particular, if we let $\mathcal{F} = \emptyset$,
then the upper bound for $|\mathcal{H}|$ in $(\ast)$ follows from the Kruskal-Katona theorem.
If $\mathcal{F} \neq \emptyset$, then $(\ast)$ is closely related to the hypergraph Tur\'{a}n problem.
In fact, $\textrm{ex}(n,\mathcal{F})$ gives a universal upper bound for $|\mathcal{H}|$ no matter what $|\partial \mathcal{H}|$ is,
and it is tight for some (at least one) values of $|\partial \mathcal{H}|$.
However, the upper bound given by $\textrm{ex}(n,\mathcal{F})$
gives us a rather limited picture of the relationship between the shadow and size  of an $\mathcal{F}$-free hypergraph. Our objective in this work is to provide a much more detailed view of this relationship.

An analogous question has been studied extensively in extremal graph theory.
Given two graphs $H$ and $G$, let $n(H;G)$ denote the number of copies of $H$ in $G$.
The density of $H$ in $G$ is $\rho(H;G) = n(H;G)/\binom{v(G)}{v(H)}$. For fixed graphs $H_1$ and $H_2$ and (large) graph $G$, the following problem is a cornerstone of extremal graph theory:
\begin{align}
&\text{\indent What are the possible values of $\rho(H_2;G)$ if $\rho(H_1;G)$ is fixed?} \tag{$\star$}
\end{align}
Even for $(H_1,H_2) = (K_2,K_t)$ with $t \ge 3$, question $(\star)$ is known to be highly nontrivial and
was asymptotically solved for $t=3$ by Razbarov \cite{RA08}, $t=4$ by Nikiforov \cite{NI11},
and for all $t$ only recently by Reiher $\cite{RE16}$.
We refer the reader to $\cite{LS83,BO75,RA07}$ for the history of $(\star)$.

The main difficulty in $(\star)$ is to determine the lower bound for $\rho(H_2;G)$.
However, it will be shown later that the main difficulty in $(\ast)$ is to determine the upper bound for  $|\mathcal{H}|$.
In order to state our results formally we need some definitions.

\begin{dfn}[Feasible Region]\label{def-feasible-region}
Fix $r \ge 3$.
\begin{enumerate}[label=(\alph*)]
\item Given an $r$-graph $\mathcal{H}$, its edge density $d(\mathcal{H}) = |\mathcal{H}| / \binom{v(\mathcal{H})}{r}$
and its shadow density $d(\partial\mathcal{H}) = |\partial\mathcal{H}| / \binom{v(\mathcal{H})}{r-1}$.
\item An $r$-graph sequence $\left(\mathcal{H}_{k}\right)_{k=1}^{\infty}$ is good if $v(\mathcal{H}) \to \infty$ as $k \to \infty$ and
both $\lim_{k \to \infty}d(\mathcal{H}_{k})$ and $\lim_{k \to \infty}d(\partial \mathcal{H}_{k})$ exist.
\item Let $\left(\mathcal{H}_{k}\right)_{k=1}^{\infty}$ be a good sequence of $\mathcal{F}$-free $r$-graphs, and $(x,y) \in [0,1]^{2}$.
Then $\left(\mathcal{H}_{k}\right)_{k=1}^{\infty}$ realizes $(x,y)$ if
$\lim_{k \to \infty}d(\partial \mathcal{H}_{k}) =x  \text{ and } \lim_{k \to \infty}d(\mathcal{H}_{k}) = y$.
\item The feasible region $\Omega(\mathcal{F})$ of $\mathcal{F}$ is
the collection of all points $(x,y)\in [0,1]^2$ that can be realized by a good sequence of $\mathcal{F}$-free $r$-graphs.
\end{enumerate}
\end{dfn}

As mentioned earlier, the upper bound given by $\textrm{ex}(n,\mathcal{F})$
gives us a rather limited picture of $\Omega(\mathcal{F})$, since it only determines
$$\sup\{y : \exists x\in [0,1] \text{ such that } (x,y) \in \Omega(\mathcal{F})\}.$$
As indicated by $(\ast)$, in this paper we study $\Omega(\mathcal{F})$.
Our results are of two flavors.
\begin{itemize}
\item
We prove some general results about the shape of $\Omega(\mathcal{F})$.
Our main results here are Theorems \ref{left-cont-and-diff} and \ref{example-discont}
which state that the boundary of $\Omega(\mathcal{F})$ is completely determined by a
left-continuous  almost everywhere differentiable
function $g(\mathcal{F})$ with at most countably many jump discontinuities,
and give examples showing that $g(\mathcal{F})$ can indeed be discontinuous.

\item
We study $\Omega(\mathcal{F})$ for some specific choices of $\mathcal{F}$ for which
$\textrm{ex}(n,\mathcal{F})$ has been investigated by many researchers.
We focus on two specific families: cancellative hypergraphs and hypergraphs without expansions of cliques.
Our results, which go  beyond  determining just the Tur\'an density,  are summarized in Corollaries \ref{cancellative-3-all} and \ref{feasible-region-K} (see Figures 6 and 7).
\end{itemize}
Regarding our  results on the shape of $\Omega({\cal F})$, there are (at least) two  previous works of a similar flavor: Razborov~\cite{RA08} determined the closure
of the set of points defined by the homomorphism density of the edge and the triangle in finite graphs
(and showed that the boundary is  almost everywhere differentiable)
and Hatami-Norine~\cite{HN19} constructed examples which show that the restrictions of the boundary to certain hyperplanes of the region defined by the homomorphism densities of a list of given graphs can have nowhere differentiable parts.

Our work can be viewed as a continuation of a long line of research in asymptotic extremal combinatorics perhaps beginning with the seminal work of
Erd\H{o}s-Lov\'{a}sz-Spencer \cite{ELS79} and continuing today in different guises such as the graph limits paradigm of Lov\'{a}sz \cite{LO12}
or the method of Flag algebras of Razborov \cite{RA07}.

\subsection{General results about $\Omega(\mathcal{F})$}
In this section we state some general results about feasible regions.

\begin{prop}\label{closed}
The region $\Omega(\mathcal{F})$ is closed for all $r \ge 3$ and all (possibly infinite) families $\mathcal{F}$ of $r$-graphs.
\end{prop}

\begin{dfn}[Projection of the feasible region]
The projection of $\Omega(\mathcal{F})$ on the $x$-axis is
\[
{\rm proj}\Omega(\mathcal{F}) = \left\{ x : \text{$\exists y \in [0,1]$ such that $(x,y) \in \Omega(\mathcal{F})$} \right\}.
\]
\end{dfn}

Note that it is not necessarily true that ${\rm proj}\Omega(\mathcal{F}) = [0,1]$ in general.
Later we will present an example of $\mathcal{F}$, which shows ${\rm proj}\Omega(\mathcal{F}) = [0,(\ell)_{r-1}/\ell^{r-1}]$ for $\ell \ge 3$.
On the other hand, by removing edges one by one from $\mathcal{H}$ one can reduce the edge density of $\partial \mathcal{H}$
continuously (in the limit sense) to $0$. This yields the following observation.

\begin{obs}\label{projection-x-interval}
For every family $\mathcal{F}$ of $r$-graphs with $r \ge 3$
there exists $\hat{c} \in [0,1]$ such that ${\rm proj}\Omega(\mathcal{F}) = [0,\hat{c}]$.
\end{obs}

Proposition \ref{closed} enables us to define the following function.
\begin{dfn}[Boundary of the feasible region]
Given a family $\mathcal{F}$ of $r$-graphs with $r \ge 3$,
let $g(\mathcal{F}) : {\rm proj}\Omega(\mathcal{F}) \to [0,1]$ be defined by
\[
g({\mathcal{F}})(x) = \max\left\{y: (x,y) \in \Omega(\mathcal{F}) \right\},
\]
for all $x \in {\rm proj}\Omega(\mathcal{F})$.
\end{dfn}

Here we abuse notation by writing $g(\mathcal{F},x)$ for $g(\mathcal{F})(x)$.
Our next result shows that $\Omega(\mathcal{F})$ is determined by ${\rm proj}\Omega(\mathcal{F})$ and $g(\mathcal{F})$.
\begin{prop}\label{only-boundary-matters}
Let $r\ge 3$ and let $\mathcal{F}$ be a family of $r$-graphs.
If $(x_0,y_0) \in \Omega(\mathcal{F})$, then $(x_0, y) \in \Omega(\mathcal{F})$ for all $y \in [0,y_0]$.
\end{prop}

Combining the Kruskal-Katona theorem with some further observations yields the following universal upper bound for $g(\mathcal{F},x)$.
\begin{prop}\label{universal-upper-bound}
Let $r \ge 3$ and $\mathcal{F}$ be a family of $r$-graphs.
Then $g({\mathcal{F}},x) \le x^{r/(r-1)}$ for all $x \in {\rm proj}\Omega(\mathcal{F})$.
In particular, ${\rm proj}\Omega(\emptyset) = [0,1]$ and
$g(\emptyset, x) = x^{r/(r-1)}$ for all $x \in [0,1]$.
\end{prop}

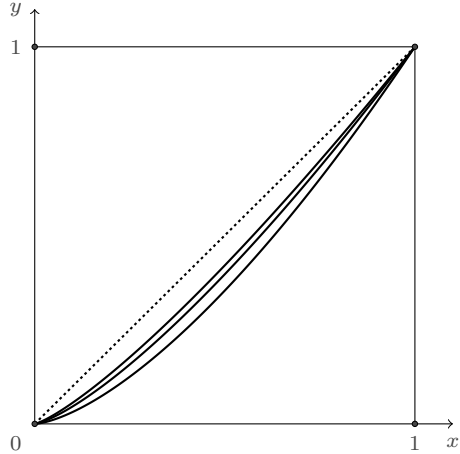
\begin{figure}[htbp]
\centering
\begin{tikzpicture}[xscale=5,yscale=5]
\draw [line width=0.8pt,dash pattern=on 1pt off 1.2pt,domain=0:1] plot(\x,\x);
\draw [->] (0,0)--(1.1,0);
\draw [->] (0,0)--(0,1.1);
\draw (0,1)--(1,1);
\draw (1,0)--(1,1);
\draw[line width=0.8pt]
(0.,0.)--(0.01,0.001)--(0.02,0.00282843)--(0.03,0.00519615)--(0.04,0.008)--(0.05,0.0111803)--
(0.06,0.0146969)--(0.07,0.0185203)--(0.08,0.0226274)--(0.09,0.027)--(0.1,0.0316228)--(0.11,0.0364829)--
(0.12,0.0415692)--(0.13,0.0468722)--(0.14,0.0523832)--(0.15,0.0580948)--(0.16,0.064)--(0.17,0.0700928)--
(0.18,0.0763675)--(0.19,0.0828191)--(0.2,0.0894427)--(0.21,0.0962341)--(0.22,0.103189)--(0.23,0.110304)--
(0.24,0.117576)--(0.25,0.125)--(0.26,0.132575)--(0.27,0.140296)--(0.28,0.148162)--(0.29,0.15617)--
(0.3,0.164317)--(0.31,0.172601)--(0.32,0.181019)--(0.33,0.189571)--(0.34,0.198252)--(0.35,0.207063)--
(0.36,0.216)--(0.37,0.225062)--(0.38,0.234248)--(0.39,0.243555)--(0.4,0.252982)--(0.41,0.262528)--
(0.42,0.272191)--(0.43,0.28197)--(0.44,0.291863)--(0.45,0.301869)--(0.46,0.311987)--(0.47,0.322216)--
(0.48,0.332554)--(0.49,0.343)--(0.5,0.353553)--(0.51,0.364213)--(0.52,0.374977)--(0.53,0.385846)--
(0.54,0.396817)--(0.55,0.407891)--(0.56,0.419066)--(0.57,0.430341)--(0.58,0.441715)--(0.59,0.453188)--
(0.6,0.464758)--(0.61,0.476425)--(0.62,0.488188)--(0.63,0.500047)--(0.64,0.512)--(0.65,0.524047)--
(0.66,0.536187)--(0.67,0.548419)--(0.68,0.560742)--(0.69,0.573157)--(0.7,0.585662)--(0.71,0.598257)--
(0.72,0.61094)--(0.73,0.623712)--(0.74,0.636572)--(0.75,0.649519)--(0.76,0.662553)--(0.77,0.675672)--
(0.78,0.688877)--(0.79,0.702167)--(0.8,0.715542)--(0.81,0.729)--(0.82,0.742542)--(0.83,0.756166)--
(0.84,0.769873)--(0.85,0.783661)--(0.86,0.797531)--(0.87,0.811482)--(0.88,0.825513)--(0.89,0.839624)--
(0.9,0.853815)--(0.91,0.868085)--(0.92,0.882433)--(0.93,0.89686)--(0.94,0.911364)--(0.95,0.925945)--
(0.96,0.940604)--(0.97,0.955339)--(0.98,0.970151)--(0.99,0.985038)--(1.,1.);
\draw[line width=0.8pt]
(0.,0.)--(0.01,0.00215443)--(0.02,0.00542884)--(0.03,0.0093217)--(0.04,0.0136798)--(0.05,0.0184202)--
(0.06,0.0234892)--(0.07,0.028849)--(0.08,0.034471)--(0.09,0.0403326)--(0.1,0.0464159)--(0.11,0.0527056)--
(0.12,0.0591891)--(0.13,0.0658554)--(0.14,0.0726949)--(0.15,0.0796994)--(0.16,0.0868614)--(0.17,0.0941742)--
(0.18,0.101632)--(0.19,0.109229)--(0.2,0.116961)--(0.21,0.124822)--(0.22,0.13281)--(0.23,0.140919)--
(0.24,0.149147)--(0.25,0.15749)--(0.26,0.165945)--(0.27,0.174509)--(0.28,0.18318)--(0.29,0.191954)--
(0.3,0.20083)--(0.31,0.209805)--(0.32,0.218877)--(0.33,0.228044)--(0.34,0.237304)--(0.35,0.246655)--
(0.36,0.256096)--(0.37,0.265625)--(0.38,0.27524)--(0.39,0.28494)--(0.4,0.294723)--(0.41,0.304587)--
(0.42,0.314533)--(0.43,0.324557)--(0.44,0.33466)--(0.45,0.344839)--(0.46,0.355094)--(0.47,0.365424)--
(0.48,0.375827)--(0.49,0.386303)--(0.5,0.39685)--(0.51,0.407468)--(0.52,0.418155)--(0.53,0.428912)--
(0.54,0.439736)--(0.55,0.450627)--(0.56,0.461584)--(0.57,0.472607)--(0.58,0.483694)--(0.59,0.494845)--
(0.6,0.50606)--(0.61,0.517336)--(0.62,0.528675)--(0.63,0.540075)--(0.64,0.551535)--(0.65,0.563055)--
(0.66,0.574635)--(0.67,0.586273)--(0.68,0.597969)--(0.69,0.609722)--(0.7,0.621533)--(0.71,0.6334)--
(0.72,0.645322)--(0.73,0.6573)--(0.74,0.669333)--(0.75,0.68142)--(0.76,0.693561)--(0.77,0.705756)--
(0.78,0.718003)--(0.79,0.730303)--(0.8,0.742654)--(0.81,0.755057)--(0.82,0.767512)--(0.83,0.780017)--
(0.84,0.792573)--(0.85,0.805178)--(0.86,0.817833)--(0.87,0.830537)--(0.88,0.84329)--(0.89,0.856091)--
(0.9,0.86894)--(0.91,0.881837)--(0.92,0.894782)--(0.93,0.907773)--(0.94,0.920811)--(0.95,0.933895)--
(0.96,0.947025)--(0.97,0.960201)--(0.98,0.973423)--(0.99,0.986689)--(1.,1.);
\draw[line width=0.8pt]
(0.,0.)--(0.01,0.00316228)--(0.02,0.00752121)--(0.03,0.0124854)--(0.04,0.0178885)--(0.05,0.0236435)--
(0.06,0.0296954)--(0.07,0.0360058)--(0.08,0.0425464)--(0.09,0.049295)--(0.1,0.0562341)--(0.11,0.0633492)--
(0.12,0.0706279)--(0.13,0.0780601)--(0.14,0.0856367)--(0.15,0.0933499)--(0.16,0.101193)--(0.17,0.109159)--
(0.18,0.117244)--(0.19,0.125442)--(0.2,0.133748)--(0.21,0.142159)--(0.22,0.150671)--(0.23,0.159279)--
(0.24,0.167983)--(0.25,0.176777)--(0.26,0.185659)--(0.27,0.194628)--(0.28,0.20368)--(0.29,0.212813)--
(0.3,0.222025)--(0.31,0.231314)--(0.32,0.240679)--(0.33,0.250117)--(0.34,0.259626)--(0.35,0.269206)--
(0.36,0.278855)--(0.37,0.288571)--(0.38,0.298352)--(0.39,0.308199)--(0.4,0.318108)--(0.41,0.32808)--(0.42,0.338113)--
(0.43,0.348205)--(0.44,0.358357)--(0.45,0.368566)--(0.46,0.378833)--(0.47,0.389155)--(0.48,0.399532)--(0.49,0.409963)--
(0.5,0.420448)--(0.51,0.430986)--(0.52,0.441575)--(0.53,0.452215)--(0.54,0.462905)--(0.55,0.473645)--(0.56,0.484434)--
(0.57,0.495272)--(0.58,0.506157)--(0.59,0.517089)--(0.6,0.528067)--(0.61,0.539091)--(0.62,0.550161)--(0.63,0.561275)--
(0.64,0.572433)--(0.65,0.583635)--(0.66,0.594881)--(0.67,0.606169)--(0.68,0.617499)--(0.69,0.628871)--(0.7,0.640284)--
(0.71,0.651738)--(0.72,0.663232)--(0.73,0.674767)--(0.74,0.686341)--(0.75,0.697954)--(0.76,0.709606)--(0.77,0.721296)--
(0.78,0.733024)--(0.79,0.74479)--(0.8,0.756593)--(0.81,0.768433)--(0.82,0.78031)--(0.83,0.792223)--(0.84,0.804172)--
(0.85,0.816157)--(0.86,0.828177)--(0.87,0.840232)--(0.88,0.852321)--(0.89,0.864445)--(0.9,0.876603)--(0.91,0.888795)--
(0.92,0.901021)--(0.93,0.913279)--(0.94,0.925571)--(0.95,0.937896)--(0.96,0.950253)--(0.97,0.962642)--(0.98,0.975063)--(0.99,0.987516)--(1.,1.);
\begin{scriptsize}
\draw [fill=uuuuuu] (1,0) circle (0.2pt);
\draw[color=uuuuuu] (1,0-0.05) node {$1$};
\draw [fill=uuuuuu] (0,0) circle (0.2pt);
\draw[color=uuuuuu] (0-0.05,0-0.05) node {$0$};
\draw [fill=uuuuuu] (0,1) circle (0.2pt);
\draw[color=uuuuuu] (0-0.05,1) node {$1$};
\draw [fill=uuuuuu] (1,1) circle (0.2pt);
\draw[color=uuuuuu] (0-0.05,1+0.1) node {$y$};
\draw[color=uuuuuu] (1+0.1,0-0.05) node {$x$};
\end{scriptsize}
\end{tikzpicture}
\caption{Upper bounds for $g(\mathcal{F},x)$ when $r = 3,4,5$ given by Proposition \ref{universal-upper-bound}.}
\end{figure}

In $\cite{HN19}$, Hatami and Norin considered the region defined by the homomorphism densities of a list of given graphs,
which is a more general version of $(\star)$ (that generalizes $(\star)$ from two graphs $H_1,H_2$ to more graphs).
They constructed examples which show that the restrictions of the boundary to certain hyperplanes can have nowhere differential parts.
However, we will show in the next result that $g(\mathcal{F})$ is well-behaved.

\begin{dfn}[Left/right continuity]
Let $f: \mathbb{R} \to \mathbb{R}$.
Then $f$ is left-continuous (resp. right-continuous) at $x$ if for any $\epsilon > 0$ there exists $\delta > 0$
such that $|f(x') - f(x)| < \epsilon$ for all $x' \in (x-\delta, x)$ (resp. $|f(x') - f(x)| < \epsilon$ for all $x' \in (x,x+\delta)$).
If $f$ is  left-continuous (resp. right-continuous) at all $x \in \mathbb{R}$,
then we say $f$ is left-continuous (resp. right-continuous).
\end{dfn}

\begin{dfn}[Types of discontinuities]
Let $f: \mathbb{R} \to \mathbb{R}$ and $x \in \mathbb{R}$ be a discontinuity of $f$.
If $\lim_{x \to x^{-}} f(x)$ and $\lim_{x \to x^{+}}f(x)$ exist,
then $f$ is said to have the discontinuity of the first kind at $x$.
Otherwise, the discontinuity is said to be of the second kind.
Furthermore, suppose that $x$ is a discontinuity of the first kind of $f$.
Then $x$ is a removable discontinuity if $\lim_{x \to x^{-}} f(x) = \lim_{x \to x^{+}} f(x)$.
Otherwise, $x$ is a jump discontinuity.
\end{dfn}

\begin{thm}\label{left-cont-and-diff}
For any $r\ge 3$ and any family $\mathcal{F}$ of $r$-graphs,
$g({\mathcal{F}})$ is left-continuous, has at most countably many jump discontinuities,
and is almost everywhere differentiable.
\end{thm}

Furthermore, the next result shows that $g({\mathcal{F}})$ can indeed be discontinuous.
\begin{thm}\label{example-discont}
There exists a family $\mathcal{D}$ of $3$-graphs with ${\rm proj}\Omega(\mathcal{D}) = [0,1]$ and $g(\mathcal{D}, 2/3) = 2/9$,
but there exists an absolute constant $\delta_{0}>0$ such that $g(\mathcal{D}, 2/3 + \epsilon) < 2/9 - \delta_{0}$ for all $\epsilon \in (0, 10^{-8})$.
\end{thm}

Actually, Theorem \ref{example-discont} can be extended to $r \ge 4$.
However, since the proofs for $r = 3$ and $r \ge 4$ share a similar idea and the proof for $r \ge 4$ is rather technical,
we give the proof for $r \ge 4$ in the Appendix.
Also, the condition that $\epsilon < 10^{-8}$ in Theorem \ref{example-discont} is not necessary, but we include it to keep our proof simple.

\begin{figure}[htbp]
\centering
\begin{tikzpicture}[xscale=5,yscale=5]
\draw [->] (0,0)--(1.1,0);
\draw [->] (0,0)--(0,0.6);
\draw (0,0.5)--(1,0.5);
\draw (1,0)--(1,0.5);
\draw [line width=1pt,dash pattern=on 1pt off 1.2pt,domain=0:1] plot(\x,{2/9});
\draw[line width=1pt]
(2/3,2/9)--(0.6630236470626231,0.2204032031146704)
--(0.6585443512860356,0.21817345951646344)--(0.654065055509448,0.21595128623168966)
--(0.6495857597328606,0.21373670913866868)--(0.645106463956273,0.21152975438293659)
--(0.6406271681796856,0.20933044838187664)--(0.636147872403098,0.2071388178294629)
--(0.6316685766265105,0.20495488970112097)--(0.6271892808499231,0.2027786912587088)
--(0.6227099850733355,0.2006102500556217)--(0.6182306892967481,0.19844959394202646)
--(0.6137513935201605,0.19629675107022665)--(0.609272097743573,0.19415174990016584)
--(0.6047928019669855,0.19201461920507196)--(0.600313506190398,0.1898853880772468)
--(0.5958342104138105,0.1877640859340081)--(0.591354914637223,0.1856507425237863)
--(0.5868756188606356,0.18354538793238254)--(0.582396323084048,0.18144805258939525)
--(0.5779170273074605,0.17935876727481742)--(0.573437731530873,0.1772775631258134)
--(0.5689584357542855,0.17520447164368083)--(0.5644791399776979,0.17313952470100336)
--(0.5599998442011105,0.17108275454900254)--(0.5555205484245229,0.16903419382509416)
--(0.5510412526479355,0.1669938755606583)--(0.5465619568713479,0.164961833189029)
--(0.5420826610947604,0.16293810055371402)--(0.537603365318173,0.16092271191685042)
--(0.5331240695415854,0.15891570196790752)--(0.528644773764998,0.1569171058326439)
--(0.5241654779884104,0.1549269590823307)--(0.519686182211823,0.1529452977432499)
--(0.5152068864352354,0.15097215830647748)--(0.5107275906586479,0.14900757773796589)
--(0.5062482948820604,0.14705159348893285)--(0.5017689991054729,0.14510424350657278)
--(0.4972897033288854,0.1431655662451021)--(0.4928104075522979,0.14123560067715077)
--(0.48833111177571037,0.13931438630551737)--(0.48385181599912286,0.1374019631752988)
--(0.4793725202225354,0.13549837188641337)--(0.4748932244459478,0.1336036536065323)
--(0.4704139286693603,0.1317178500844364)--(0.46593463289277287,0.12984100366381837)
--(0.46145533711618536,0.1279731572975465)--(0.45697604133959785,0.12611435456241368)
--(0.45249674556301034,0.12426463967439019)--(0.44801744978642283,0.12242405750440374)
--(0.4435381540098353,0.12059265359467064)--(0.43905885823324786,0.11877047417560284)
--(0.43457956245666024,0.11695756618331694)--(0.43010026668007284,0.11515397727777338)
--(0.4256209709034853,0.1133597558615753)--(0.4211416751268978,0.11157495109945828)
--(0.4166623793503103,0.1097996129385029)--(0.4121830835737228,0.10803379212910673)
--(0.4077037877971353,0.10627754024675189)--(0.4032244920205478,0.10453090971460692)
--(0.39874519624396026,0.10279395382700543)--(0.3942659004673728,0.10106672677384512)
--(0.3897866046907853,0.09934928366595529)--(0.38530730891419773,0.09764168056148004)
--(0.3808280131376103,0.09594397449333411)--(0.3763487173610227,0.09425622349778398)
--(0.37186942158443526,0.0925784866442176)--(0.36739012580784774,0.0909108240661645)
--(0.36291083003126023,0.0892532969936367)--(0.3584315342546727,0.08760596778686161)
--(0.35395223847808527,0.08596889997148711)--(0.34947294270149776,0.0843421582753402)
--(0.3449936469249102,0.08272580866683046)--(0.34051435114832274,0.08111991839509318)
--(0.33603505537173517,0.07952455603197399)--(0.3315557595951477,0.07793979151596783)
--(0.3270764638185602,0.07636569619822708)--(0.3225971680419727,0.0748023428907691)
--(0.3181178722653852,0.07324980591701827)--(0.31363857648879767,0.07170816116483066)
--(0.3091592807122102,0.07017748614216043)--(0.30467998493562265,0.06865786003553928)
--(0.3002006891590352,0.06714936377155435)--(0.2957213933824477,0.0656520800815259)
--(0.2912420976058602,0.06416609356960143)--(0.28676280182927266,0.06269149078450184)
--(0.28228350605268515,0.06122836029517568)--(0.27780421027609764,0.05977679277063996)
--(0.27332491449951013,0.058336881064308725)--(0.2688456187229227,0.05690872030314095)
--(0.26436632294633516,0.05549240798196618)--(0.25988702716974765,0.05408804406338191)
--(0.25540773139316014,0.05269573108365302)--(0.25092843561657263,0.05131557426508493)
--(0.24644913983998512,0.04994768163538859)--(0.2419698440633976,0.04859216415460634)
--(0.23749054828681013,0.04724913585022642)--(0.23301125251022262,0.045918713961177304)
--(0.2285319567336351,0.04460101909146803)--(0.2240526609570476,0.04329617537431961)
--(0.21957336518046008,0.042004310647729165)--(0.21509406940387255,0.04072555664250949)
--(0.2106147736272851,0.03946004918396895)--(0.20613547785069758,0.03820792840853079)
--(0.20165618207411007,0.03696933899674451)--(0.19717688629752256,0.035744430424320926)
--(0.19269759052093502,0.03453335723302326)--(0.18821829474434756,0.03333627932348077)
--(0.18373899896776005,0.03215336227225981)--(0.17925970319117254,0.03098477767583987)
--(0.17478040741458506,0.02983070352450204)--(0.17030111163799752,0.028691324609560667)
--(0.16582181586141004,0.027566832967862585)--(0.16134252008482253,0.02645742836805515)
--(0.15686322430823502,0.025363318843811742)--(0.15238392853164753,0.024284721280009886)
--(0.14790463275506,0.023221862058822252)--(0.1434253369784725,0.0221749777738319)
--(0.138946041201885,0.02114431602166764)--(0.1344667454252975,0.020130136282327293)
--(0.12998744964871,0.019132710901390154)--(0.1255081538721225,0.01815232618980085)
--(0.12102885809553499,0.017189283659967112)--(0.11654956231894749,0.01624390142069605)
--(0.11207026654235998,0.015316515758221273)--(0.10759097076577247,0.014407482936513701)
--(0.10311167498918497,0.013517181257608283)--(0.09863237921259746,0.01264601343232816)
--(0.09415308343600996,0.01179440932425909)--(0.08967378765942245,0.010962829146125673)
--(0.08519449188283494,0.01015176720926047)--(0.08071519610624744,0.009361756355687817)
--(0.07623590032965993,0.008593373241477962)--(0.07175660455307244,0.007847244693951734)
--(0.06727730877648493,0.00712405544089494)--(0.06279801299989743,0.006424557617874988)
--(0.05831871722330992,0.0057495826171624865)--(0.053839421446722414,0.005100056076844675)
--(0.04936012567013491,0.0044770171693905856)--(0.04488082989354741,0.003881643919684239)
--(0.040401534116959896,0.0033152872187856555)--(0.03592223834037239,0.0027795178016977455)
--(0.03144294256378489,0.0022761933405199843)--(0.026963646787197385,0.0018075583218622703)
--(0.02248435101060988,0.0013764007788228095)--(0.018005055234022373,0.0009863159625249872)
--(0.01352575945743487,0.000642194917737317)--(0.009046463680847366,0.00035127127796605074)
--(0.00456716790425986,0.00012600704929554806);
\draw[line width=1pt,dash pattern=on 1pt off 1.2pt] (2/3,2/9-0.05) to [out=315,in=170] (1,0.05);
\begin{scriptsize}
\draw [fill=uuuuuu] (2/3,2/9) circle (0.2pt);
\draw [fill=white] (2/3,2/9-0.05) circle (0.2pt);
\draw [fill=uuuuuu] (2/3,0) circle (0.2pt);
\draw[color=uuuuuu] (2/3,0-0.07) node {$\frac{2}{3}$};
\draw [fill=uuuuuu] (1,0) circle (0.2pt);
\draw[color=uuuuuu] (1,0-0.07) node {$1$};
\draw [fill=uuuuuu] (0,0) circle (0.2pt);
\draw[color=uuuuuu] (0-0.05,0-0.05) node {$0$};
\draw [fill=uuuuuu] (0,2/9) circle (0.2pt);
\draw[color=uuuuuu] (0-0.08,2/9) node {$2/9$};
\draw [fill=uuuuuu] (0,1/2) circle (0.2pt);
\draw[color=uuuuuu] (0-0.08,1/2) node {$1/2$};
\draw[color=uuuuuu] (0-0.05,1/2+0.1) node {$y$};
\draw[color=uuuuuu] (1+0.1,0-0.05) node {$x$};
\end{scriptsize}
\end{tikzpicture}
\caption{The function $g(\mathcal{D})$ is discontinuous at $x =2/3$.}
\end{figure}
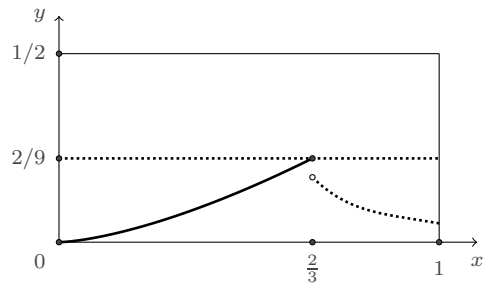

\subsection{Cancellative hypergraphs}
In this section we consider the feasible region of cancellative hypergraphs,
which is perhaps the first example of an extremal hypergraph problem that was well understood.
Our results are summarized in Corollary \ref{cancellative-3-all} stated at the end of this section.

\begin{dfn}
Let $\mathcal{T}_{r}$ be the collection of all $r$-graphs on at most $2r-1$ vertices with $3$ edges $A,B,C$ such that $A \triangle B \subset C$.
An $r$-graph is cancellative iff it is $\mathcal{T}_{r}$-free.
\end{dfn}

For $r = 2$ the family $\mathcal{T}_{2}$ comprises only one graph $K_3$.
For $r = 3$ the family $\mathcal{T}_{3}$ comprises two hypergraphs $K_{4}^{3-}$ and $F_5$,
where $K_{4}^{3-}$ is the $3$-graph on $4$ vertices with exactly $3$ edges,
and $F_{5}$ is the $3$-graph on $5$ vertices with edge set $\{123, 124, 345\}$.

Let $[n] = \{1,2,...,n\}$.
Fix $\ell \ge r\ge 2$.
Let $V_1\cup \cdots \cup V_{\ell}$ be a partition of $[n]$ with each part of size either $\lf n/{\ell} \rf$ or $\lc n/{\ell} \rc$.
The generalized Tur\'{a}n graph $T_{r}(n,{\ell})$ is the collection of all $r$-sets that intersect each $V_i$ on at most one vertex.
Notice that $T_{2}(n, \ell)$ is just the ordinary Tur\'{a}n graph.
Let
\[
t_{r}(n,{\ell}) = |T_{r}(n,{\ell})| \approx \binom{\ell}{r}\left(\frac{n}{\ell}\right)^{r}.
\]

In $\cite{BO74}$, Bollob\'{a}s proved that $\textrm{ex}(n,\mathcal{T}_{3}) \le t_{3}(n,3)$ and $T_{3}(n,3)$ is the unique $\mathcal{T}_{3}$-free
$3$-graph on $n$ vertices with exactly $t_{3}(n,3)$ edges.
Therefore, $g(\mathcal{T}_3 ,x) \le 2/9$ for all $x\in {\rm proj}\Omega(\mathcal{T}_3)$.
Later, Keevash and the second author $\cite{KM04}$ proved a stability theorem for $\mathcal{T}_{3}$-free hypergraphs.
The first author $\cite{LIU19}$ gave a new proof to both the exact and the stability result for $\mathcal{T}_{3}$-free hypergraphs.
Moreover, $\cite{LIU19}$ proves that a $\mathcal{T}_{3}$-free $3$-graph $\mathcal{H}$ on $n$-vertices satisfies the inequality
\[
4\left(\frac{ 3|\mathcal{H}|/|\partial \mathcal{H}| }{ n-3|\mathcal{H}|/|\partial \mathcal{H}| } \right)^2 |\partial \mathcal{H}|\le n^2-2| \partial \mathcal{H}|,
\]
which implies
\begin{equation}\label{cancellative-3-XL19}
g(\mathcal{T}_{3},x) \le \frac{ \sqrt{2(1-x) x^3}+x^2-x}{3 x-1}, \text{ for all } x \in {\rm proj}\Omega(\mathcal{T}_3).
\end{equation}

\begin{figure}[htbp]
\centering
\begin{tikzpicture}[xscale=5,yscale=5]
\draw [->] (0,0)--(1.1,0);
\draw [->] (0,0)--(0,0.6);
\draw (0,0.5)--(1,0.5);
\draw (1,0)--(1,0.5);
\draw [line width=1pt,dash pattern=on 1pt off 1.2pt,domain=0:1] plot(\x,{2/9});
\draw [line width=1pt,dash pattern=on 1pt off 1.2pt] (2/3,0) -- (2/3,2/9);
\draw[line width=1pt,color=sqsqsq,fill=sqsqsq,fill opacity=0.25]
(0., 0.)--(0.01, 0.00875554)--(0.02, 0.0166385)--(0.03, 0.0240248)--
(0.04, 0.0310396)--(0.05, 0.0377517)--(0.06, 0.0442056)--(0.07, 0.0504326)--
(0.08, 0.0564562)--(0.09, 0.0622945)--(0.1, 0.0679623)--
(0.11, 0.0734713)--(0.12, 0.0788316)--(0.13, 0.0840514)--(0.14,0.0891381)--
(0.15, 0.0940977)--(0.16, 0.0989356)--(0.17, 0.103657)--
(0.18, 0.108265)--(0.19, 0.112764)--(0.2, 0.117157)--
(0.21, 0.121448)--(0.22, 0.125638)--(0.23, 0.12973)--(0.24, 0.133726)--
(0.25, 0.137628)--(0.26, 0.141437)--(0.27, 0.145156)--(0.28,0.148784)--
(0.29, 0.152325)--(0.3, 0.155778)--(0.31, 0.159144)--
(0.32, 0.162425)--(0.33, 0.16562)--(0.34, 0.168731)--(0.35,0.171758)--
(0.36, 0.174701)--(0.37, 0.177561)--(0.38, 0.180337)--
(0.39, 0.183031)--(0.4, 0.185641)--(0.41, 0.188167)--
(0.42, 0.190611)--(0.43, 0.19297)--(0.44, 0.195246)--
(0.45, 0.197437)--(0.46, 0.199544)--(0.47, 0.201564)--(0.48, 0.203499)--
(0.49, 0.205347)--(0.5, 0.207107)--(0.51, 0.208778)--(0.52, 0.210359)--
(0.53, 0.2118)--(0.54, 0.213248)--(0.55, 0.214553)--(0.56, 0.215762)--
(0.57, 0.216875)--(0.58, 0.21789)--(0.59, 0.218804)--(0.6, 0.219615)--
(0.61, 0.220322)--(0.62, 0.220922)--(0.63, 0.221412)--(0.64, 0.22179)--
(0.65, 0.222052)--(0.66, 0.222195)--(0.67, 0.222215)--
(0.68, 0.22211)--(0.69, 0.221873)--(0.7, 0.221502)--(0.71, 0.22099)--
(0.72, 0.220333)--(0.73, 0.219524)--(0.74, 0.218556)--
(0.75, 0.217423)--(0.76, 0.216117)--(0.77, 0.214628)--(0.78, 0.212947)--
(0.79, 0.211063)--(0.8, 0.208963)--(0.81, 0.206633)--(0.82, 0.204058)--
(0.83, 0.201219)--(0.84, 0.198096)--(0.85, 0.194664)--(0.86, 0.190895)--
(0.87, 0.186755)--(0.88, 0.182206)--(0.89, 0.177197)--(0.9, 0.171669)--
(0.91, 0.165547)--(0.92, 0.158735)--(0.93, 0.151103)--(0.94, 0.142476)--
(0.95, 0.1326)--(0.96, 0.121087)--(0.97, 0.107282)--(0.98, 0.0899124)--(0.99, 0.065688)--(1., 0.)--(0,0);
\begin{scriptsize}
\draw [fill=uuuuuu] (2/3,2/9) circle (0.2pt);
\draw [fill=uuuuuu] (2/3,0) circle (0.2pt);
\draw[color=uuuuuu] (2/3,0-0.07) node {$\frac{2}{3}$};
\draw [fill=uuuuuu] (1,0) circle (0.2pt);
\draw[color=uuuuuu] (1,0-0.07) node {$1$};
\draw [fill=uuuuuu] (0,0) circle (0.2pt);
\draw[color=uuuuuu] (0-0.05,0-0.05) node {$0$};
\draw [fill=uuuuuu] (0,2/9) circle (0.2pt);
\draw[color=uuuuuu] (0-0.08,2/9) node {$2/9$};
\draw [fill=uuuuuu] (0,1/2) circle (0.2pt);
\draw[color=uuuuuu] (0-0.08,1/2) node {$1/2$};
\draw[color=uuuuuu] (1+0.1,0-0.07) node {$x$};
\draw[color=uuuuuu] (0-0.08,1/2+0.1) node {$y$};
\end{scriptsize}
\end{tikzpicture}
\caption{$\Omega(\mathcal{T}_3)$ is contained in the dark area above according to $(\ref{cancellative-3-XL19})$.}
\end{figure}
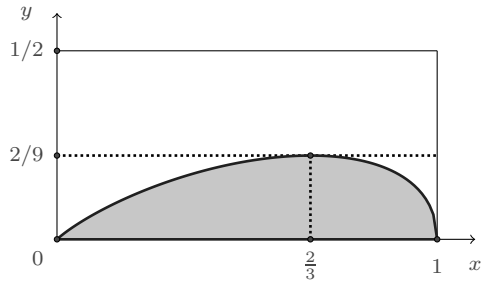

Our next result concerns cancellative $r$-graphs for $r \ge 3$,
and improves the bound in Proposition \ref{universal-upper-bound} as well as that in $(\ref{cancellative-3-XL19})$ for $x\in [0,2/3]$.
\begin{thm}\label{cancellative-r-left}
Let $r \ge 3$ and $x \in {\rm proj}\Omega(\mathcal{T}_{r})$. Then
\[
g(\mathcal{T}_{r},x) \le \left( \frac{x^{r}}{r!} \right)^{\frac{1}{r-1}}.
\]
Moreover, equality holds for all $x \in [0,(r-1)! / r^{r-2}]$.
\end{thm}

\begin{figure}[htbp]
\centering
\begin{tikzpicture}[xscale=8,yscale=8]
\draw [->] (0,0)--(1+0.05,0);
\draw [->] (0,0)--(0,0.5+0.05);
\draw (0,0.5)--(1,0.5);
\draw (1,0)--(1,0.5);
\draw [line width=0.8pt,dash pattern=on 1pt off 1.2pt,domain=0:1] plot(\x,2/9);
\draw [line width=0.8pt,dash pattern=on 1pt off 1.2pt,domain=0:1] plot(\x,3/32);
\draw [line width=0.5pt,dash pattern=on 1pt off 1.2pt] (2/3,0) -- (2/3,2/9+0.05);
\draw [line width=0.5pt,dash pattern=on 1pt off 1.2pt] (3/8,0) -- (3/8,3/32+0.05);
\draw [line width=0.5pt]
(0.,0.)--(0.01,0.000408248)--(0.02,0.0011547)--(0.03,0.00212132)--(0.04,0.00326599)--(0.05,0.00456435)--
(0.06,0.006)--(0.07,0.00756086)--(0.08,0.0092376)--(0.09,0.0110227)--(0.1,0.0129099)--(0.11,0.0148941)--
(0.12,0.0169706)--(0.13,0.0191355)--(0.14,0.0213854)--(0.15,0.0237171)--(0.16,0.0261279)--(0.17,0.0286153)--
(0.18,0.0311769)--(0.19,0.0338107)--(0.2,0.0365148)--(0.21,0.0392874)--(0.22,0.0421268)--(0.23,0.0450315)--
(0.24,0.048)--(0.25,0.051031)--(0.26,0.0541233)--(0.27,0.0572756)--(0.28,0.0604869)--(0.29,0.063756)--(0.3,0.067082)--
(0.31,0.0704639)--(0.32,0.0739008)--(0.33,0.0773919)--(0.34,0.0809362)--(0.35,0.084533)--(0.36,0.0881816)--
(0.37,0.0918813)--(0.38,0.0956312)--(0.39,0.0994309)--(0.4,0.10328)--(0.41,0.107177)--(0.42,0.111122)--(0.43,0.115114)--
(0.44,0.119153)--(0.45,0.123238)--(0.46,0.127368)--(0.47,0.131544)--(0.48,0.135765)--(0.49,0.140029)--(0.5,0.144338)--
(0.51,0.148689)--(0.52,0.153084)--(0.53,0.157521)--(0.54,0.162)--(0.55,0.166521)--(0.56,0.171083)--(0.57,0.175686)--
(0.58,0.180329)--(0.59,0.185013)--(0.6,0.189737)--(0.61,0.1945)--(0.62,0.199302)--(0.63,0.204143)--(0.64,0.209023)--
(0.65,0.213941)--(0.66,0.218897)--(0.67,0.223891)--(0.68,0.228922)--(0.69,0.23399)--(0.7,0.239096)--(0.71,0.244237)--
(0.72,0.249415)--(0.73,0.254629)--(0.74,0.259879)--(0.75,0.265165)--(0.76,0.270486)--(0.77,0.275842)--(0.78,0.281233)--
(0.79,0.286659)--(0.8,0.292119)--(0.81,0.297613)--(0.82,0.303141)--(0.83,0.308703)--(0.84,0.314299)--(0.85,0.319928)--
(0.86,0.325591)--(0.87,0.331286)--(0.88,0.337014)--(0.89,0.342775)--(0.9,0.348569)--(0.91,0.354394)--(0.92,0.360252)--
(0.93,0.366141)--(0.94,0.372063)--(0.95,0.378016)--(0.96,0.384)--(0.97,0.390016)--(0.98,0.396062)--(0.99,0.40214)--(1.,0.408248);
\draw [line width=0.5pt]
(0.,0.)--(0.01,0.000746901)--(0.02,0.00188207)--(0.03,0.00323165)--(0.04,0.00474252)--(0.05,0.00638591)--(0.06,0.00814325)--
(0.07,0.0100014)--(0.08,0.0119504)--(0.09,0.0139825)--(0.1,0.0160915)--(0.11,0.018272)--(0.12,0.0205197)--(0.13,0.0228308)--
(0.14,0.0252019)--(0.15,0.0276302)--(0.16,0.0301132)--(0.17,0.0326484)--(0.18,0.0352338)--(0.19,0.0378676)--(0.2,0.040548)--
(0.21,0.0432735)--(0.22,0.0460426)--(0.23,0.048854)--(0.24,0.0517064)--(0.25,0.0545988)--(0.26,0.05753)--(0.27,0.060499)--
(0.28,0.0635049)--(0.29,0.0665468)--(0.3,0.0696238)--(0.31,0.0727353)--(0.32,0.0758804)--(0.33,0.0790584)--(0.34,0.0822687)--
(0.35,0.0855107)--(0.36,0.0887836)--(0.37,0.092087)--(0.38,0.0954204)--(0.39,0.098783)--(0.4,0.102175)--(0.41,0.105595)--
(0.42,0.109042)--(0.43,0.112518)--(0.44,0.11602)--(0.45,0.119549)--(0.46,0.123104)--(0.47,0.126685)--(0.48,0.130292)--
(0.49,0.133924)--(0.5,0.13758)--(0.51,0.141261)--(0.52,0.144966)--(0.53,0.148695)--(0.54,0.152448)--(0.55,0.156224)--(0.56,0.160022)--
(0.57,0.163844)--(0.58,0.167687)--(0.59,0.171553)--(0.6,0.175441)--(0.61,0.179351)--(0.62,0.183281)--(0.63,0.187234)--(0.64,0.191207)--
(0.65,0.1952)--(0.66,0.199215)--(0.67,0.203249)--(0.68,0.207304)--(0.69,0.211379)--(0.7,0.215473)--(0.71,0.219587)--(0.72,0.223721)--
(0.73,0.227873)--(0.74,0.232045)--(0.75,0.236235)--(0.76,0.240444)--(0.77,0.244672)--(0.78,0.248918)--(0.79,0.253182)--(0.8,0.257464)--
(0.81,0.261764)--(0.82,0.266082)--(0.83,0.270417)--(0.84,0.27477)--(0.85,0.27914)--(0.86,0.283527)--(0.87,0.287931)--(0.88,0.292352)--
(0.89,0.29679)--(0.9,0.301245)--(0.91,0.305716)--(0.92,0.310203)--(0.93,0.314707)--(0.94,0.319227)--(0.95,0.323763)--(0.96,0.328315)--
(0.97,0.332883)--(0.98,0.337467)--(0.99,0.342066)--(1.,0.346681);
\begin{scriptsize}
\draw [fill=uuuuuu] (2/3,2/9) circle (0.2pt);
\draw [fill=uuuuuu] (3/8,3/32) circle (0.2pt);
\draw [fill=uuuuuu] (2/3,0) circle (0.2pt);
\draw[color=uuuuuu] (2/3,0-0.07) node {$\frac{2}{3}$};
\draw [fill=uuuuuu] (3/8,0) circle (0.2pt);
\draw[color=uuuuuu] (3/8,0-0.07) node {$\frac{3}{8}$};
\draw [fill=uuuuuu] (1,0) circle (0.2pt);
\draw[color=uuuuuu] (1,0-0.07) node {$1$};
\draw [fill=uuuuuu] (0,0) circle (0.2pt);
\draw[color=uuuuuu] (0-0.07,0-0.07) node {$0$};
\draw [fill=uuuuuu] (0,2/9) circle (0.2pt);
\draw[color=uuuuuu] (0-0.08,2/9) node {$2/9$};
\draw [fill=uuuuuu] (0,3/32) circle (0.2pt);
\draw[color=uuuuuu] (0-0.08,3/32) node {$3/32$};
\draw [fill=uuuuuu] (0,1/2) circle (0.2pt);
\draw[color=uuuuuu] (0-0.08,1/2) node {$1/2$};
\draw[color=uuuuuu] (0-0.07,1/2+0.07) node {$y$};
\draw[color=uuuuuu] (1+0.07,0-0.07) node {$x$};
\end{scriptsize}
\end{tikzpicture}
\caption{Upper bounds for $g(\mathcal{T}_{r},x)$ when $r= 3,4$ given by Theorem \ref{cancellative-r-left}.}
\end{figure}
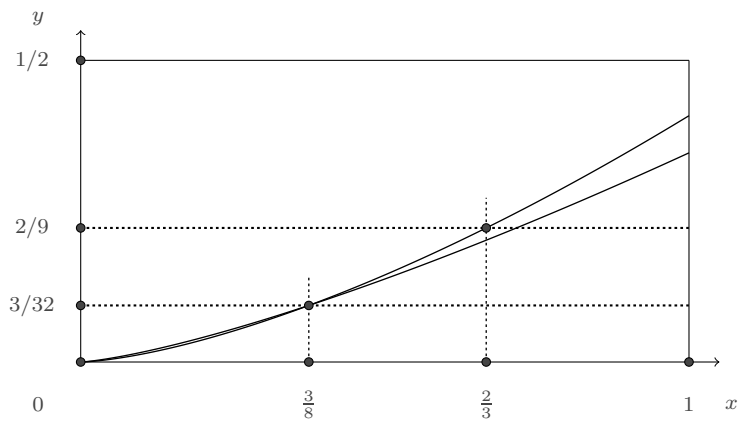

For $r = 3$, the bound given by Theorem \ref{cancellative-r-left} is not tight for any $x\in (2/3, 1]$ according to Bollob\'{a}s' theorem \cite{BO74}.
Our next result will present an improved bound for $g(\mathcal{T}_{3}, x)$ for $x \in (2/3, 1]$.
\begin{thm}\label{cancellative-3-right}
The inequality $g(\mathcal{T}_{3}, x) \le x(1-x)$ holds for all $x \in [0,1]$.
In particular, $g(\mathcal{T}_3,(k-1)/k) = (k-1)/k^2$ when $k \equiv 1 \text{ or } 3$ (mod $6$).
\end{thm}

The lower bound for $g(\mathcal{T}_{3}, (k-1)/k)$ when $k \equiv 1 \text{ or } 3$ (mod $6$) comes from the balanced blow up
of Steiner triple systems on $k$ vertices, this will be explained in more detail in Section 4.

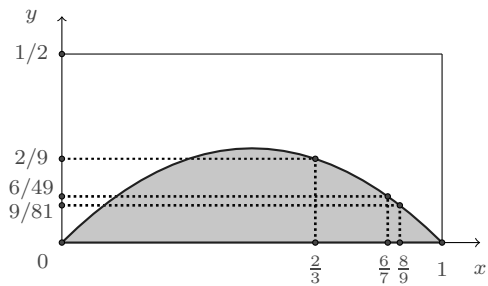
\begin{figure}[htbp]
\centering
\begin{tikzpicture}[xscale=5,yscale=5]
\draw [->] (0,0)--(1.1,0);
\draw [->] (0,0)--(0,0.6);
\draw (0,0.5)--(1,0.5);
\draw (1,0)--(1,0.5);
\draw [line width=0.8pt,dash pattern=on 1pt off 1.2pt,domain=0:2/3] plot(\x,2/9);
\draw [line width=1pt,dash pattern=on 1pt off 1.2pt] (2/3,0) -- (2/3,2/9);
\draw [line width=1pt,dash pattern=on 1pt off 1.2pt] (6/7,0) -- (6/7,6/49);
\draw [line width=1pt,dash pattern=on 1pt off 1.2pt] (8/9,0) -- (8/9,8/81);
\draw [line width=1pt,dash pattern=on 1pt off 1.2pt] (0,6/49) -- (6/7,6/49);
\draw [line width=1pt,dash pattern=on 1pt off 1.2pt] (0,8/81) -- (8/9,8/81);
\draw [line width=0.8pt,color=sqsqsq,fill=sqsqsq,fill opacity=0.25,domain=0:1] plot(\x,{\x*(1-\x)}) -- (0,0);
\begin{scriptsize}
\draw [fill=uuuuuu] (2/3,2/9) circle (0.2pt);
\draw [fill=uuuuuu] (6/7,6/49) circle (0.2pt);
\draw [fill=uuuuuu] (8/9,8/81) circle (0.2pt);
\draw [fill=uuuuuu] (2/3,0) circle (0.2pt);
\draw [fill=uuuuuu] (6/7,0) circle (0.2pt);
\draw [fill=uuuuuu] (8/9,0) circle (0.2pt);
\draw[color=uuuuuu] (2/3,0-0.07) node {$\frac{2}{3}$};
\draw[color=uuuuuu] (6/7-0.01,0-0.07) node {$\frac{6}{7}$};
\draw[color=uuuuuu] (8/9+0.01,0-0.07) node {$\frac{8}{9}$};
\draw [fill=uuuuuu] (1,0) circle (0.2pt);
\draw[color=uuuuuu] (1,0-0.07) node {$1$};
\draw [fill=uuuuuu] (0,0) circle (0.2pt);
\draw[color=uuuuuu] (0-0.05,0-0.05) node {$0$};
\draw [fill=uuuuuu] (0,2/9) circle (0.2pt);
\draw[color=uuuuuu] (0-0.08,2/9) node {$2/9$};
\draw [fill=uuuuuu] (0,1/2) circle (0.2pt);
\draw[color=uuuuuu] (0-0.08,1/2) node {$1/2$};
\draw [fill=uuuuuu] (0,6/49) circle (0.2pt);
\draw[color=uuuuuu] (0-0.08,6/49+0.02) node {$6/49$};
\draw [fill=uuuuuu] (0,8/81) circle (0.2pt);
\draw[color=uuuuuu] (0-0.08,8/81-0.02) node {$9/81$};
\draw[color=uuuuuu] (1+0.1,0-0.07) node {$x$};
\draw[color=uuuuuu] (0-0.08,1/2+0.1) node {$y$};
\end{scriptsize}
\end{tikzpicture}
\caption{$\Omega(\mathcal{T}_{3})$ is contained in the dark area above by Theorem \ref{cancellative-3-right}.}
\end{figure}

Combining Theorems \ref{cancellative-r-left} and \ref{cancellative-3-right} yields the following result for $g(\mathcal{T}_{3},x)$,
which provides a rather comprehensive picture of $\Omega(\mathcal{T}_3)$.

\begin{figure}[htbp]
\centering
\begin{tikzpicture}[xscale=5,yscale=5]
\draw [->] (0,0)--(1.1,0);
\draw [->] (0,0)--(0,0.6);
\draw (0,0.5)--(1,0.5);
\draw (1,0)--(1,0.5);
\draw [line width=1pt,dash pattern=on 1pt off 1.2pt,domain=0:2/3] plot(\x,{2/9});
\draw [line width=1pt,dash pattern=on 1pt off 1.2pt] (2/3,0) -- (2/3,2/9);
\draw [line width=1pt,dash pattern=on 1pt off 1.2pt] (6/7,0) -- (6/7,6/49);
\draw [line width=1pt,dash pattern=on 1pt off 1.2pt] (8/9,0) -- (8/9,8/81);
\draw [line width=1pt,dash pattern=on 1pt off 1.2pt] (0,6/49) -- (6/7,6/49);
\draw [line width=1pt,dash pattern=on 1pt off 1.2pt] (0,8/81) -- (8/9,8/81);
\draw[line width=1pt,color=sqsqsq,fill=sqsqsq,fill opacity=0.25]
(2/3,2/9)--(0.6630236470626231,0.2204032031146704)
--(0.6585443512860356,0.21817345951646344)--(0.654065055509448,0.21595128623168966)
--(0.6495857597328606,0.21373670913866868)--(0.645106463956273,0.21152975438293659)
--(0.6406271681796856,0.20933044838187664)--(0.636147872403098,0.2071388178294629)
--(0.6316685766265105,0.20495488970112097)--(0.6271892808499231,0.2027786912587088)
--(0.6227099850733355,0.2006102500556217)--(0.6182306892967481,0.19844959394202646)
--(0.6137513935201605,0.19629675107022665)--(0.609272097743573,0.19415174990016584)
--(0.6047928019669855,0.19201461920507196)--(0.600313506190398,0.1898853880772468)
--(0.5958342104138105,0.1877640859340081)--(0.591354914637223,0.1856507425237863)
--(0.5868756188606356,0.18354538793238254)--(0.582396323084048,0.18144805258939525)
--(0.5779170273074605,0.17935876727481742)--(0.573437731530873,0.1772775631258134)
--(0.5689584357542855,0.17520447164368083)--(0.5644791399776979,0.17313952470100336)
--(0.5599998442011105,0.17108275454900254)--(0.5555205484245229,0.16903419382509416)
--(0.5510412526479355,0.1669938755606583)--(0.5465619568713479,0.164961833189029)
--(0.5420826610947604,0.16293810055371402)--(0.537603365318173,0.16092271191685042)
--(0.5331240695415854,0.15891570196790752)--(0.528644773764998,0.1569171058326439)
--(0.5241654779884104,0.1549269590823307)--(0.519686182211823,0.1529452977432499)
--(0.5152068864352354,0.15097215830647748)--(0.5107275906586479,0.14900757773796589)
--(0.5062482948820604,0.14705159348893285)--(0.5017689991054729,0.14510424350657278)
--(0.4972897033288854,0.1431655662451021)--(0.4928104075522979,0.14123560067715077)
--(0.48833111177571037,0.13931438630551737)--(0.48385181599912286,0.1374019631752988)
--(0.4793725202225354,0.13549837188641337)--(0.4748932244459478,0.1336036536065323)
--(0.4704139286693603,0.1317178500844364)--(0.46593463289277287,0.12984100366381837)
--(0.46145533711618536,0.1279731572975465)--(0.45697604133959785,0.12611435456241368)
--(0.45249674556301034,0.12426463967439019)--(0.44801744978642283,0.12242405750440374)
--(0.4435381540098353,0.12059265359467064)--(0.43905885823324786,0.11877047417560284)
--(0.43457956245666024,0.11695756618331694)--(0.43010026668007284,0.11515397727777338)
--(0.4256209709034853,0.1133597558615753)--(0.4211416751268978,0.11157495109945828)
--(0.4166623793503103,0.1097996129385029)--(0.4121830835737228,0.10803379212910673)
--(0.4077037877971353,0.10627754024675189)--(0.4032244920205478,0.10453090971460692)
--(0.39874519624396026,0.10279395382700543)--(0.3942659004673728,0.10106672677384512)
--(0.3897866046907853,0.09934928366595529)--(0.38530730891419773,0.09764168056148004)
--(0.3808280131376103,0.09594397449333411)--(0.3763487173610227,0.09425622349778398)
--(0.37186942158443526,0.0925784866442176)--(0.36739012580784774,0.0909108240661645)
--(0.36291083003126023,0.0892532969936367)--(0.3584315342546727,0.08760596778686161)
--(0.35395223847808527,0.08596889997148711)--(0.34947294270149776,0.0843421582753402)
--(0.3449936469249102,0.08272580866683046)--(0.34051435114832274,0.08111991839509318)
--(0.33603505537173517,0.07952455603197399)--(0.3315557595951477,0.07793979151596783)
--(0.3270764638185602,0.07636569619822708)--(0.3225971680419727,0.0748023428907691)
--(0.3181178722653852,0.07324980591701827)--(0.31363857648879767,0.07170816116483066)
--(0.3091592807122102,0.07017748614216043)--(0.30467998493562265,0.06865786003553928)
--(0.3002006891590352,0.06714936377155435)--(0.2957213933824477,0.0656520800815259)
--(0.2912420976058602,0.06416609356960143)--(0.28676280182927266,0.06269149078450184)
--(0.28228350605268515,0.06122836029517568)--(0.27780421027609764,0.05977679277063996)
--(0.27332491449951013,0.058336881064308725)--(0.2688456187229227,0.05690872030314095)
--(0.26436632294633516,0.05549240798196618)--(0.25988702716974765,0.05408804406338191)
--(0.25540773139316014,0.05269573108365302)--(0.25092843561657263,0.05131557426508493)
--(0.24644913983998512,0.04994768163538859)--(0.2419698440633976,0.04859216415460634)
--(0.23749054828681013,0.04724913585022642)--(0.23301125251022262,0.045918713961177304)
--(0.2285319567336351,0.04460101909146803)--(0.2240526609570476,0.04329617537431961)
--(0.21957336518046008,0.042004310647729165)--(0.21509406940387255,0.04072555664250949)
--(0.2106147736272851,0.03946004918396895)--(0.20613547785069758,0.03820792840853079)
--(0.20165618207411007,0.03696933899674451)--(0.19717688629752256,0.035744430424320926)
--(0.19269759052093502,0.03453335723302326)--(0.18821829474434756,0.03333627932348077)
--(0.18373899896776005,0.03215336227225981)--(0.17925970319117254,0.03098477767583987)
--(0.17478040741458506,0.02983070352450204)--(0.17030111163799752,0.028691324609560667)
--(0.16582181586141004,0.027566832967862585)--(0.16134252008482253,0.02645742836805515)
--(0.15686322430823502,0.025363318843811742)--(0.15238392853164753,0.024284721280009886)
--(0.14790463275506,0.023221862058822252)--(0.1434253369784725,0.0221749777738319)
--(0.138946041201885,0.02114431602166764)--(0.1344667454252975,0.020130136282327293)
--(0.12998744964871,0.019132710901390154)--(0.1255081538721225,0.01815232618980085)
--(0.12102885809553499,0.017189283659967112)--(0.11654956231894749,0.01624390142069605)
--(0.11207026654235998,0.015316515758221273)--(0.10759097076577247,0.014407482936513701)
--(0.10311167498918497,0.013517181257608283)--(0.09863237921259746,0.01264601343232816)
--(0.09415308343600996,0.01179440932425909)--(0.08967378765942245,0.010962829146125673)
--(0.08519449188283494,0.01015176720926047)--(0.08071519610624744,0.009361756355687817)
--(0.07623590032965993,0.008593373241477962)--(0.07175660455307244,0.007847244693951734)
--(0.06727730877648493,0.00712405544089494)--(0.06279801299989743,0.006424557617874988)
--(0.05831871722330992,0.0057495826171624865)--(0.053839421446722414,0.005100056076844675)
--(0.04936012567013491,0.0044770171693905856)--(0.04488082989354741,0.003881643919684239)
--(0.040401534116959896,0.0033152872187856555)--(0.03592223834037239,0.0027795178016977455)
--(0.03144294256378489,0.0022761933405199843)--(0.026963646787197385,0.0018075583218622703)
--(0.02248435101060988,0.0013764007788228095)--(0.018005055234022373,0.0009863159625249872)
--(0.01352575945743487,0.000642194917737317)--(0.009046463680847366,0.00035127127796605074)
--(0.00456716790425986,0.00012600704929554806)
--(0,0)--(1,0)
--(0.99,0.0099)--(0.98,0.0196)--(0.97,0.0291)--(0.96,0.0384)--(0.95,0.0475)
--(0.94,0.0564)--(0.93,0.0651)--(0.92,0.0736)--(0.91,0.0819)--(0.9,0.09)--(0.89,0.0979)
--(0.88,0.1056)--(0.87,0.1131)--(0.86,0.1204)--(0.85,0.1275)--(0.84,0.1344)--(0.83,0.1411)
--(0.82,0.1476)--(0.81,0.1539)--(0.8,0.16)--(0.79,0.1659)--(0.78,0.1716)--(0.77,0.1771)
--(0.76,0.1824)--(0.75,0.1875)--(0.74,0.1924)--(0.73,0.1971)--(0.72,0.2016)--(0.71,0.2059)
--(0.7,0.21)--(0.69,0.2139)--(0.68,0.2176)--(0.67,0.2211);
\begin{scriptsize}
\draw [fill=uuuuuu] (2/3,2/9) circle (0.2pt);
\draw [fill=uuuuuu] (6/7,6/49) circle (0.2pt);
\draw [fill=uuuuuu] (8/9,8/81) circle (0.2pt);
\draw [fill=uuuuuu] (2/3,0) circle (0.2pt);
\draw[color=uuuuuu] (2/3,0-0.07) node {$\frac{2}{3}$};
\draw [fill=uuuuuu] (6/7,0) circle (0.2pt);
\draw[color=uuuuuu] (6/7-0.01,0-0.07) node {$\frac{6}{7}$};
\draw [fill=uuuuuu] (8/9,0) circle (0.2pt);
\draw[color=uuuuuu] (8/9+0.01,0-0.07) node {$\frac{8}{9}$};
\draw [fill=uuuuuu] (1,0) circle (0.2pt);
\draw[color=uuuuuu] (1,0-0.07) node {$1$};
\draw [fill=uuuuuu] (0,0) circle (0.2pt);
\draw[color=uuuuuu] (0-0.05,0-0.05) node {$0$};
\draw [fill=uuuuuu] (0,2/9) circle (0.2pt);
\draw[color=uuuuuu] (0-0.08,2/9) node {$2/9$};
\draw [fill=uuuuuu] (0,6/49) circle (0.2pt);
\draw[color=uuuuuu] (0-0.08,6/49+0.02) node {$6/49$};
\draw [fill=uuuuuu] (0,8/81) circle (0.2pt);
\draw[color=uuuuuu] (0-0.08,8/81-0.02) node {$9/81$};
\draw [fill=uuuuuu] (0,1/2) circle (0.2pt);
\draw[color=uuuuuu] (0-0.08,1/2) node {$1/2$};
\draw[color=uuuuuu] (1+0.1,0-0.07) node {$x$};
\draw[color=uuuuuu] (0-0.08,1/2+0.1) node {$y$};
\end{scriptsize}
\end{tikzpicture}
\caption{$\Omega(\mathcal{T}_{3})$ is contained in the dark area above according to Corollary \ref{cancellative-3-all}.}
\end{figure}
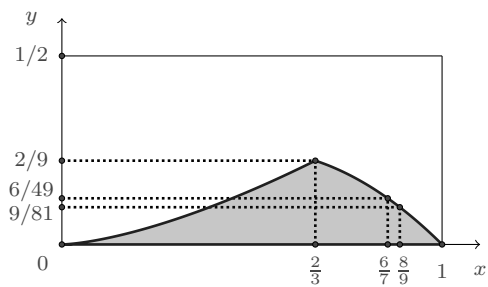

\begin{coro}\label{cancellative-3-all}
	We have $g(\mathcal{T}_{3},x) = x^{3/2}/\sqrt{6}$ for all $x\in [0,2/3]$,
	and $g(\mathcal{T}_{3},x) \le x(1-x)$ for all $x \in (2/3,1]$.
	Moreover, $g(\mathcal{T}_{3},(k-1)/k) = (k-1)/k^2$ for all integers $k \equiv 1$ or $3$ (mod $6$).
\end{coro}


\subsection{Hypergraphs without an expansion of a large clique}
In this section we consider the feasible region of hypergraphs without expansion of cliques.
These hypergraphs were introduced by the second author in \cite{MU06} as a way to generalize Tur\'{a}n's theorem to hypergraphs. Another reason for their importance is that they provide the first (and still the only) explicitly defined examples which yield an infinite family of numbers realizable as Tur\'{a}n densities for hypergraphs.

Let  $\mathcal{K}_{\ell+1}^{r}$ be the collection of all $r$-graphs $F$ with at most $\binom{{\ell}+1}{2}$
edges such that for some $({\ell}+1)$-set $S$, which will be called the core of $F$, every pair $\{u,v\}\subset S$ is covered by an edge in $F$.
Let the $r$-graph $H_{\ell+1}^{r}$ be obtained from the complete graph $K_{\ell}$ by adding $r-2$ new vertices into each edge.
The graph $H_{\ell+1}^{r}$ is called the expansion of $K_{\ell}$.
It is an easy observation that $H_{\ell+1}^{r} \in \mathcal{K}_{\ell+1}^{r}$.

It was shown by the second author $\cite{MU06}$ that $\textrm{ex}(n,\mathcal{K}_{\ell+1}^{r}) = t_{r}(n,\ell)$ and $T_{r}(n,\ell)$ is the unique
$\mathcal{K}_{\ell+1}^{r}$-free $r$-graph on $n$ vertices with exactly $t_{r}(n,\ell)$ edges.
In $\cite{PI08}$, Pikhurko improved the result in \cite{MU06} and proved that if $n$ is sufficiently large
then $\textrm{ex}(n,H_{\ell+1}^{r}) = t_{r}(n,\ell)$
and $T_{r}(n,\ell)$ is the unique $H_{\ell+1}^{r}$-free $r$-graph on $n$ vertices with exactly $t_{r}(n,\ell)$ edges.

In order to state our result, we need to extend the definition of shadows.
Let $\mathcal{H}$ be an $r$-graph and $S \subset V(\mathcal{H})$.
Then $\mathcal{H}[S]$ is the induced subgraph of $\mathcal{H}$ on $S$.
For $1 \le i \le r-1$ the $i$-th shadow of $\mathcal{H}$ is
\begin{align}
\partial_{i}\mathcal{H} = \left\{ A\in \binom{V(\mathcal{H})}{r-i}: \exists B \in \mathcal{H} \text{ such that } A\subset B \right\} \notag.
\end{align}
For $i \le 0$ we extend the definition of the $i$-th shadow $\partial_{i}\mathcal{H}$ as follows.
\begin{align}\label{non-normal-shadow}
\partial_{i}\mathcal{H} = \left\{ A\in \binom{V(\mathcal{H})}{r-i}: \text{$\mathcal{H}[A]$ is a complete $r$-graph} \right\}.
\end{align}
In particular, $\partial_{1}\mathcal{H} = \partial\mathcal{H}$ and $\partial_{0}\mathcal{H} = \mathcal{H}$.
By definition, $\partial_{i+1}\mathcal{H} = \partial \left( \partial_{i}\mathcal{H} \right)$ for all $0 \le i \le r-2$,
and $\partial \left( \partial_{i}\mathcal{H} \right) \subset \partial_{i+1}\mathcal{H}$ for all $i \le -1$.

Our first result here relates the sizes of different shadows of a $\mathcal{K}_{\ell+1}^{r}$-free $r$-graph $\mathcal{H}$.
This generalizes an important result of Fisher and Ryan~\cite{FR92} from graphs to hypergraphs.
\begin{thm}\label{fisher-ryan-hygp}
Let $\ell \ge r\ge 2$ and $\mathcal{H}$ be a $\mathcal{K}_{\ell+1}^{r}$-free $r$-graph.
Then
\[
\left(\frac{|\partial_{r-\ell}\mathcal{H}|}{\binom{\ell}{\ell}} \right)^{\frac{1}{\ell}} \le \cdots \le
\left(\frac{|\partial_{-1}\mathcal{H}|}{\binom{\ell}{r+1}} \right)^{\frac{1}{r+1}} \le
\left(\frac{|\mathcal{H}|}{\binom{\ell}{r}} \right)^{\frac{1}{r}} \le
\left(\frac{|\partial_{1}\mathcal{H}|}{\binom{\ell}{r-1}} \right)^{\frac{1}{r-1}} \le \cdots \le
\left(\frac{|\partial_{r-1}\mathcal{H}|}{\binom{\ell}{1}} \right)^{\frac{1}{1}}.
\]
\end{thm}

Using Theorem \ref{fisher-ryan-hygp} we are able to determine $g({\mathcal{K}_{\ell+1}^{r}})$ completely via the following result.
We will use $(\ell)_{r}$ to denote $\ell(\ell -1 ) \cdots (\ell -r+1)$.
\begin{coro}\label{feasible-region-K}
Let $\ell \ge r \ge 3$. Then ${\rm proj}\Omega(\mathcal{K}_{\ell+1}^{r}) = [0, (\ell)_{r-1}/\ell^{r-1}]$ and
\[
g({\mathcal{K}_{\ell+1}^{r}},x) = (\ell -r + 1) \left( \frac{x^{r}}{(\ell)_{r}} \right)^{\frac{1}{r-1}}
\]
for all $x\in [0, (\ell)_{r-1}/\ell^{r-1}]$.
\end{coro}

\begin{figure}[htbp]
\centering
\subfigure[$\ell = 3, r = 3$.]{
\begin{minipage}[t]{0.4\linewidth}
\centering
\begin{tikzpicture}[xscale=5,yscale=5]
\draw [->] (0,0)--(1.1,0);
\draw [->] (0,0)--(0,0.6);
\draw (0,0.5)--(1,0.5);
\draw (1,0)--(1,0.5);
\draw [line width=1pt,dash pattern=on 1pt off 1.2pt,domain=0:2/3] plot(\x,{2/9});
\draw [line width=0.5pt,color=sqsqsq,fill=sqsqsq,fill opacity=0.25]
(0.,0.)--(0.01,0.000408248)--(0.02,0.0011547)--(0.03,0.00212132)--(0.04,0.00326599)--(0.05,0.00456435)--
(0.06,0.006)--(0.07,0.00756086)--(0.08,0.0092376)--(0.09,0.0110227)--(0.1,0.0129099)--(0.11,0.0148941)--
(0.12,0.0169706)--(0.13,0.0191355)--(0.14,0.0213854)--(0.15,0.0237171)--(0.16,0.0261279)--(0.17,0.0286153)--
(0.18,0.0311769)--(0.19,0.0338107)--(0.2,0.0365148)--(0.21,0.0392874)--(0.22,0.0421268)--(0.23,0.0450315)--
(0.24,0.048)--(0.25,0.051031)--(0.26,0.0541233)--(0.27,0.0572756)--(0.28,0.0604869)--(0.29,0.063756)--(0.3,0.067082)--
(0.31,0.0704639)--(0.32,0.0739008)--(0.33,0.0773919)--(0.34,0.0809362)--(0.35,0.084533)--(0.36,0.0881816)--
(0.37,0.0918813)--(0.38,0.0956312)--(0.39,0.0994309)--(0.4,0.10328)--(0.41,0.107177)--(0.42,0.111122)--(0.43,0.115114)--
(0.44,0.119153)--(0.45,0.123238)--(0.46,0.127368)--(0.47,0.131544)--(0.48,0.135765)--(0.49,0.140029)--(0.5,0.144338)--
(0.51,0.148689)--(0.52,0.153084)--(0.53,0.157521)--(0.54,0.162)--(0.55,0.166521)--(0.56,0.171083)--(0.57,0.175686)--
(0.58,0.180329)--(0.59,0.185013)--(0.6,0.189737)--(0.61,0.1945)--(0.62,0.199302)--(0.63,0.204143)--(0.64,0.209023)--
(0.65,0.213941)--(0.66,0.218897)--(0.67,0.223891)--(0.67,0)--(0,0);
\begin{scriptsize}
\draw [fill=uuuuuu] (2/3,2/9) circle (0.2pt);
\draw [fill=uuuuuu] (2/3,0) circle (0.2pt);
\draw[color=uuuuuu] (2/3,0-0.07) node {$\frac{2}{3}$};
\draw [fill=uuuuuu] (1,0) circle (0.2pt);
\draw[color=uuuuuu] (1,0-0.07) node {$1$};
\draw [fill=uuuuuu] (0,0) circle (0.2pt);
\draw[color=uuuuuu] (0-0.05,0-0.05) node {$0$};
\draw [fill=uuuuuu] (0,2/9) circle (0.2pt);
\draw[color=uuuuuu] (0-0.08,2/9) node {$2/9$};
\draw [fill=uuuuuu] (0,1/2) circle (0.2pt);
\draw[color=uuuuuu] (0-0.08,1/2) node {$1/2$};
\end{scriptsize}
\end{tikzpicture}
\end{minipage}
}
\subfigure[$\ell = 4, r = 4$.]{
\begin{minipage}[t]{0.4\linewidth}
\centering
\begin{tikzpicture}[xscale=5,yscale=5]
\draw [->] (0,0)--(1.1,0);
\draw [->] (0,0)--(0,0.6);
\draw (0,0.5)--(1,0.5);
\draw (1,0)--(1,0.5);
\draw [line width=0.8pt,dash pattern=on 1pt off 1.2pt,domain=0:3/8] plot(\x,3/32);
\draw [line width=0.5pt,color=sqsqsq,fill=sqsqsq,fill opacity=0.25]
(0.,0.)--(0.01,0.000746901)--(0.02,0.00188207)--(0.03,0.00323165)--(0.04,0.00474252)--(0.05,0.00638591)--(0.06,0.00814325)--
(0.07,0.0100014)--(0.08,0.0119504)--(0.09,0.0139825)--(0.1,0.0160915)--(0.11,0.018272)--(0.12,0.0205197)--(0.13,0.0228308)--
(0.14,0.0252019)--(0.15,0.0276302)--(0.16,0.0301132)--(0.17,0.0326484)--(0.18,0.0352338)--(0.19,0.0378676)--(0.2,0.040548)--
(0.21,0.0432735)--(0.22,0.0460426)--(0.23,0.048854)--(0.24,0.0517064)--(0.25,0.0545988)--(0.26,0.05753)--(0.27,0.060499)--
(0.28,0.0635049)--(0.29,0.0665468)--(0.3,0.0696238)--(0.31,0.0727353)--(0.32,0.0758804)--(0.33,0.0790584)--(0.34,0.0822687)--
(0.35,0.0855107)--(0.36,0.0887836)--(0.37,0.092087)--(0.375,3/32)--(0.375,0)--(0,0);
\begin{scriptsize}
\draw [fill=uuuuuu] (3/8,3/32) circle (0.2pt);
\draw [fill=uuuuuu] (3/8,0) circle (0.2pt);
\draw[color=uuuuuu] (3/8,0-0.07) node {$\frac{3}{8}$};
\draw [fill=uuuuuu] (0,0) circle (0.2pt);
\draw[color=uuuuuu] (0-0.05,0-0.05) node {$0$};
\draw [fill=uuuuuu] (0,3/32) circle (0.2pt);
\draw[color=uuuuuu] (0-0.08,3/32) node {$3/32$};
\draw [fill=uuuuuu] (1,0) circle (0.2pt);
\draw[color=uuuuuu] (1,0-0.07) node {$1$};
\draw [fill=uuuuuu] (0,0) circle (0.2pt);
\draw[color=uuuuuu] (0-0.05,0-0.05) node {$0$};
\draw [fill=uuuuuu] (0,1/2) circle (0.2pt);
\draw[color=uuuuuu] (0-0.08,1/2) node {$1/2$};
\end{scriptsize}
\end{tikzpicture}
\end{minipage}
}
\centering
\caption{The region $\Omega(\mathcal{K}_{\ell+1}^{r})$ determined by Corollary \ref{feasible-region-K}.}
\end{figure}
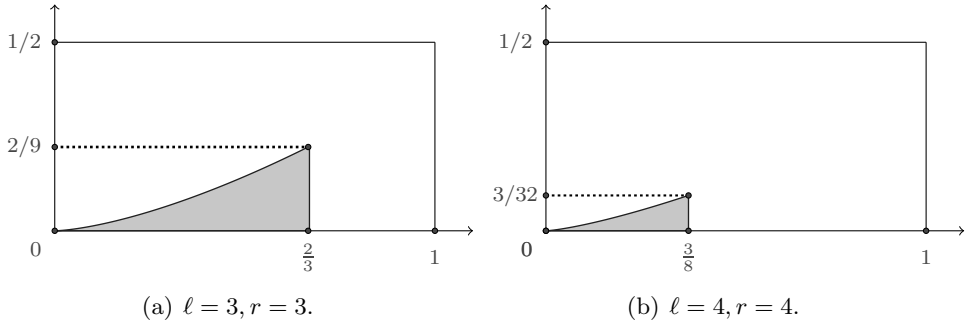

Determining $\Omega(H_{\ell+1}^{r})$ is much more difficult than $\Omega(\mathcal{K}_{\ell+1}^{r})$
because the shadow density of an $H_{\ell+1}^{r}$-free $r$-graph can be greater than $(\ell)_{r-1}/\ell^{r-1}$.
An $r$-graph $\mathcal{S}$ is called a star if all edges in $\mathcal{S}$ contain a fixed vertex, which is called the center of $\mathcal{S}$.
It is easy to see that a star does not contain $H_{\ell+1}^{r}$ as a subgraph, and the shadow density of a star
can be arbitrarily close to $1$.
Still, we are able to determine $g(H_{\ell+1}^{r},x)$ for all $x\in [0, (\ell)_{r-1}/\ell^{r-1}]$.
\begin{thm}\label{feasible-region-H}
Let $\ell \ge r \ge 3$. Then ${\rm proj}\Omega(H_{\ell+1}^{r}) = [0,1]$ and
\[
g(H_{\ell+1}^{r},x) = (\ell -r + 1) \left( \frac{x^{r}}{(\ell)_{r}} \right)^{\frac{1}{r-1}}
\]
for all $x\in [0, (\ell)_{r-1}/\ell^{r-1}]$.
\end{thm}

\begin{figure}[htbp]
\centering
\subfigure[$\ell = 3, r = 3$.]{
\begin{minipage}[t]{0.4\linewidth}
\centering
\begin{tikzpicture}[xscale=5,yscale=5]
\draw [->] (0,0)--(1.1,0);
\draw [->] (0,0)--(0,0.6);
\draw (0,0.5)--(1,0.5);
\draw (1,0)--(1,0.5);
\draw [line width=1pt,dash pattern=on 1pt off 1.2pt,domain=0:2/3] plot(\x,{2/9});
\draw [line width=1pt,dash pattern=on 1pt off 1.2pt] (2/3,0)--(2/3,2/9);
\draw [line width=0.5pt,color=sqsqsq,fill=sqsqsq,fill opacity=0.25]
(0.,0.)--(0.01,0.000408248)--(0.02,0.0011547)--(0.03,0.00212132)--(0.04,0.00326599)--(0.05,0.00456435)--
(0.06,0.006)--(0.07,0.00756086)--(0.08,0.0092376)--(0.09,0.0110227)--(0.1,0.0129099)--(0.11,0.0148941)--
(0.12,0.0169706)--(0.13,0.0191355)--(0.14,0.0213854)--(0.15,0.0237171)--(0.16,0.0261279)--(0.17,0.0286153)--
(0.18,0.0311769)--(0.19,0.0338107)--(0.2,0.0365148)--(0.21,0.0392874)--(0.22,0.0421268)--(0.23,0.0450315)--
(0.24,0.048)--(0.25,0.051031)--(0.26,0.0541233)--(0.27,0.0572756)--(0.28,0.0604869)--(0.29,0.063756)--(0.3,0.067082)--
(0.31,0.0704639)--(0.32,0.0739008)--(0.33,0.0773919)--(0.34,0.0809362)--(0.35,0.084533)--(0.36,0.0881816)--
(0.37,0.0918813)--(0.38,0.0956312)--(0.39,0.0994309)--(0.4,0.10328)--(0.41,0.107177)--(0.42,0.111122)--(0.43,0.115114)--
(0.44,0.119153)--(0.45,0.123238)--(0.46,0.127368)--(0.47,0.131544)--(0.48,0.135765)--(0.49,0.140029)--(0.5,0.144338)--
(0.51,0.148689)--(0.52,0.153084)--(0.53,0.157521)--(0.54,0.162)--(0.55,0.166521)--(0.56,0.171083)--(0.57,0.175686)--
(0.58,0.180329)--(0.59,0.185013)--(0.6,0.189737)--(0.61,0.1945)--(0.62,0.199302)--(0.63,0.204143)--(0.64,0.209023)--
(0.65,0.213941)--(0.66,0.218897)--(0.67,0.223891)--(1,0.223891)--(1,0)--(0,0);
\begin{scriptsize}
\draw [fill=uuuuuu] (2/3,2/9) circle (0.2pt);
\draw [fill=uuuuuu] (2/3,0) circle (0.2pt);
\draw[color=uuuuuu] (2/3,0-0.07) node {$\frac{2}{3}$};
\draw [fill=uuuuuu] (1,0) circle (0.2pt);
\draw[color=uuuuuu] (1,0-0.07) node {$1$};
\draw [fill=uuuuuu] (0,0) circle (0.2pt);
\draw[color=uuuuuu] (0-0.05,0-0.05) node {$0$};
\draw [fill=uuuuuu] (0,2/9) circle (0.2pt);
\draw[color=uuuuuu] (0-0.08,2/9) node {$2/9$};
\draw [fill=uuuuuu] (0,1/2) circle (0.2pt);
\draw[color=uuuuuu] (0-0.08,1/2) node {$1/2$};
\end{scriptsize}
\end{tikzpicture}
\end{minipage}
}
\subfigure[$\ell = 4, r = 4$.]{
\begin{minipage}[t]{0.4\linewidth}
\centering
\begin{tikzpicture}[xscale=5,yscale=5]
\draw [->] (0,0)--(1.1,0);
\draw [->] (0,0)--(0,0.6);
\draw (0,0.5)--(1,0.5);
\draw (1,0)--(1,0.5);
\draw [line width=0.8pt,dash pattern=on 1pt off 1.2pt,domain=0:3/8] plot(\x,3/32);
\draw [line width=0.8pt,dash pattern=on 1pt off 1.2pt] (3/8,0)--(3/8,3/32) ;
\draw [line width=0.5pt,color=sqsqsq,fill=sqsqsq,fill opacity=0.25]
(0.,0.)--(0.01,0.000746901)--(0.02,0.00188207)--(0.03,0.00323165)--(0.04,0.00474252)--(0.05,0.00638591)--(0.06,0.00814325)--
(0.07,0.0100014)--(0.08,0.0119504)--(0.09,0.0139825)--(0.1,0.0160915)--(0.11,0.018272)--(0.12,0.0205197)--(0.13,0.0228308)--
(0.14,0.0252019)--(0.15,0.0276302)--(0.16,0.0301132)--(0.17,0.0326484)--(0.18,0.0352338)--(0.19,0.0378676)--(0.2,0.040548)--
(0.21,0.0432735)--(0.22,0.0460426)--(0.23,0.048854)--(0.24,0.0517064)--(0.25,0.0545988)--(0.26,0.05753)--(0.27,0.060499)--
(0.28,0.0635049)--(0.29,0.0665468)--(0.3,0.0696238)--(0.31,0.0727353)--(0.32,0.0758804)--(0.33,0.0790584)--(0.34,0.0822687)--
(0.35,0.0855107)--(0.36,0.0887836)--(0.37,0.092087)--(0.375,3/32)--(1,3/32)--(1,0)--(0,0);
\begin{scriptsize}
\draw [fill=uuuuuu] (3/8,3/32) circle (0.2pt);
\draw [fill=uuuuuu] (3/8,0) circle (0.2pt);
\draw[color=uuuuuu] (3/8,0-0.07) node {$\frac{3}{8}$};
\draw [fill=uuuuuu] (0,0) circle (0.2pt);
\draw[color=uuuuuu] (0-0.05,0-0.05) node {$0$};
\draw [fill=uuuuuu] (0,3/32) circle (0.2pt);
\draw[color=uuuuuu] (0-0.08,3/32) node {$3/32$};
\draw [fill=uuuuuu] (1,0) circle (0.2pt);
\draw[color=uuuuuu] (1,0-0.07) node {$1$};
\draw [fill=uuuuuu] (0,0) circle (0.2pt);
\draw[color=uuuuuu] (0-0.05,0-0.05) node {$0$};
\draw [fill=uuuuuu] (0,1/2) circle (0.2pt);
\draw[color=uuuuuu] (0-0.08,1/2) node {$1/2$};
\end{scriptsize}
\end{tikzpicture}
\end{minipage}
}
\centering
\caption{The region $\Omega(H_{\ell+1}^{r})$ is contained in the dark areas according to Theorem \ref{feasible-region-H} and results in \cite{MU06} and \cite{PI13}.}
\end{figure}
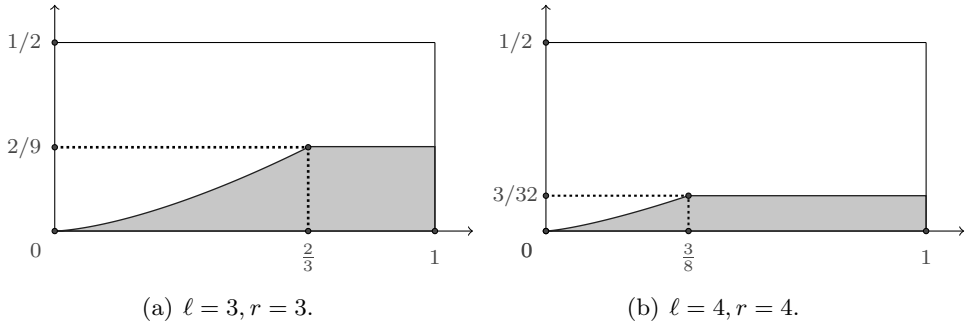

The remainder of this paper is organized as follows.
In Section 2 we will prove Propositions \ref{closed}, \ref{only-boundary-matters}, and \ref{universal-upper-bound},
and Theorem \ref{left-cont-and-diff}.
Section 3 will be devoted to prove Theorem \ref{example-discont}.
Then we will prove Theorems \ref{cancellative-r-left} and \ref{cancellative-3-right} in Section 4.
In Section 5 we will prove Theorem \ref{fisher-ryan-hygp}, Corollary \ref{feasible-region-K}, and Theorem \ref{feasible-region-H}.
In Section 6 we will include some remarks and open problems.
We will omit the floor and ceiling signs when they are not crucial in the proofs.

%
%
%
\section{General theory}
In this section we will prove several general results about the feasible region.
First let us present a simple but useful idea that will be used in our proofs.

\begin{fact}\label{unique-subset}
Let $r \ge 2$.
Suppose that $\mathcal{H}$ is an $r$-graph on $n$ vertices, and every edge in $\mathcal{H}$ contains an $(r-1)$-subset
that is not covered by any other edge in $\mathcal{H}$. Then $|\mathcal{H}| \le \binom{n}{r-1}$.
\end{fact}
Indeed, if every edge in $\mathcal{H}$ contains a unique $(r-1)$-subset, then
we can map every edge $E \in \mathcal{H}$ to an $(r-1)$-subset of $E$ that is not covered by any other edge in $\mathcal{H}$.
This map is an injection from $\mathcal{H}$ to $\binom{[n]}{r-1}$ and it implies the upper bound in Fact \ref{unique-subset}.
Actually, it was shown by Bollob\'{a}s \cite{BO65} that $|\mathcal{H}| \le \binom{n-1}{r-1}$.

\medskip

\noindent\textbf{Algorithm 1} Remove edges with the edge density threshold $d$.\\
\noindent\textbf{Input}: An $r$-graph $\mathcal{H}$ and the density threshold $d \in [0,1]$.\\
\noindent\textbf{Operation}: If $d(\mathcal{H}) \le d$ or $|\mathcal{H}| \le \binom{n}{r-1}$, then do nothing and let $\mathcal{H}$ be the output.
Otherwise, by Fact \ref{unique-subset}, there exists $E \in \mathcal{H}$ such that every $(r-1)$-subset of $E$ is covered by another edge in $\mathcal{H}$.
Remove $E$ from the edge set of $\mathcal{H}$, and let $\mathcal{H}$ denote the resulting $r$-graph.
Repeat this operation until $d - 1/\binom{n}{r} < d(\mathcal{H}) \le d$. \\
\noindent\textbf{Output}: Either the original $r$-graph $\mathcal{H}$ or a subgraph $\mathcal{H}' \subset \mathcal{H}$
with $d - 1/\binom{n}{r} < d(\mathcal{H}') \le d$,
and $|\partial \mathcal{H}'| = |\partial \mathcal{H}|$.

\medskip

Notice that the Operation above does not change $|\partial \mathcal{H}|$ since all $(r-1)$-subsets of the removed edge $E$
are covered by some edge in $\mathcal{H}$.
Therefore, the output $r$-graph $\mathcal{H}'$ satisfies $|\partial \mathcal{H}'| = |\partial \mathcal{H}|$.
On the other hand, since each step of the operation reduces $|\mathcal{H}|$ by exactly one,
$d(\mathcal{H})$ can be reduced to some real number $d'$ with $d - 1/\binom{n}{r} < d' \le d$.

\subsection{Basic properties}
In this section we will prove Propositions \ref{closed}, \ref{only-boundary-matters}, and \ref{universal-upper-bound},
and Theorem \ref{left-cont-and-diff}.
First we prove Proposition \ref{closed} and we need the following lemma.

\begin{lemma}\label{lemma-close-analysis}
Let $\left( (x_{m,k}, y_{m,k}) \right)_{k=1}^{\infty}$ be a sequence of pairs with $\lim_{k \to \infty}(x_{m,k}, y_{m,k}) = (x_{m},y_{m})$
for all $m\ge 1$.
Suppose that $\lim_{m \to \infty}(x_{m},y_{m}) = (x,y)$.
Then there exists $(x_{m,k_m},y_{m,k_m}) \in \left( (x_{m,k}, y_{m,k}) \right)_{k=1}^{\infty}$ for all $m \ge 1$
such that $\lim_{m \to \infty}(x_{m,k_m},y_{m,k_m}) = (x,y)$ and $\lim_{m \to \infty}k_m = \infty$.
\end{lemma}
\begin{proof}
Let $k_1 = 1$. For every $m \ge 2$ since $\lim_{k \to \infty}(x_{m,k}, y_{m,k}) = (x_{m},y_{m})$,
there exists $k_m$ with $k_m \ge k_{m-1} + 1$ such that for all $k \ge k_{m}$
\begin{equation}\label{lemma-bound-x-mk}
|x_{m,k} - x_{m}| < 1/m,
\end{equation}
and
\begin{equation}\label{lemma-bound-y-mk}
|y_{m,k} - y_m| < 1/m.
\end{equation}
Now fix $\epsilon > 0$.
Since $\lim_{m \to \infty}(x_{m},y_{m}) = (x,y)$, there exists $m(\epsilon)$ with $m(\epsilon) > 3/\epsilon$ such that
for all $m \ge m(\epsilon)$
\begin{equation}\label{lemma-bound-x-m}
|x_{m} - x| < \epsilon/3,
\end{equation}
and
\begin{equation}\label{lemma-bound-y-m}
|y_{m} - y| < \epsilon/3.
\end{equation}
Therefore, for all $m \ge m(\epsilon) > 3/\epsilon$
\[
|x_{m,k_m} - x| \le |x_{m,k_m} - x_{m}| + |x_{m} - x| \overset{(\ref{lemma-bound-x-mk}),(\ref{lemma-bound-x-m})}{<} 1/m + \epsilon/3 < \epsilon
\]
and
\[
|y_{m,k_m} - y| \le |y_{m,k_m} - y_{m}| + |y_{m} - y| \overset{(\ref{lemma-bound-y-mk}),(\ref{lemma-bound-y-m})}{<} 1/m + \epsilon/3 < \epsilon,
\]
which implies that $\lim_{m \to \infty}(x_{m,k_m},y_{m,k_m}) = (x,y)$.
Since $\left( k_m \right)_{m=1}^{\infty}$ is a strictly increasing sequence of integers, $\lim_{m \to \infty} k_m = \infty$,
and this completes the proof.
\end{proof}

Now we prove Proposition \ref{closed}.
\begin{proof}[Proof of Proposition \ref{closed}]
Let $\left( (x_{m}, y_{m}) \right)_{k=1}^{\infty}$ be a sequence with  $(x_{m}, y_{m}) \in \Omega(\mathcal{F})$ for all $m \ge 1$
and $\lim_{m \to \infty}(x_{m}, y_{m}) = (x_0,y_0)$.
We need to show that $(x_0,y_0) \in \Omega(\mathcal{F})$ as well.

By the definition of $\Omega(\mathcal{F})$, for every $m \ge 1$
there exists a good sequence $\left( \mathcal{H}_{m,k} \right)_{k = 1}^{\infty}$ of $\mathcal{F}$-free $r$-graphs
that realizes $(x_m,y_m)$.
Without loss of generality we may assume that $v(\mathcal{H}_{m,k+1}) \ge v(\mathcal{H}_{m,k}) + 1$ for all $k \ge 1$ and $m \ge 1$.
Let $x_{m,k} = d(\partial \mathcal{H}_{m,k})$ and $y_{m,k} = d(\mathcal{H}_{m,k})$ for all $m\ge 1$ and all $k\ge 1$.
Since $\lim_{k\to \infty} (x_{m,k}, y_{m,k}) = (x_m,y_m)$,
by Lemma 2.1, there exists $(x_{m,k_m},y_{m,k_m}) \in \left( (x_{m,k}, y_{m,k}) \right)_{k=1}^{\infty}$ for all $m \ge 1$
such that $\lim_{m \to \infty}(x_{m,k_m},y_{m,k_m}) = (x_0,y_0)$ and $\lim_{m \to \infty}k_m = \infty$.
This implies that $\left( \mathcal{H}_{m,k_m} \right)_{k = 1}^{\infty}$ is
a good sequence of $\mathcal{F}$-free $r$-graphs that realizes $(x_0,y_0)$.
So, $(x_0,y_0) \in \Omega(\mathcal{F})$ and this completes the proof.
\end{proof}

Next we prove Proposition \ref{only-boundary-matters}. Its proof uses Algorithm 1.
\begin{proof}[Proof of Proposition \ref{only-boundary-matters}]
Since $(x_0,y_0) \in \Omega(\mathcal{F})$, there exists a good sequence of $\mathcal{F}$-free $r$-graphs
$\left( \mathcal{H}_{k} \right)_{k=1}^{\infty}$ for which
$\lim_{k \to \infty} d(\partial \mathcal{H}_k) = x_0$ and $\lim_{k \to \infty} d(\mathcal{H}_k) = y_0$.
Now fix $y \in [0,y_0)$.
For every $k\ge 1$ apply Algorithm 1 to $\mathcal{H}_{k}$ with edge density threshold $y$
and let $\mathcal{H}_{k}'$ denote the $r$-graph that Algorithm 1 outputs.
We claim that $\left( \mathcal{H}_{k}' \right)_{k=1}^{\infty}$ is a good sequence of $\mathcal{F}$-free $r$-graphs that realizes $(x_0, y)$.
Indeed, choose $\epsilon = (y_0-y)/2 > 0$,
by the assumption that $\lim_{k\to \infty}d(\mathcal{H}_k) = y_0$,
there exists $k_0$ such that $d(\mathcal{H}_k) \in (y_0 - \epsilon, y_0 + \epsilon)$ for all $k \ge k_0$.
Therefore, by Algorithm $1$, $y - 1/\binom{v(\mathcal{H}_k)}{r} < d(\mathcal{H}'_{k}) \le y$ for all $k\ge k_0$,
and hence $\lim_{k \to \infty} d(\mathcal{H}'_{k}) = y$.
On the other hand, since $|\partial \mathcal{H}'_{k}| = |\partial \mathcal{H}_{k}|$ for all $k\ge 1$, $\lim_{k \to \infty} d(\partial \mathcal{H}'_{k}) = x$.
Therefore, $\left( \mathcal{H}_{k}' \right)_{k=1}^{\infty}$ is a good sequence of $\mathcal{F}$-free $r$-graphs that realizes $(x_0, y)$,
and hence $(x_0, y) \in \Omega(\mathcal{F})$.
\end{proof}

Recall that $\textrm{ex}(n, \mathcal{F}_{1}) \le \textrm{ex}(n, \mathcal{F}_{2})$ whenever $\mathcal{F}_{2} \subset \mathcal{F}_{1}$.
By the definition of $g(\mathcal{F})$, a similar inequality also holds for $g(\mathcal{F})$.
\begin{obs}\label{subset-inequality}
Let $r \ge 3$.
Suppose that $\mathcal{F}_1$ and $\mathcal{F}_2$ are two families of $r$-graphs with $\mathcal{F}_1\subset \mathcal{F}_2$.
Then $\Omega(\mathcal{F}_2) \subset \Omega(\mathcal{F}_1)$.
In particular, $g({\mathcal{F}_2},x) \le g({\mathcal{F}_1},x)$ for all $x \in {\rm proj}\Omega(\mathcal{F}_2)$.
\end{obs}

Now we are ready to prove Proposition \ref{universal-upper-bound}.
\begin{proof}[Proof of Proposition \ref{universal-upper-bound}]
By Observation \ref{subset-inequality}, it suffices to show that ${\rm proj}\Omega(\emptyset) = [0,1]$
and $g(\emptyset, x) = x^{r/(r-1)}$ for all $x \in [0,1]$.
The first part is easy, since the complete $r$-graph on $n$ vertices has shadow density $1$,
and it follows from Observation \ref{projection-x-interval} that ${\rm proj}\Omega(\emptyset) = [0,1]$.

Now we consider the second part.
First we show that $g(\emptyset, x) \le x^{r/(r-1)}$ for all $x \in [0,1]$.
Let $\left(\mathcal{H}_{k}\right)_{k = 1}^{\infty}$ be a good sequence of $r$-graph that realizes $(x,y)$.
For every $k \ge 1$ let $\alpha_{k}$ denote the real number that satisfies $|\partial \mathcal{H}_{k}| = \binom{\alpha_k v(\mathcal{H}_{k})}{r-1}$.
By the Kruskal-Katona theorem, $|\mathcal{H}_{k}| \le \binom{\alpha_k v(\mathcal{H}_{k})}{r}$ for all $k \ge 1$.
By assumption and $\lim_{k \to \infty}v(\mathcal{H}_k) = \infty$,
\[
x = \lim_{k \to \infty} \frac{|\partial \mathcal{H}_{k}|} {\binom{v(\mathcal{H}_{k})}{r-1}} =
\lim_{k \to \infty} \frac{\binom{\alpha_k v(\mathcal{H}_{k})}{r-1}} {\binom{v(\mathcal{H}_{k})}{r-1}}=
\lim_{k \to \infty} (\alpha_{k})^{r-1},
\]
which implies that $\lim_{k \to \infty} \alpha_{k} = x^{1/(r-1)}$.
Therefore, by assumption,
\[
y = \lim_{k \to \infty} \frac{|\mathcal{H}_{k}|} {\binom{v(\mathcal{H}_{k})}{r-1}} \le
\lim_{k \to \infty} \frac{\binom{\alpha_k v(\mathcal{H}_{k})}{r}} {\binom{v(\mathcal{H}_{k})}{r}}=
\lim_{k \to \infty} (\alpha_{k})^{r} = x^{\frac{r}{r-1}},
\]
and this proves that $g(\emptyset, x) \le x^{r/(r-1)}$ for all $x \in [0,1]$.

Next we show that $g(\emptyset, x) \ge x^{r/(r-1)}$ for all $x \in [0,1]$.
Choose an arbitray $x \in [0,1]$ and let $\alpha = x^{1/(r-1)}$.
Let $\mathcal{H}_{n}(\alpha)$ denote the vertex disjoint union of a complete $r$-graph on $\alpha n$ vertices and a set of $(1-\alpha)n$ isolated vertices.
Then we claim that $\left(\mathcal{H}_k(\alpha)\right)_{k = 1}^{\infty}$ is a good sequence of $r$-graphs that realizes $(x,x^{r/(r-1)})$.
Indeed,
\[
\lim_{k \to \infty} \frac{|\partial \mathcal{H}_{k}(\alpha)|} {\binom{n}{r-1}} =
\lim_{k \to \infty} \frac{\binom{\alpha n}{r-1}} {\binom{n}{r-1}} =
\alpha^{r-1} = x,
\]
and
\[
\lim_{k \to \infty} \frac{|\mathcal{H}_{k}(\alpha)|} {\binom{n}{r}} =
\lim_{k \to \infty} \frac{\binom{\alpha n}{r}} {\binom{n}{r}} =
\alpha^{r} = x^{\frac{r}{r-1}},
\]
and it follows from the definition that $g(\emptyset, x) \ge x^{r/(r-1)}$ for all $x \in [0,1]$.
\end{proof}

\subsection{Continuity and differentiability}
In this section we will prove Theorem~\ref{left-cont-and-diff} and some other related corollaries.
We will use the following theorem in our proofs.

\begin{thm}[see Section 3 of Chapter 3, \cite{SS05l}]\label{monotone-function}
Let $f: \mathbb{R} \to \mathbb{R}$ be a monotone function.
Then $f$ has at most countably many discontinuities of the first kind and no discontinuity of the second kind.
Moreover, $f$ is almost everywhere differentiable.
\end{thm}

The following lemma is the main tool in our proofs.
\begin{lemma}\label{main-inequality-monotone}
Let $r\ge 3$ and $\mathcal{F}$ be a family of $r$-graphs.
Then
\[
\left( g({\mathcal{F}},x+h) \right)^{\frac{r-1}{r}} \le
\left( g({\mathcal{F}},x) \right)^{\frac{r-1}{r}} + \frac{\left( g({\mathcal{F}},x)  \right)^{\frac{r-1}{r}}}{x} h
\]
for all $x\in {\rm proj}\Omega(\mathcal{F})\setminus \{0\}$ and all $h\ge 0$ with $x+h \in {\rm proj}\Omega(\mathcal{F})$.
\end{lemma}
\begin{proof}
Suppose that $x+h \in {\rm proj}\Omega(\mathcal{F})$.
Choose
\[
\alpha = \left( \frac{x+h}{x} \right)^{\frac{1}{r-1}} -1.
\]
Let $\left( \mathcal{H}_k\right)_{k=1}^{\infty}$ be a good sequence of  $\mathcal{F}$-free $r$-graphs that realizes $(x+h,g({\mathcal{F}},x+h))$.
For every $k\ge 1$ let $n_k = v(\mathcal{H}_{k})$ and
let $\mathcal{H}'_k$ be obtained from $\mathcal{H}_k$ by adding a set of $\alpha n_{k}$ isolated vertices and let $n'_{k} = (1+\alpha) n_{k}$.
Then,
\[
\lim_{k\to \infty}\frac{|\partial\mathcal{H}'_k|}{\binom{n'_{k}}{r-1}} =
\lim_{k\to \infty}\frac{|\partial\mathcal{H}_k|}{\binom{(1+\alpha) n_{k}}{r-1}} = \frac{x+h}{(1+\alpha)^{r-1}} = x,
\]
and
\[
\lim_{k\to \infty}\frac{|\mathcal{H}'_k|}{\binom{n'_{k}}{r}} =
\lim_{k\to \infty}\frac{|\mathcal{H}_k|}{ \binom{(1+\alpha) n_{k}}{r}}=
\frac{g({\mathcal{F}},x+h)}{(1+\alpha)^r} = \left( \frac{x}{x+h}\right)^{\frac{r}{r-1}}  g({\mathcal{F}},x+h).
\]
Therefore, $\left( \mathcal{H}'_k\right)_{k=1}^{\infty}$ a good sequence of  $\mathcal{F}$-free $r$-graphs that realizes
\[
\left( x,\left( \frac{x}{x+h}\right)^{\frac{r}{r-1}}  g({\mathcal{F}},x+h) \right).
\]
Consequently,
\begin{align}\label{general-continuious-key-lemma}
g({\mathcal{F}},x) \ge \left( \frac{x}{x+h}\right)^{\frac{r}{r-1}}  g({\mathcal{F}},x+h),
\end{align}
which gives
\[
\left( g({\mathcal{F}},x+h) \right)^{\frac{r-1}{r}} \le
\left( g({\mathcal{F}},x) \right)^{\frac{r-1}{r}} + \frac{\left( g({\mathcal{F}},x)  \right)^{\frac{r-1}{r}}}{x} h.
\]
\end{proof}

\begin{coro}\label{down-left-cont}
Let $r \ge 3$ and $\mathcal{F}$ be a family of $r$-graphs.
Then for any $x \in {\rm proj}\Omega(\mathcal{F})\setminus\{0\}$ and any $\delta > 0$, there exists $\epsilon > 0$
such that $g(\mathcal{F},x') > g(\mathcal{F},x) - \delta$ for all $x' \in (x-\epsilon, x)$.
\end{coro}
\begin{proof}
We may assume that $\delta < 1$.
Choose $\epsilon  = \delta x /3$ and let $x' \in (x-\epsilon, x)$.
Then $(\ref{general-continuious-key-lemma})$ gives
\[
\begin{split}
g({\mathcal{F}},x') & \ge \left( \frac{x'}{x}\right)^{\frac{r}{r-1}}  g({\mathcal{F}},x) \\
& =\left( 1 - \frac{x-x'}{x}\right)^{\frac{r}{r-1}}  g({\mathcal{F}},x) \\
& \ge \left( 1 - \frac{2 \epsilon}{x}\right)  g({\mathcal{F}},x)
=  g({\mathcal{F}},x) - \frac{2 g({\mathcal{F}},x)\epsilon}{x}
> g({\mathcal{F}},x) - \delta,
\end{split}
\]
where the second inequality follows from the fact that $(1-x)^a \ge 1 - ax$ for all $x \in [0,1]$ and all $a \ge 1$.
\end{proof}

Proposition \ref{closed} together with Corollary \ref{down-left-cont} will show that $g(\mathcal{F})$ does not contain removable discontinuities.
\begin{coro}\label{no-reomvable-discont}
Let $r \ge 3$ and $\mathcal{F}$ be a family of $r$-graphs.
Then $g(\mathcal{F})$ does not contain removable discontinuities.
\end{coro}
\begin{proof}
Suppose that $x_0 \in {\rm proj}\Omega(\mathcal{F})$ is a removable discontinuity of $g(\mathcal{F})$.
Then $x_0 > 0$ and $\lim_{x \to x_0^{-}} g(\mathcal{F}, x) = \lim_{x \to x_0^{+}} g(\mathcal{F}, x) \neq g(\mathcal{F}, x_0)$.
Let $y_0 = \lim_{x \to x_0^{-}} g(\mathcal{F}, x)$.
By Proposition \ref{closed}, $(x_0, y_0) \in \Omega(\mathcal{F})$, and by the definition of $g(\mathcal{F})$, $g(\mathcal{F}, x_0) > y_0$.
Letting $\delta = (g(\mathcal{F}, x_0) - y_0)/2$ in Corollary \ref{down-left-cont}, we obtain
\[
y_0 = \lim_{x \to x_0^{-}} g(\mathcal{F}, x) > g(\mathcal{F}, x_0) - \delta = \frac{g(\mathcal{F}, x_0) + y_0}{2} > y_0,
\]
a contradiction.
\end{proof}

Now we are ready to prove Theorem \ref{left-cont-and-diff}.
\begin{proof}[Proof of Theorem \ref{left-cont-and-diff}]
First we show that $g(\mathcal{F})$ is almost everywhere differentiable.
Let $f(x) =  \left(g(\mathcal{F},x)\right)^{\frac{r-1}{r}} - x$.
It follows from Lemma \ref{main-inequality-monotone} and Theorem \ref{universal-upper-bound} that
\[
\begin{split}
\left(g(\mathcal{F},x+h)\right)^{\frac{r-1}{r}} & \le
\left(g(\mathcal{F},x)\right)^{\frac{r-1}{r}} + \frac{\left(g(\mathcal{F},x)\right)^{\frac{r-1}{r}}}{x}  h \\
& \le \left(g(\mathcal{F},x)\right)^{\frac{r-1}{r}} + \frac{\left(x^{\frac{r}{r-1}}\right)^{\frac{r-1}{r}}}{x}  h \\
& = \left(g(\mathcal{F},x)\right)^{\frac{r-1}{r}} + h,
\end{split}
\]
which implies that $f$ is decreasing on ${\rm proj}\Omega(\mathcal{F})$.
By Theorem \ref{monotone-function}, $f$ is almost everywhere differentiable, and so is $g(\mathcal{F})$.

Next, we show that $g(\mathcal{F})$ has at most countably many jump discontinuities.
By Theorem \ref{monotone-function}, $f$ has at most countably many discontinuities of the first kind,
and so does $g(\mathcal{F})$ since $g(\mathcal{F},x) = \left( f(x)+x \right)^{r/(r-1)}$ for all $x \in {\rm proj}\Omega(\mathcal{F})$.
Corollary \ref{down-left-cont} shows that $g(\mathcal{F})$ does not have a removable discontinuity,
therefore, $g(\mathcal{F})$ has at most countably many jump discontinuities.

Finally, we show that $g(\mathcal{F})$ is left-continuous.
Let $x_0 \in {\rm proj}\Omega(\mathcal{F})$ be a discontinuity of $g(\mathcal{F})$.
By the previous result, $x_0$ can only be a jump discontinuity.
Let $y^{-}_0 = \lim_{x \to x^{-}_0} g(\mathcal{F},x)$ and $y^{+}_0 = \lim_{x \to x^{+}_0} g(\mathcal{F},x)$.
By Proposition \ref{closed}, $(x_0, y_{0}^{-})\in \Omega(\mathcal{F})$ and $(x_0, y_{0}^{+}) \in \Omega(\mathcal{F})$.
So, it suffices to show that $y_{0}^{-} > y_{0}^{+}$.
Indeed, suppose that $y_0^{+}> y_{0}^{-}$.
Then, by the definition of $g(\mathcal{F})$ we would have $g(\mathcal{F},x_0) = y_{0}^{+}$.
Letting $\delta = (y_{0}^{+} - y_{0}^{-})/2$ in Corollary \ref{down-left-cont}, we obtain
\[
y_0^{-} = \lim_{x \to x^{-}_0} g(\mathcal{F},x) > g(\mathcal{F},x_0) - \delta = \frac{y_{0}^{-} + y_{0}^{+}}{2} > y_0^{-},
\]
a contradiction, and this completes the proof.
\end{proof}

The proof of Theorem \ref{left-cont-and-diff} also gives the following corollary.
\begin{coro}\label{increasing-implies-continuious}
Let $r \ge 3$ and $\mathcal{F}$ be a family of $r$-graphs.
Suppose that $x_0 \in {\rm proj}\Omega(\mathcal{F})$ is a discontinuity of $g(\mathcal{F})$.
Then both $\lim_{x \to x^{-}_0} g(\mathcal{F},x)$ and $\lim_{x \to x^{+}_0} g(\mathcal{F},x)$ exist and
$\lim_{x \to x^{-}_0} g(\mathcal{F},x) > \lim_{x \to x^{+}_0} g(\mathcal{F},x)$.
In particular, if $g(\mathcal{F})$ is increasing on $[c_1,c_2]$ for some $c_2 > c_1 \ge 0$,
then $g(\mathcal{F})$ is continuous on $[c_1,c_2]$.
\end{coro}

\section{A point of discontinuity}
In this section we will prove Theorem \ref{example-discont} by defining a family $\mathcal{D}$ of $3$-graphs,
and showing that $g({\mathcal{D}})$ is discontinuous at $x=2/3$.

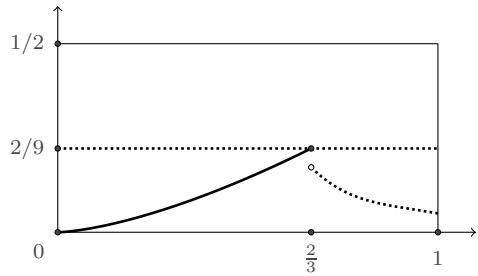
\begin{figure}[htbp]
\centering
\begin{tikzpicture}[xscale=5,yscale=5]
\draw [->] (0,0)--(1.1,0);
\draw [->] (0,0)--(0,0.6);
\draw (0,0.5)--(1,0.5);
\draw (1,0)--(1,0.5);
\draw [line width=1pt,dash pattern=on 1pt off 1.2pt,domain=0:1] plot(\x,{2/9});
\draw[line width=1pt]
(2/3,2/9)--(0.6630236470626231,0.2204032031146704)
--(0.6585443512860356,0.21817345951646344)--(0.654065055509448,0.21595128623168966)
--(0.6495857597328606,0.21373670913866868)--(0.645106463956273,0.21152975438293659)
--(0.6406271681796856,0.20933044838187664)--(0.636147872403098,0.2071388178294629)
--(0.6316685766265105,0.20495488970112097)--(0.6271892808499231,0.2027786912587088)
--(0.6227099850733355,0.2006102500556217)--(0.6182306892967481,0.19844959394202646)
--(0.6137513935201605,0.19629675107022665)--(0.609272097743573,0.19415174990016584)
--(0.6047928019669855,0.19201461920507196)--(0.600313506190398,0.1898853880772468)
--(0.5958342104138105,0.1877640859340081)--(0.591354914637223,0.1856507425237863)
--(0.5868756188606356,0.18354538793238254)--(0.582396323084048,0.18144805258939525)
--(0.5779170273074605,0.17935876727481742)--(0.573437731530873,0.1772775631258134)
--(0.5689584357542855,0.17520447164368083)--(0.5644791399776979,0.17313952470100336)
--(0.5599998442011105,0.17108275454900254)--(0.5555205484245229,0.16903419382509416)
--(0.5510412526479355,0.1669938755606583)--(0.5465619568713479,0.164961833189029)
--(0.5420826610947604,0.16293810055371402)--(0.537603365318173,0.16092271191685042)
--(0.5331240695415854,0.15891570196790752)--(0.528644773764998,0.1569171058326439)
--(0.5241654779884104,0.1549269590823307)--(0.519686182211823,0.1529452977432499)
--(0.5152068864352354,0.15097215830647748)--(0.5107275906586479,0.14900757773796589)
--(0.5062482948820604,0.14705159348893285)--(0.5017689991054729,0.14510424350657278)
--(0.4972897033288854,0.1431655662451021)--(0.4928104075522979,0.14123560067715077)
--(0.48833111177571037,0.13931438630551737)--(0.48385181599912286,0.1374019631752988)
--(0.4793725202225354,0.13549837188641337)--(0.4748932244459478,0.1336036536065323)
--(0.4704139286693603,0.1317178500844364)--(0.46593463289277287,0.12984100366381837)
--(0.46145533711618536,0.1279731572975465)--(0.45697604133959785,0.12611435456241368)
--(0.45249674556301034,0.12426463967439019)--(0.44801744978642283,0.12242405750440374)
--(0.4435381540098353,0.12059265359467064)--(0.43905885823324786,0.11877047417560284)
--(0.43457956245666024,0.11695756618331694)--(0.43010026668007284,0.11515397727777338)
--(0.4256209709034853,0.1133597558615753)--(0.4211416751268978,0.11157495109945828)
--(0.4166623793503103,0.1097996129385029)--(0.4121830835737228,0.10803379212910673)
--(0.4077037877971353,0.10627754024675189)--(0.4032244920205478,0.10453090971460692)
--(0.39874519624396026,0.10279395382700543)--(0.3942659004673728,0.10106672677384512)
--(0.3897866046907853,0.09934928366595529)--(0.38530730891419773,0.09764168056148004)
--(0.3808280131376103,0.09594397449333411)--(0.3763487173610227,0.09425622349778398)
--(0.37186942158443526,0.0925784866442176)--(0.36739012580784774,0.0909108240661645)
--(0.36291083003126023,0.0892532969936367)--(0.3584315342546727,0.08760596778686161)
--(0.35395223847808527,0.08596889997148711)--(0.34947294270149776,0.0843421582753402)
--(0.3449936469249102,0.08272580866683046)--(0.34051435114832274,0.08111991839509318)
--(0.33603505537173517,0.07952455603197399)--(0.3315557595951477,0.07793979151596783)
--(0.3270764638185602,0.07636569619822708)--(0.3225971680419727,0.0748023428907691)
--(0.3181178722653852,0.07324980591701827)--(0.31363857648879767,0.07170816116483066)
--(0.3091592807122102,0.07017748614216043)--(0.30467998493562265,0.06865786003553928)
--(0.3002006891590352,0.06714936377155435)--(0.2957213933824477,0.0656520800815259)
--(0.2912420976058602,0.06416609356960143)--(0.28676280182927266,0.06269149078450184)
--(0.28228350605268515,0.06122836029517568)--(0.27780421027609764,0.05977679277063996)
--(0.27332491449951013,0.058336881064308725)--(0.2688456187229227,0.05690872030314095)
--(0.26436632294633516,0.05549240798196618)--(0.25988702716974765,0.05408804406338191)
--(0.25540773139316014,0.05269573108365302)--(0.25092843561657263,0.05131557426508493)
--(0.24644913983998512,0.04994768163538859)--(0.2419698440633976,0.04859216415460634)
--(0.23749054828681013,0.04724913585022642)--(0.23301125251022262,0.045918713961177304)
--(0.2285319567336351,0.04460101909146803)--(0.2240526609570476,0.04329617537431961)
--(0.21957336518046008,0.042004310647729165)--(0.21509406940387255,0.04072555664250949)
--(0.2106147736272851,0.03946004918396895)--(0.20613547785069758,0.03820792840853079)
--(0.20165618207411007,0.03696933899674451)--(0.19717688629752256,0.035744430424320926)
--(0.19269759052093502,0.03453335723302326)--(0.18821829474434756,0.03333627932348077)
--(0.18373899896776005,0.03215336227225981)--(0.17925970319117254,0.03098477767583987)
--(0.17478040741458506,0.02983070352450204)--(0.17030111163799752,0.028691324609560667)
--(0.16582181586141004,0.027566832967862585)--(0.16134252008482253,0.02645742836805515)
--(0.15686322430823502,0.025363318843811742)--(0.15238392853164753,0.024284721280009886)
--(0.14790463275506,0.023221862058822252)--(0.1434253369784725,0.0221749777738319)
--(0.138946041201885,0.02114431602166764)--(0.1344667454252975,0.020130136282327293)
--(0.12998744964871,0.019132710901390154)--(0.1255081538721225,0.01815232618980085)
--(0.12102885809553499,0.017189283659967112)--(0.11654956231894749,0.01624390142069605)
--(0.11207026654235998,0.015316515758221273)--(0.10759097076577247,0.014407482936513701)
--(0.10311167498918497,0.013517181257608283)--(0.09863237921259746,0.01264601343232816)
--(0.09415308343600996,0.01179440932425909)--(0.08967378765942245,0.010962829146125673)
--(0.08519449188283494,0.01015176720926047)--(0.08071519610624744,0.009361756355687817)
--(0.07623590032965993,0.008593373241477962)--(0.07175660455307244,0.007847244693951734)
--(0.06727730877648493,0.00712405544089494)--(0.06279801299989743,0.006424557617874988)
--(0.05831871722330992,0.0057495826171624865)--(0.053839421446722414,0.005100056076844675)
--(0.04936012567013491,0.0044770171693905856)--(0.04488082989354741,0.003881643919684239)
--(0.040401534116959896,0.0033152872187856555)--(0.03592223834037239,0.0027795178016977455)
--(0.03144294256378489,0.0022761933405199843)--(0.026963646787197385,0.0018075583218622703)
--(0.02248435101060988,0.0013764007788228095)--(0.018005055234022373,0.0009863159625249872)
--(0.01352575945743487,0.000642194917737317)--(0.009046463680847366,0.00035127127796605074)
--(0.00456716790425986,0.00012600704929554806);
\draw[line width=1pt,dash pattern=on 1pt off 1.2pt] (2/3,2/9-0.05) to [out=315,in=170] (1,0.05);
\begin{scriptsize}
\draw [fill=uuuuuu] (2/3,2/9) circle (0.2pt);
\draw [fill=white] (2/3,2/9-0.05) circle (0.2pt);
\draw [fill=uuuuuu] (2/3,0) circle (0.2pt);
\draw[color=uuuuuu] (2/3,0-0.07) node {$\frac{2}{3}$};
\draw [fill=uuuuuu] (1,0) circle (0.2pt);
\draw[color=uuuuuu] (1,0-0.07) node {$1$};
\draw [fill=uuuuuu] (0,0) circle (0.2pt);
\draw[color=uuuuuu] (0-0.05,0-0.05) node {$0$};
\draw [fill=uuuuuu] (0,2/9) circle (0.2pt);
\draw[color=uuuuuu] (0-0.08,2/9) node {$2/9$};
\draw [fill=uuuuuu] (0,1/2) circle (0.2pt);
\draw[color=uuuuuu] (0-0.08,1/2) node {$1/2$};
\end{scriptsize}
\end{tikzpicture}
\caption{The function $g(\mathcal{D})$ is discontinuous at $x =2/3$.}
\end{figure}

First we define a $3$-graph $\mathcal{S}_n$ on $[n]$ as follows.
Fix $u\in [n]$, let
\[
\mathcal{S}_n = \left\{uvw: vw\in \binom{[n]\setminus\{u\}}{2} \right\},
\]
and note that $\mathcal{S}_n$ is a star with $|\mathcal{S}_n| = \binom{n-1}{2}$.

\begin{dfn}\label{family-D}
Let $\mathcal{D}$ be the collection of all $3$-graphs $F \in \mathcal{K}_{4}^{3}$ such that $F\not\subset\mathcal{S}_n$ for all $n\ge 4$.
\end{dfn}
Note that $\mathcal{D} \neq \emptyset$ as $H_{4}^{3} \in \mathcal{D}$.
Since $\mathcal{S}_n$ is $\mathcal{D}$-free and $\lim_{n\to \infty}|\partial \mathcal{S}_n|/\binom{n}{2} = 1$,
by Observation \ref{projection-x-interval}, ${\rm proj}\Omega(\mathcal{D}) = [0,1]$.

Since $T_{3}(n,3)$ is $\mathcal{K}_{4}^{3}$-free, $\textrm{ex}(n,\mathcal{D}) \ge t_{3}(n,3)$.
On the other hand, $\textrm{ex}(n,\mathcal{D}) \le \textrm{ex}(n,H_{4}^{3})$, which, by \cite{PI13},
is at most $t_{3}(n,3)$ when $n$ is sufficiently large.
Therefore, we obtain the following result.
\begin{thm}\label{Turan-number-D}
Let $n$ be sufficiently large.
Then $\textrm{ex}(n,\mathcal{D}) = t_{3}(n,3)$
and $T_{3}(n,3)$ is the unique $\mathcal{D}$-free $3$-graph with $n$ vertices and $t_{3}(n,3)$ edges.
\end{thm}

Theorem \ref{Turan-number-D} implies that $g(\mathcal{D},x) \le 2/9$ for all $x \in [0,1]$ and equality holds for $x = 2/3$.
Therefore, in order to prove Theorem \ref{example-discont} it suffices to prove the following result.
\begin{thm}\label{nonstable-D}
There exists an absolute constant $\delta_0 > 0$ such that the following is true
for all $\epsilon \in (0, 10^{-8})$ and sufficiently large $n$.
Suppose that $\mathcal{H}$ is a $\mathcal{D}$-free $3$-graph on $n$ vertices with $|\partial \mathcal{H}| = (1/3+\epsilon) n^2$.
Then $|\mathcal{H}| \le (1/27 - \delta_0)n^3$.
\end{thm}

The proof of Theorem \ref{nonstable-D} uses a stability result for $\mathcal{D}$-free $3$-graphs,
which can be easily obtained from a stability theorem for ${H}_{\ell + 1}^{r}$-free $r$-graphs proved by Pikhurko \cite{PI13}.
\begin{thm}[Stability]\label{stability-D}
For every $\xi > 0$ there exists $\delta > 0$ (we may assume that $\delta \le \xi$) and $n_{0} = n_0(\xi)$ such that the following holds for all $n \ge n_0$.
Suppose that $\mathcal{H}$ is a $\mathcal{D}$-free $3$-graph on $n$ vertices with $|\mathcal{H}| \ge (1/27 - \delta)n^3$.
Then $V(\mathcal{H})$ has a partition $V_{1} \cup V_{2} \cup V_{3}$ such that all but at most $\xi n^3$ edges in $\mathcal{H}$
have exactly one vertex in each $V_{i}$.
\end{thm}

Now we are ready to prove Theorem \ref{nonstable-D}.
\begin{proof}[Proof of Theorem \ref{nonstable-D}]
We prove Theorem \ref{nonstable-D} by contradiction.
Suppose that for all constant $\delta > 0$ and all integers $n_0$ there exists
$\epsilon = \epsilon(\delta) \in (0, 10^{-8})$ such that
there exists a $3$-graph $\mathcal{H}$ on $n > n_0$ vertices for some $n$ with $|\partial \mathcal{H}| = (1/3 + \epsilon)n^2$
and $|\mathcal{H}| > (1/27 - \delta) n^3$.

Choose $\xi > 0$ to be sufficiently small, and let $\delta >0$ and $n_0 = n_{0}(\xi)$ be given by Theorem \ref{stability-D}
and note that we may assume that $\delta \le \xi$.
By assumption, there exists $\epsilon \in (0, 10^{-8})$ and
a $\mathcal{D}$-free $3$-graphs $\mathcal{H}$ on $n > n_0$ vertices with $|\partial \mathcal{H}| = (1/3 + \epsilon)n^2$
and $|\mathcal{H}| > (1/27 - \delta) n^3$.
Apply Theorem \ref{stability-D} to $\mathcal{H}$.
We obtain a partition $V(\mathcal{H}) = V_1 \cup V_2 \cup V_3$ such that
all but at most $\xi n^3$ edges in $\mathcal{H}$ have exactly one vertex in each $V_{i}$.
Let $\mathcal{H}'$ denote the induced $3$-partite $3$-graph of $\mathcal{H}$ on $V_1 \cup V_2 \cup V_3$,
that is,
\[
\mathcal{H}' = \left\{ E \in \mathcal{H}: |E \cap V_i| = 1 \text{ for all } i \in [3]  \right\}.
\]
Note that
\begin{equation}\label{size-H'}
|\mathcal{H}'| > \frac{n^3}{27} - \delta n^3 - \xi n^3.
\end{equation}

\begin{claim}\label{size-Vi}
$\left| |V_{i}| - \frac{n}{3} \right| < 4 (\delta + \xi)^{1/2}n$ for all $i \in [3]$.
\end{claim}
\begin{proof}
Fix $1 \le i \le 3$ and let $\alpha = |V_{i}|$.
Then $|\mathcal{H}'| \le \alpha(n-\alpha)^2/4$ and $(9)$ gives
\[
\frac{\alpha(n-\alpha)^2}{4} > \frac{n^3}{27} - \delta n^3 -\xi n^3,
\]
which implies $n/3 - 4 (\delta + \xi)^{1/2}n < \alpha <n/3 + 4 (\delta + \xi)^{1/2}n$.
\end{proof}

Let $G = \partial \mathcal{H}$ and $G' = \partial \mathcal{H}'$.
Note that $\mathcal{H}' \subset \mathcal{H}$, $G' \subset G$, and $G'$ is $3$-partite.
Let $K$ be a $3$-partite subgraph of $G$ with the maximum number of edges among all $3$-partite subgraphs of $G$,
and let $X_1,X_2,X_3$ denote the three parts of $K$.

\begin{claim}\label{size-G'-and-K}
$|K| \ge |G'| > \frac{n^2}{3} - 5(\delta + \xi)^{1/2}n^2$.
\end{claim}
\begin{proof}
Counting the number of edges in $\mathcal{H}'$ we obtain
\[
|G'| \left( \frac{n}{3}+ 4(\delta + \xi)^{1/2} n \right) > 3 |\mathcal{H}'| \overset{(\ref{size-H'})}{>} \frac{n^3}{9} - 3\left( \delta+\xi \right)n^3,
\]
which implies $|G'| > n^2/3 - 5(\delta + \xi)^{1/2} n^2$.
Since $G'$ is also a $3$-partite subgraph of $G$,
by the maximality of $K$, we obtain $|K| \ge |G'|$.
\end{proof}

\begin{claim}\label{size-of-V'i}
$\left| |X_i| - \frac{n}{3} \right| < 4 \left(\delta + \xi\right)^{1/4}n$ for all $i \in [3]$.
\end{claim}
\begin{proof}
Fix $i \in [3]$ and let $\alpha' = |X_i|$.
By Claim \ref{size-G'-and-K},
\[
\alpha' (n-\alpha') + \frac{(n-\alpha')^2}{4} \ge |K| \ge |G'| > \frac{n^2}{3} - 5(\delta + \xi)^{1/2} n^2,
\]
which implies $n/3 - 4 \left(\delta + \xi \right)^{1/4}n < \alpha' <n/3 + 4 \left(\delta + \xi \right)^{1/4}n$.
\end{proof}

For $uv \in K$ the degree of $uv$ in $\mathcal{H}$ is $d(uv) := |\{E \in \mathcal{H}: \{u,v\} \subset E\}|$.
Our next claim shows that most edges in $K$ have a large degree.
\begin{claim}\label{few-small-degree-edge}
The number of edges in $K$ that have degree at most $10$ in $\mathcal{H}$ is at most $n^2/40000$.
\end{claim}
\begin{proof}
Suppose not. Then the assumption that $|G| = |\partial\mathcal{H}| = (1/3 + \epsilon)n^2$ together with
Claims \ref{size-G'-and-K} and \ref{size-of-V'i} imply
\[
\begin{split}
|\mathcal{H}|
&\overset{\text{Claim \ref{size-of-V'i}}}{\le} \frac{1}{3}  \left( |K| - \frac{n^2}{40000} \right) \left( \frac{n}{3}+ 4(\delta + \xi)^{1/4}n \right)
 + \frac{10n^2}{40000}  + \left(|G| - |K| \right)n \\
&\overset{\text{Claim \ref{size-G'-and-K}}}{\le} \frac{1}{3} \left( \frac{n^2}{3}-\frac{n^2}{40000} \right)\left( \frac{n}{3}+ 4(\delta + \xi)^{1/4}n \right)
 + \frac{n^2}{4000} + \epsilon n^3 + 5(\delta + \xi)^{1/4} n^3\\
& < \frac{n^3}{27} - \frac{n^3}{500000},
\end{split}
\]
which contradicts the assumption that $|\mathcal{H}| > (1/27 - \delta)n^3$.
Here we used the fact that $\delta, \xi$ are sufficiently small, $n$ is sufficiently large, and $\epsilon < 10^{-8}$.
\end{proof}

The next claim shows that if $G$ has a large complete $4$-partite subgraph, then it contains many edges
that have degree at most $10$ in $\mathcal{H}$.
This is the only place where we use the definition of $\mathcal{D}$.
\begin{claim}\label{4-partite-subgp}
Let $v_1v_2\in G$ and $U_1, U_2 \subset V(\mathcal{H})\setminus\{v_1,v_2\}$.
Let
\[
L = \left\{ \{u_1,u_2\} : u_1\in U_1, u_2\in U_{2} \text{ and } d(u_1u_2)\ge 10 \right\}.
\]
Suppose that $v_{1}$ and $v_2$ are adjacent to all vertices in $U_1 \cup U_2$.
Then $L$ is an intersecting family, and hence $|L| < n$.
\end{claim}
\begin{proof}
Let $u_1u_2 \in L$ and
\[
\mathcal{E}_{v_1v_2} = \left\{E \in \mathcal{H}: \{v_1,v_2\} \subset E \right\}.
\]
We claim that every set $E \in \mathcal{E}_{v_1v_2}$ satisfies $E \cap \{u_1,u_2\} \neq \emptyset$.
Indeed, suppose that there exists $E_{v_1v_2} \in \mathcal{E}_{v_1v_2}$ with $E_{v_1v_2} \cap \{u_1,u_2\} = \emptyset$.
Since $d(u_1u_2)\ge 10$, there exists $E_{u_1u_2} \in \mathcal{H}$ such that $\{u_1,u_2\} \in E_{u_1u_2}$ and
$E_{u_1u_2} \cap E_{v_1v_2} = \emptyset$.
Let $E_{v_1u_1}, E_{v_1u_2}, E_{v_2u_1}$, and $E_{v_2u_2}$ be edges in $\mathcal{H}$ that cover $v_1u_1,v_1u_2, v_2u_1,v_2u_2$, respectively,
and let $F_1$ denote the $3$-graph with edge set
\[
\{E_{v_1v_2}, E_{v_1u_1}, E_{v_1u_2}, E_{v_2u_1}, E_{v_2u_2}, E_{u_1u_2}\}.
\]
Note that $F_1 \subset \mathcal{H}$ and $F_1 \in \mathcal{K}_{4}^{3}$.
However, since $E_{u_1u_2} \cap E_{v_1v_2} = \emptyset$, $F_1 \not\subset \mathcal{S}_n$ for any $n$, and hence $F_1 \in \mathcal{D}$,
which is a contradiction.
Therefore, every set $E \in \mathcal{E}_{v_1v_2}$ satisfies $E \cap \{u_1,u_2\} \neq \emptyset$.

Suppose that $L$ contains another edge $w_1w_2$ that is disjoint from $u_1u_2$.
Then, the same argument as above implies that every set $E \in \mathcal{E}_{v_1v_2}$ satisfies $E \cap \{w_1,w_2\} \neq \emptyset$.
Therefore, every set $E \in \mathcal{E}_{v_1v_2}$ satisfies $E \cap \{u_1,u_2\} \neq \emptyset$ and $E \cap \{w_1,w_2\} \neq \emptyset$,
which is impossible since $E$ is a $3$-set.
Therefore, $L$ is intersecting and it follows from the Erd\H{o}s-Ko-Rado theorem \cite{EKR} that $|L| < n$.
\end{proof}

Our goal in the rest of the proof is to find $v_1v_2\in G$ and $U_1, U_2 \subset V(\mathcal{H})\setminus\{v_1,v_2\}$ with $|U_1||U_2|$ large,
such that $v_{1}$ and $v_2$ are adjacent to all vertices in $U_1 \cup U_2$.
Then, by Claim \ref{4-partite-subgp}, many edges in the induced subgraph of $K$ on $U_1 \cup U_2$ would have degree at most $10$,
which contradicts Claim \ref{few-small-degree-edge}.

Let
\[
B = \left\{ uv \in G : \{u,v\} \subset X_{i} \text{ for some } i \in [3]\right\},
\]
and
\[
M = \left\{ \{u,v\} \in \binom{V(\mathcal{H})}{2}\setminus K : u\in X_i, v\in X_j \text{ for some } i,j\in [3] \text{ and } i\neq j\right\}.
\]
Sets in $B$ are called bad edges of $K$ and sets in $M$ are called missing edges of $K$.
For $v \in V(\mathcal{H})$ let $d_{M}(v)$ denote the number of missing edges that contain $v$.
By Claim \ref{size-G'-and-K},
\begin{align}\label{munber-missing-edges}
|M| \le 5(\delta + \xi)^{1/2}n^2.
\end{align}
On the other hand, the assumption $|G| = n^2/3 + \epsilon n^2$ implies
\begin{equation}\label{bad-and-missing-pairs}
|B| \ge |M| + \epsilon n^2.
\end{equation}
Let $B_{i}$ be the collection of bad edges in $G$ that are completely contained in $X_i$ for $i\in[3]$.
Without loss of generality, we may assume that $|B_1| \ge |B|/3$.
Let $\Delta$ denote the maximum degree of $B_{1}$.

\medskip

\noindent\textbf{Case 1}: $\Delta < n/100$.\\
Then there are at least $|B_1|/(2\Delta) > 15 |B| /n$ pairwise disjoint edges in $B_1$.
Fix $uv \in B_1$.
Let $U_i(uv) = N_{K}(u) \cap N_{K}(v) \cap X_2$ for $i \in \{2,3\}$ and
let $K_{uv}$ denote the induced subgraph of $K$ on $U_2(uv) \cup U_3(uv)$.
By Claim \ref{4-partite-subgp}, all but at most $n$ edges in $K_{uv}$ have degree at most $10$ in $\mathcal{H}$.
It follows that
\[
|U_2(uv)| |U_3(uv)| \overset{\text{Claim \ref{few-small-degree-edge}}}{\le}
\frac{n^2}{40000} + n + |M| \overset{(\ref{munber-missing-edges})}{\le}
\frac{n^2}{40000} + n + 5(\delta + \xi)^{1/2} n^2 < \frac{n^2}{30000}.
\]
Therefore, by Claim \ref{size-of-V'i},
\[
|N_{K}(u) \cap N_{K}(v)| < \frac{n}{3} + 4(\delta + \xi)^{1/4} n + \frac{n^2/30000}{{n}/{3} + 4(\delta + \xi)^{1/4} n}
<  \frac{n}{3} + 4(\delta + \xi)^{1/4} n  + \frac{n}{10000},
\]
and it follows from Inclusion-Exclusion and Claim \ref{size-of-V'i} that
\begin{align}\label{discon-upper-bound-degree-sum}
d_{K}(u) + d_{K}(v) &= |N_{K}(u) \cup N_{K}(v)| + |N_{K}(u) \cap N_{K}(v)| \notag\\
& \le 2\left( \frac{n}{3} + 4(\delta + \xi)^{1/4} n \right) +   \frac{n}{3} + 4(\delta + \xi)^{1/4} n  + \frac{n}{10000} \notag \\
& < \frac{101 n }{100}.
\end{align}
Note that
\[
d_{K}(u) + d_{M}(u) + d_{K}(v) + d_{M}(v) = 2 \left( |X_2|+ |X_3| \right),
\]
which implies
\begin{align}
|M|  \ge \sum_{uv \in B_1}\left( d_{M}(u) + d_{M}(v) \right)
& \ge \frac{15|B|}{n} \left(2 \left( |X_2|+ |X_3| \right) - d_{K}(u) - d_{K}(v) \right) \notag \\
&\overset{\text{Claim \ref{size-of-V'i} and (\ref{discon-upper-bound-degree-sum})}}{>}
\frac{15|B|}{n} \left( \frac{4n}{3} -\frac{102n}{100} \right) >
4 |B| \overset{(\ref{bad-and-missing-pairs})}{>} |M|, \notag
\end{align}
a contradiction.

\medskip

\noindent\textbf{Case 2}: $\Delta \ge n/100$.\\
Then choose a vertex $v_{1} \in X_{1}$ with degree $\Delta$.
Let $N_{i} = N_{K}(v_{1}) \cap X_{i}$ for $1 \le  i \le 3$.
The maximality of $K$ implies that $|N_{2}| \ge \Delta$ and $|N_{3}| \ge \Delta$,
since otherwise we could move $v_{1}$ into $V_2$ or $V_3$ to get a larger $3$-partite subgraph of $G$.
Choose $v_2 \in N_{1}$ and let $U_i(v_1v_2) =  N_{K}(v_2) \cap N_i$ for $i \in \{2,3\}$.
Similar to Case 1, we have $|U_2(v_1v_2)||U_3(v_1v_2)| \le {n^2}/{30000}$.
Therefore, $v_2$ is not adjacent (in $K$) to at least $n/200$ vertices in $N_{2} \cup N_{3}$,
which implies
\[
|M| \ge \sum_{u \in N_{1}} d_{M}(u) \ge
\frac{n}{100} \times \frac{n}{200} = \frac{n^2}{20000} > 5(\delta + \xi)^{1/2} n^2 \overset{(\ref{munber-missing-edges})}{\ge} |M|,
\]
a contradiction.
\end{proof}


\section{Cancellative hypergraphs}
In this section we will prove Theorems \ref{cancellative-r-left} and \ref{cancellative-3-right}.
First let us present some useful lemmas.

Let $\mathcal{H}$ be an $r$-graph.
The link of $v$ in $\mathcal{H}$ is
\[
L_{\mathcal{H}}(v) = \left\{ A\in \binom{V(\mathcal{H})}{r-1} : \{v\}\cup A \in \mathcal{H} \right\}.
\]
Let $d_{\mathcal{H}}(v) = |L_{\mathcal{H}}(v)|$.
For a subset $S\subset V(\mathcal{H})$ let $\sigma_{\mathcal{H}}(S) = \sum_{v\in S}d_{\mathcal{H}}(v)$.
When it is clear from context we will omit the subscript $\mathcal{H}$.

\begin{lemma}\label{cancel-induction-lemma}
Let $r\ge 3$ and let $\mathcal{H}$ be a cancellative $r$-graph.
Then, for any $v\in V(\mathcal{H})$ the link $L(v)$ is a cancellative $(r-1)$-graph.
\end{lemma}
\begin{proof}
Suppose that there exist $A,B,C\in L(v)$ such that $A\triangle B\subset C$.
Let $A'=A\cup \{v\}$, $B'=B\cup \{v\}$ and $C'=C\cup \{v\}$, and note that $A', B', C' \in \mathcal{H}$.
Then, $A'\triangle B'\subset C'$, which is a contradiction.
\end{proof}

\begin{lemma}\label{cancel-disjoint-lemma}
Let $r\ge 3$ and let $\mathcal{H}$ be a cancellative $r$-graph.
Suppose that $\{u,v\} \subset V(\mathcal{H})$ is covered by an edge in $\mathcal{H}$.
Then $L(u)\cap L(v)=\emptyset$.
\end{lemma}
\begin{proof}
Suppose that there exists $E\in L(u)\cap L(v)$.
Let $A=E\cup\{u\}$ and $B=E\cup \{v\}$, and note that $A, B\in \mathcal{H}$.
Then $A\triangle B=\{u,v\}$, which by assumption is covered by another edge $C$ in $\mathcal{H}$,
a contradiction.
\end{proof}

Lemma \ref{cancel-disjoint-lemma} gives the following corollary.
\begin{coro}\label{cancel-clique-corol}
Let $r\ge 3$ and $\mathcal{H}$ be a cancellative $r$-graph.
Let $S \subset V(\mathcal{H})$ and suppose that $(\partial_{r-2}\mathcal{H})[S]$ is a complete graph.
Then,
\[
\sum_{v\in S}d(v) \le |\partial\mathcal{H}|.
\]
\end{coro}
\begin{proof}
Suppose that $S = \{v_1, \ldots, v_{s}\}$.
Lemma \ref{cancel-disjoint-lemma} implies that the links $L(v_1), \ldots, L(v_{s})$ are pairwise edge disjoint.
Since $\bigcup_{i=1}^{s} L(v_{i}) \subset \partial \mathcal{H}$, we have $\sum_{v\in S}d(v) \le |\partial\mathcal{H}|$.
\end{proof}

\subsection{Proof of Theorem \ref{cancellative-r-left}}
In this section we will prove Theorem \ref{cancellative-r-left}, but
instead of proving it directly we will prove the following stronger statement.

\begin{thm}\label{fisher-ryan-cancel}
Let $r \ge 2$ and let $\mathcal{H}$ be a cancellative $r$-graph.
Then
\[
|\mathcal{H}| \le \left(\frac{|\partial \mathcal{H}|}{r} \right)^{\frac{r}{r-1}}.
\]
\end{thm}

First we show that Theorem \ref{fisher-ryan-cancel} implies Theorem \ref{cancellative-r-left}.
\begin{proof}[Proof of Theorem \ref{cancellative-r-left}]
Let us consider the lower bound first.
Let $\alpha \in [0,1]$ and let $\mathcal{H}_{n}(\alpha)$ be the vertex disjoint union of $T_{r}(\alpha n, r)$ and a set of $(1-\alpha)n$ isolated vertices.
It is clear that $\mathcal{T}_{r} \not\subset \mathcal{H}_{n}(\alpha)$.
Let
\[
x = \lim_{n \to \infty} \frac{|\partial \mathcal{H}_{n}(\alpha)|}{\binom{n}{r-1}} =
\lim_{n \to \infty} \frac{r \left( \alpha n/r \right)^{r-1}}{\binom{n}{r-1}} = \frac{\alpha^{r-1} (r-1)!}{r^{r-2}},
\]
and
\[
y = \lim_{n\to \infty} \frac{|\mathcal{H}_{n}(\alpha)|}{\binom{n}{r}} =
\lim_{n\to \infty} \frac{\left(\alpha n/r \right)^{r}}{\binom{n}{r}} = \frac{\alpha^r (r-1)!}{r^{r-1}}.
\]
Then, $y = \left( {x^{r}}/{r!} \right)^{{1}/{(r-1)}}$.
Letting $\alpha$ vary from $0$ to $1$, we obtain
$g(\mathcal{T}_{r},x) \ge \left( {x^{r}}/{r!} \right)^{{1}/{(r-1)}}$ for all $x \in [0,(r-1)! / r^{r-2}]$.

Next we prove the upper bound.
Suppose that $\left( \mathcal{H}_{k} \right)_{k=1}^{\infty}$ is a good sequence of cancellative $r$-graphs that realizes $(x,y)$.
Let $x_{k} = (r-1)!|\partial \mathcal{H}_{k}| / \left( v(\mathcal{H}_{k}) \right)^{r-1}$ and
$y_{k} = r!|\mathcal{H}_{k}| / \left( v(\mathcal{H}_{k}) \right)^{r}$ for all $k \ge 1$.
Then Theorem \ref{fisher-ryan-cancel} gives
\[
\frac{y_{k}\left( v(\mathcal{H}_{k}) \right)^{r}}{r!}  \le \left( \frac{x_{k}\left( v(\mathcal{H}_{k}) \right)^{r-1}}{r(r-1)!} \right)^{\frac{r}{r-1}},
\]
which implies
\[
y_k \le \left( \frac{(x_{k})^r}{r!} \right)^{\frac{1}{r-1}}.
\]
Letting $k \to \infty$, we obtain $y \le  \left( {x^r}/{r!} \right)^{{1}/{r-1}}$, and this completes the proof.
\end{proof}

Now we prove Theorem \ref{fisher-ryan-cancel}.
We will use the following fact.
\begin{fact}\label{inequalities}
Let $X$ be a collection of non-negative real numbers and $a \in [0,1]$.
Then
\begin{equation}\label{Jensen-1}
\sum_{x \in X} x^a \le |X| \left( \frac{\sum_{x \in X} x}{|X|} \right)^a = |X|^{1-a} \left( \sum_{x \in X} x \right)^a,
\end{equation}
and
\begin{equation}\label{Jensen-2}
\left( \sum_{x \in X} x \right)^2 \le |X| \sum_{x \in X} x^2.
\end{equation}
\end{fact}

\begin{proof}[Proof of Theorem \ref{fisher-ryan-cancel}]
We proceed by induction on $r$. When $r=2$, this is just Mantel's theorem, so we may assume that $r\ge 3$.

By Lemma \ref{cancel-induction-lemma}, $L(v)$ is a cancellative $(r-1)$-graph for all $v\in V(\mathcal{H})$.
Therefore, by the induction hypothesis,
\begin{align}\label{cancel-induction-inequality}
d(v) \le \left( \frac{|\partial L(v)|}{r-1} \right)^{\frac{r-1}{r-2}}.
\end{align}
It follows that
\begin{align}\label{size-H-inequality}
|\mathcal{H}|  =\frac{1}{r}\sum_{v\in V(\mathcal{H})}d(v)
& =\frac{1}{r}\sum_{v\in V(\mathcal{H})}\left(d(v)\right)^{\frac{1}{r-1}} \left(d(v)\right)^{\frac{r-2}{r-1}} \notag\\
&\overset{(\ref{cancel-induction-inequality})}{\le} \frac{1}{r(r-1)}\sum_{v\in V(\mathcal{H})}\left(d(v)\right)^{\frac{1}{r-1}}  |\partial L(v)|.
\end{align}
Notice that
\begin{align}\label{double-count-inequality}
\sum_{v\in V(\mathcal{H})} \left(d(v)\right)^{\frac{1}{r-1}} |\partial L(v)|
& =  \sum_{v \in V(\mathcal{H})} \sum_{\substack{S \in \partial \mathcal{H} \notag\\ v \in S}} \left(d(v)\right)^{\frac{1}{r-1}} \\
& = \sum_{S\in \partial\mathcal{H}}\sum_{v\in S} \left(d(v)\right)^{\frac{1}{r-1}} \notag \\
&\overset{(\ref{Jensen-1})}{\le}
\left((r-1)|\partial\mathcal{H}| \right)^{\frac{r-2}{r-1}}\left(\sum_{S\in \partial\mathcal{H}}\sum_{v\in S}d(v) \right)^{\frac{1}{r-1}} \notag \\
& =\left((r-1)|\partial\mathcal{H}| \right)^{\frac{r-2}{r-1}}\left(\sum_{S\in \partial\mathcal{H}}\sigma(S) \right)^{\frac{1}{r-1}}.
\end{align}
Define $\hat{\sigma}=\max\left\{\sigma(H): H\in \mathcal{H} \right\}$
and suppose that $E \in \mathcal{H}$ satisfies $\sum_{v\in E}d(v) = \hat{\sigma}$.
Then,
\begin{align}\label{sigma-S-inequality}
\sum_{S\in \partial\mathcal{H}}\sigma(S)
&=  \sum_{S \in \bigcup_{v\in E}L(v)}  \sigma(S) + \sum_{S\in \partial\mathcal{H}\setminus\bigcup_{v\in E}L(v)}  \sigma(S) \notag\\
&\overset{\text{Lemma \ref{cancel-disjoint-lemma}}}{=}\sum_{v \in E}\sum_{S\in L(v)}\sigma(S)
+\sum_{S\in \partial\mathcal{H}\setminus\bigcup_{v\in E}L(v)}\sigma(S) \notag\\
& \le \sum_{v \in E}d(v) \left(\hat{\sigma}-d(v)\right)
+ \left( |\partial\mathcal{H}|-\hat{\sigma} \right)\hat{\sigma} \notag\\
&\overset{(\ref{Jensen-2})}{\le} \left( \sum_{v \in E}d(v)\right)\left(\hat{\sigma}- \frac{\sum_{v \in E}d(v)}{r}\right)
+ \left(|\partial\mathcal{H}|-\hat{\sigma} \right)\hat{\sigma} \notag\\
& = \hat{\sigma} \left(\hat{\sigma}- \frac{\hat{\sigma}}{r}\right) + \left(|\partial\mathcal{H}|-\hat{\sigma} \right)\hat{\sigma} \notag\\
& =\left(|\partial\mathcal{H}|-\frac{\hat{\sigma}}{r} \right)\hat{\sigma}.
\end{align}
Note that Corollary \ref{cancel-clique-corol} gives $\hat{\sigma} \le |\partial\mathcal{H}|$.
On the other hand,
since $\left(|\partial\mathcal{H}|- {\hat{\sigma}}/{r} \right)\hat{\sigma}$ is increasing in $\hat{\sigma}$
when $\hat{\sigma} \le r|\partial \mathcal{H}|/2$, it follows from $(\ref{sigma-S-inequality})$ and $r \ge 3$ that
\begin{align}\label{final-sigma-inequality-cancell}
\sum_{S\in \partial\mathcal{H}}\sigma(S) \le \left(|\partial\mathcal{H}|-\frac{\hat{\sigma}}{r} \right)\hat{\sigma}\le \frac{r-1}{r}|\partial\mathcal{H}|^2.
\end{align}
Plugging $(\ref{double-count-inequality})$ and $(\ref{final-sigma-inequality-cancell})$ into $(\ref{size-H-inequality})$, we obtain
\[
\begin{split}
|\mathcal{H}| & \le
\frac{1}{r(r-1)} \left( (r-1)|\partial \mathcal{H}| \right)^{\frac{r-2}{r-1}} \left( \frac{r-1}{r} |\partial \mathcal{H}|^2 \right)^{\frac{1}{r-1}}
= \left(\frac{|\partial \mathcal{H}|}{r} \right)^{\frac{r}{r-1}},
\end{split}
\]
and this completes the proof.
\end{proof}

\subsection{Proof of Theorem \ref{cancellative-3-right}}
In this section we will prove Theorem \ref{cancellative-3-right}.
As before, we will prove a stronger statement which implies Theorem \ref{cancellative-3-right}.

\begin{thm}\label{cancel-shadow-inequ-right}
Suppose that $\mathcal{H}$ is a cancellative $3$-graph on $n$ vertices.
Then
\[
|\mathcal{H}| \le \frac{\left(n^2 - 2|\partial\mathcal{H}|\right)|\partial\mathcal{H}|}{3n} + 3n^2.
\]
\end{thm}

First we show that Theorem \ref{cancel-shadow-inequ-right} implies Theorem \ref{cancellative-3-right}.
\begin{proof}[Proof of Theorem \ref{cancellative-3-right}]
Let us consider the lower bound first.

A $k$-vertex Steiner triple system ($STS$ for short) is a $3$-graph on $k$ vertices such that every pair of vertices is covered by exactly one edge.
It is known that a $k$-vertex $STS$ exists iff $k \equiv 1$ or $3$ (mod 6) (e.g. see \cite{RW03}).
Let $STS(k)$ denote the family of all Steiner triple systems on $k$ vertices.
Let $\mathcal{S}(n,k)$ denote the collection of all $3$-graphs on $n$ vertices that can be obtained from a $3$-graph $H\in STS(k)$
by blowing up every vertex in $H$ into a set of size either $\lf n/k \rf$ or $\lc n/k \rc$.
It is easy to see that every $3$-graph in $\mathcal{S}(n,k)$ is cancellative.

Fix an integer $k$ with $k \equiv 1$ or $3$ (mod 6).
Let $\mathcal{H}_{n} \in \mathcal{S}(n,k)$
and in order to keep the calculations simple let us assume that $k$ divides $n$.
Then
\[
\lim_{n \to \infty} \frac{|\partial \mathcal{H}_{n}|}{\binom{n}{2}} = \frac{(k-1)n^2/(2k)}{\binom{n}{2}} = \frac{k-1}{k},
\]
and
\[
\lim_{n\to \infty} \frac{|\mathcal{H}_{n}|}{\binom{n}{3}} = \frac{(k-1)n^3/(6k^2)}{\binom{n}{3}} = \frac{k-1}{k^2}.
\]
Therefore, the sequence $\left( \mathcal{H}_{n} \right)_{n = 1}^{\infty}$ realizes $\left({(k-1)}/{k}, {(k-1)}/{k^2} \right)$.
So, $g(\mathcal{T}_{3}, (k-1)/k) \ge (k-1)/k^2$ for all integers $k$ with $k \equiv 1$ or $3$ (mod 6).

Next we prove the upper bound.
Let $\left( \mathcal{H}_{k} \right)_{k = 1}^{\infty}$ be a good sequence of cancellative $3$-graph that realizes $(x,y)$.
Let $x_{k} = 2|\partial \mathcal{H}_{k}| / \left(v(\mathcal{H})\right)^2$ and $y_{k} = 6|\mathcal{H}_{k}| / \left(v(\mathcal{H})\right)^3$ for $k \ge 1$.
Then, it follows from Theorem \ref{cancel-shadow-inequ-right} that
\[
\frac{y_k \left(v(\mathcal{H}_{k})\right)^3}{6} \le
\frac{\left( \left(v(\mathcal{H}_{k})\right)^2 - x_k \left(v(\mathcal{H}_{k})\right)^2 \right) {x_{k} \left(v(\mathcal{H}_{k})\right)^2 }/{2}}{3 v(\mathcal{H}_{k})} + 3\left(v(\mathcal{H}_{k})\right)^2,
\]
which implies
\[
y_k \le x_k(1-x_k) + \frac{18}{v(\mathcal{H}_{k})}.
\]
Letting $k \to \infty$, we obtain $y \le x(1-x)$, and this completes the proof.
\end{proof}

The idea of the proof of Theorem \ref{cancel-shadow-inequ-right}
is to first choose $S \subset V(\mathcal{H})$ such that $\left(\partial\mathcal{H}\right)[S]$ is a clique.
Then we apply the induction hypothesis to $V(\mathcal{H})\setminus S$.
However, in order to do the induction we need to prove a stronger statement
which implies Theorem \ref{cancel-shadow-inequ-right}.

We will use $G$ to denote the graph $\partial \mathcal{H}$.
Let $U \subset V(\mathcal{H})$ and let $G_{U} = G[U]$ and $\mathcal{H}_{U} = \mathcal{H}[U]$.
\begin{thm}\label{induction-thm-cancel}
Let $\mathcal{H}$ be a cancellative $3$-graph on $n$ vertices.
Let $U\subset V(\mathcal{H})$ be a set of size $m$.
Suppose that $|G_{U}|=xm^2/2$ for some real number $x$ with $0 \le x \le (m-1)/m$.
Then,
\[
|\mathcal{H}_{U}|\le \frac{(1-x)x}{6}m^3+3m^2.
\]
\end{thm}

In particular, letting $U = V(\mathcal{H})$ in Theorem \ref{induction-thm-cancel} we obtain
\[
|\mathcal{H}| \le \frac{\left( n^2 - 2|\partial \mathcal{H}|\right) |\partial \mathcal{H}| }{3n} + 3n^2,
\]
which is exactly Theorem \ref{cancel-shadow-inequ-right}.

The proof of Theorem \ref{induction-thm-cancel} is by induction on $m$.
Note that Theorem \ref{induction-thm-cancel} holds trivially for all $m\le 20$ since $\binom{m}{3} \le 3m^2$ for all $m\le 20$.
Also, by  Theorem \ref{fisher-ryan-cancel},
\[
|\mathcal{H}_{U}|\le \frac{|\partial\left( \mathcal{H}_{U} \right)|^{3/2}}{3\sqrt{3}}
\le \frac{|G_{U}|^{3/2}}{3\sqrt{3}} = \frac{x^{3/2}}{6\sqrt{6}}m^3,
\]
which is less than $x(1-x)m^3/6 + 3m^2$ when $x\le 2/3$.
Therefore, Theorem \ref{induction-thm-cancel} is true for all $x\le 2/3$, and hence we may assume that $x>2/3$ in the rest of the proof.

In the proof of Theorem \ref{induction-thm-cancel} we need the following extension of Tur\'{a}n's theorem.
The clique number $\omega(G)$ of a graph $G$ is the largest integer $\omega$ such that there is a copy of $K_{\omega}$ in $G$.
Tur\'{a}n's theorem implies that any $n$-vertex graph with no $K_{\omega+1}$ has at most $(\omega-1)n^2/(2\omega)$ edges.
\begin{thm}[\cite{TU41}]\label{Turan-theorem}
Let $G$ be an $n$-vertex graph with at least $xn^2/2$ edges for some real number $x \ge 0$.
Then $\omega(G) \ge \lc1/(1-x) \rc$.
\end{thm}
\begin{proof}
Let $\omega = \omega(G)$.
By Tur\'{a}n's theorem, $xn^2/2\le (\omega-1)n^2/(2\omega)$.
Simplifying this inequality we obtain $\omega\ge 1/(1-x)$.
Since $\omega$ is an integer, $\omega\ge \lc1/(1-x) \rc$.
\end{proof}

The idea in the proof of Theorem \ref{induction-thm-cancel} is to first apply Tur\'{a}n's theorem on $G_{U}$ to find a large clique, say on $S$,
and then apply the induction hypothesis to $T = U\setminus S$ to get an upper bound for $|\mathcal{H}_{T}|$.
In order to get an upper bound for $|\mathcal{H}_{U}|$ we just need to apply Corollary \ref{cancel-clique-corol} to $\mathcal{H}_U$
to get an upper bound for $|\mathcal{H}_U \setminus \mathcal{H}_T|$.

\begin{proof}[Proof of Theorem \ref{induction-thm-cancel}]
Suppose that $G_{U}$ contains a clique on $\omega$ vertices.
Then choose $S \subset U$ of size $\omega$ so that $G_{S} \cong K_{\omega}$.
Let $T = U \setminus S$.
Let $e_{s}$ denote the number of edges in $G_{U}$ that have nonempty intersection with $S$.
Applying the induction hypothesis to $T$ we obtain
\begin{align}\label{size-HT}
|\mathcal{H}_{T}| & \le  \frac{1}{6} {\frac{xm^2 - 2e_s}{(m-\omega)^2}\left( 1 - \frac{xm^2 - 2e_s}{(m-\omega)^2}\right)} (m-\omega)^3 + 3(m-\omega)^2 \notag \\
& = \frac{(xm^2 - 2e_s)\left((m-\omega)^2 - xm^2 +2e_s\right)}{6(m-\omega)} + 3(m-\omega)^2 \notag \\
& = \frac{-4e_s^2 + \left(4xm^2 - 2(m-\omega)^2\right)e_s + xm^2(m-\omega)^2 - x^2m^4 }{6(m-\omega)} + 3(m-\omega)^2.
\end{align}
On the other hand, Corollary \ref{cancel-clique-corol} gives
\begin{align}
|\mathcal{H}_{U} \setminus \mathcal{H}_{T}| \le \sum_{v\in S} d(v) \le |G_{U}| = \frac{x}{2}m^2. \notag
\end{align}
Let
\begin{align}\label{Delta}
\Delta = \frac{x(1-x)}{6} m^3 + 3m^2 - \left(|\mathcal{H}_{T}| +  \frac{x}{2}m^2\right),
\end{align}
and in order to prove Theorem \ref{induction-thm-cancel} it suffices to show that $\Delta \ge 0$.

Next, we will consider two cases depending on the size of $\omega(G_{U})$,
and in order to keep the calculations simple, we will omit the floor and ceiling signs.

\medskip

\noindent\textbf{Case 1:} $\omega(G_{U}) \ge m/10$.\\
Then, we may let $\omega = m/10$ in $(\ref{size-HT})$, which gives
\begin{align}\label{case1-HT}
|\mathcal{H}_{T}|
& \le \frac{1}{6} \frac{xm^2 - 2e_s}{\left(9m/10\right)^2} \left(1-\frac{xm^2 - 2e_s}{\left(9m/10\right)^2}\right) \left(\frac{9m}{10}\right)^3
+ 3\left(\frac{9m}{10}\right)^2 \notag \\
& = \frac{-4 e_s^2 - \left( 4xm^2 - 81m^2/50\right) e_s + 81xm^4/100 - x^2m^4}{27m/5} + \frac{243}{100}m^2.
\end{align}
It follows from $(\ref{Delta})$ and $(\ref{case1-HT})$ that
\begin{align}\label{cas1-delta-1}
\Delta & \ge \frac{2000\left(e_s^2 - \frac{200x-81}{200}m^2 e_s \right) + (45x+50x^2)m^4 +(1539-1350x)m^3}{2700m}.
\end{align}
Note that $e_s \le 9m^2/100 + \binom{m/10}{2} \le (200x - 81)m^2/400$ since $x > 2/3$.
On the other hand, since $9m^2/100 + \binom{m/10}{2} < (200x - 81)m^2/400$ when $x > 2/3$,
we may substitute $e_s = 9m^2/100 + \binom{m/10}{2}$ into $(\ref{cas1-delta-1})$ and obtain
\[
\Delta \ge \frac{ \left(100 x^2-290 x+190\right)m^3 + (2959-2500x) m^2 +10 m}{5400} > 0,
\]
which implies $|\mathcal{H}_{U}| \le x(1-x)m^3/6+3m^2$.

\medskip

\noindent\textbf{Case 2:} $\omega(G_{U}) < m/10$.\\
Then, let $\omega = \omega(G_{U})$ in $(\ref{size-HT})$.
A simple but crucial observation is that every vertex in $T$ is adjacent to at most $\omega-1$ vertices in $S$,
since otherwise there would be a copy of $K_{\omega+1}$ in $G_{U}$, which contradicts the definition of $\omega$.
Therefore,
\begin{align}\label{size-es}
e_{s} \le (\omega-1)(m-\omega)+\binom{\omega}{2}.
\end{align}
Plugging $(\ref{size-HT})$ into $(\ref{Delta})$ we obtain
\begin{align}\label{case2-delta-1}
\Delta & \ge \frac{4\left( e_s^2 -  \frac{2 xm^2 -(m-\omega)^2 }{2} e_s \right) + x^2 \omega m^3 }{6(m-\omega)}
+ \frac{x \omega m^2}{6} + 3(\omega m - \omega^2) -\frac{x}{2}m^2.
\end{align}
Since $x \ge 2/3$ and $\omega < m/10$, we have $\left( 2xm^2 - (m-\omega)^2 \right)/4 > (\omega-1)(m-\omega)+\binom{\omega}{2}$.
Since $e_s^2 -  \frac{2 xm^2 -(m-\omega)^2 }{2} e_s$ is decreasing in $e_s$ when $e_s \le \left( 2xm^2 - (m-\omega)^2 \right)/4$,
by $(\ref{size-es})$,
we may substitute $e_s = (\omega-1)(m-\omega)+\binom{\omega}{2}$ into $(\ref{case2-delta-1})$ and obtain
\begin{align}\label{case2-delta-2}
& \Delta > \frac{(1-x)\left(-\omega^2+(2-x)m\omega\right) m^2- (2-x)m^3 + (33m-50\omega)\omega m}{6(m-\omega)}.
\end{align}
Here, we omitted a positive lower order term $\left( x\omega m^2+ (2m-\omega)^2 + 17\omega^3 \right)/\left({6(m-\omega)}\right)$.
Notice that $(-\omega^2+(2-x)m\omega)$ is increasing in $\omega$ when $\omega\le (2-x)m/2$.
On the other hand, Tur\'{a}n's theorem together with our assumption give $1/(1-x)\le \omega < m/10 < (1-x/2)m$ when $x > 2/3$.
Since $-\omega^2+(2-x)m\omega$ is increasing in $\omega$ when $\omega \le (1-x/2)m$,
\begin{align}\label{case2-turan-inequality}
& (1-x)\left(-\omega^2+(2-x)m\omega\right) m^2- (2-x)m^3  \notag \\
\ge & (1-x)\left(-\left(\frac{1}{1-x} \right)^2+(2-x)m\left(\frac{1}{1-x} \right)\right) m^2- (2-x)m^3  \notag \\
= &  -\frac{m^2}{1-x}.
\end{align}
It follows from $(\ref{case2-delta-2})$ and $(\ref{case2-turan-inequality})$ that
\[
\begin{split}
\Delta & > \frac{ -\frac{m^2}{1-x} + \omega m^2 +(32m-50\omega)\omega m }{6(m-\omega)} \\
& = \frac{ \left(\omega - \frac{1}{1-x} \right)m^2 + 50 \left( \frac{32}{50}m - \omega \right)\omega m }{6(m-\omega)} > 0,
\end{split}
\]
which implies $|\mathcal{H}_{U}| \le x(1-x)m^3/6+3m^2$.
\end{proof}

\section{Hypergraphs without expansion of cliques.}
In this section we consider the feasible region of hypergraphs without expansion of cliques.
First we will prove the following result, from which Theorem \ref{fisher-ryan-hygp} can be easily obtained.

\begin{thm}\label{fisher-ryan-K-middle}
Let $\ell \ge r\ge 2$.
Let $\mathcal{H}$ be a $\mathcal{K}_{\ell+1}^{r}$-free $r$-graph.
Then
\[
\left(\frac{|\mathcal{H}|}{\binom{\ell}{r}} \right)^{1/r}\le \left(\frac{|\partial\mathcal{H}|}{\binom{\ell}{r-1}} \right)^{1/(r-1)}.
\]
\end{thm}

In order to derive Theorem \ref{fisher-ryan-hygp} from Theorem \ref{fisher-ryan-K-middle} we need an easy observation.
Recall from (\ref{non-normal-shadow}) that for $i \le -1$,
\[
\partial_{i}\mathcal{H} = \left\{ A\in \binom{V(\mathcal{H})}{r-i}: \text{$\mathcal{H}[A]$ is a complete $r$-graph} \right\}.
\]

\begin{obs}\label{observa-fisher-ryan-K}
Let $r \ge 3$ and $\mathcal{H}$ be an $r$-graph.
If $0 \le i \le r-2$, then $\mathcal{H}$ is $\mathcal{K}_{\ell+1}^{r}$-free iff $\partial_{i} \mathcal{H}$ is $K_{\ell+1}^{r-i}$-free.
In particular, $\mathcal{H}$ is $\mathcal{K}_{\ell+1}^{r}$-free iff $\partial_{r-2} \mathcal{H}$ is $K_{\ell+1}$-free.
If $i \le -1$, then $\mathcal{H}$ is $\mathcal{K}_{\ell+1}^{r}$-free implies that $\partial_{i} \mathcal{H}$ is $\mathcal{K}_{\ell+1}^{r-i}$-free.
\end{obs}

Now we show how to prove Theorem \ref{fisher-ryan-hygp} using Theorem \ref{fisher-ryan-K-middle}.
\begin{proof}[Proof of Theorem \ref{fisher-ryan-hygp}]
Fix $r-\ell \le i \le r-2$.
Then by Observation \ref{observa-fisher-ryan-K},
$\partial_{i}\mathcal{H}$ is $\mathcal{K}_{\ell+1}^{r-i}$-free.
Since $\partial \left(\partial_{i}\mathcal{H}\right) \subset \partial_{i+1}\mathcal{H}$, it follows from Theorem \ref{fisher-ryan-K-middle} that
\[
\left( \frac{|\partial_{i}\mathcal{H}|}{\binom{\ell}{r-i}} \right)^{1/(r-i)} \le
\left( \frac{|\partial(\partial_{i}\mathcal{H}|)}{\binom{\ell}{r-i-1}} \right)^{1/(r-i-1)} \le
\left( \frac{|\partial_{i+1}\mathcal{H}|}{\binom{\ell}{r-i-1}} \right)^{1/(r-i-1)},
\]
and this completes the proof.
\end{proof}

To show that all inequalities in Theorem \ref{fisher-ryan-hygp} are tight, consider the following construction.
Fix $\alpha \in [0,1]$ and let $\mathcal{H}_{n}(\alpha)$ be
the vertex disjoint union of $T_{r}(\alpha n,\ell)$
and a set of $(1-\alpha) n$ isolated vertices.
It is clear that $\mathcal{H}_{n}(\alpha)$ is $\mathcal{K}_{\ell+1}^{r}$-free.
In order to keep the calculations simple, let us assume that $\alpha n$ is an integer that is an multiple of $\ell$.
For fixed $\ell-r \le i \le r-1$,
\[
|\partial_{i} \mathcal{H}_{n}(\alpha)| = \binom{\ell}{r-i}\left( \frac{\alpha n}{\ell} \right)^{r-i},
\]
and hence
\[
\left( \frac{|\partial_{i} \mathcal{H}_{n}(\alpha)|}{\binom{\ell}{r-i}} \right)^{\frac{1}{r-i}} = \frac{\alpha n}{\ell}.
\]
Therefore, all inequalities in Theorem \ref{fisher-ryan-hygp} are tight.

Notice that the construction above also proves the lower bound in Corollary \ref{feasible-region-K} and we omit the calculations here.

The proof of Theorem \ref{fisher-ryan-K-middle} uses some ideas in Fisher and Ryan's proof $\cite{FR92}$.
However we need to translate their proof into the language of hypergraphs, since an edge in $\partial_{i}\mathcal{H}$
might  not be equivalent to a copy of $K_{r-i}$ in $\partial_{r-2}\mathcal{H}$ for $-\ell \le i \le r-3$.
Define the clique set $\mathcal{K}_{\mathcal{H}}$ of $\mathcal{H}$ as
\[
\mathcal{K}_{\mathcal{H}}=\left\{A\subset V(\mathcal{H}): (\partial_{r-2} \mathcal{H})[A] \cong K_{|A|}\right\}.
\]
For every $E \in \partial \mathcal{H}$ let $N(E) = \{v \in V(\mathcal{H}): \{v\} \cup E \in \mathcal{H}\}$.
Recall from Section 4 that $\sigma(S) = \sum_{v\in S}d(v)$.
We first prove a lemma that will be used in the proof of Theorem \ref{fisher-ryan-K-middle}.

\begin{lemma}\label{fisher0ryan-K-main-lemma}
$\sum_{E \in \partial\mathcal{H}}\sigma(E)\le \frac{(\ell-r+1)(r-1)}{\ell}|\partial\mathcal{H}|^2$.
\end{lemma}
\begin{proof}
Let $S \subset V(\mathcal{H})$.
For every $v\in V(\mathcal{H})$ we have $d(v)=\sum_{E \in \partial\mathcal{H}}|N(E)\cap \{v\}|$.
So,
\begin{align}\label{K-sigma-equal}
\sigma(S)=\sum_{v\in S}d(v)=\sum_{v\in S}\sum_{E\in \partial\mathcal{H}}|N(E)\cap \{v\}|=\sum_{E\in \partial\mathcal{H}}|N(E)\cap S|.
\end{align}
On the other hand,
\begin{align}
\left(\sigma(S)\right)^2  =\left(\sum_{v\in S}d(v) \right)^2
&\overset{(\ref{Jensen-2})}{\le} |S|\sum_{v\in S}\left(d(v)\right)^2 = |S|\sum_{v\in S}\sum_{E\in L(v)}d(v) \notag\\
&=|S|\sum_{v\in S}\sum_{\substack{E \in \partial \mathcal{H} \\ v\in N(E)}}d(v)=|S|\sum_{E\in \partial\mathcal{H}}\sum_{v\in S\cap N(E)}d(v) \notag\\
&=|S|\sum_{E\in \partial\mathcal{H}}\sigma\left(N(E)\cap S\right), \notag
\end{align}
which implies
\begin{align}\label{K-sigma-lower-bound}
\sum_{E\in \partial\mathcal{H}}\sigma\left(N(E)\cap S\right) \ge \frac{\left(\sigma(S)\right)^2}{|S|}.
\end{align}
Now suppose that $S \in \mathcal{K}_{\mathcal{H}}$.
Since $\mathcal{H}$ is $\mathcal{K}_{\ell+1}^{r}$-free, $|E|+|N(E)\cap S|\le \ell$ for all $E \in \partial\mathcal{H}$.
It follows from $(\ref{K-sigma-equal})$ that
\begin{align}\label{K-sigma-upper-bound}
\sigma(S) = \sum_{T\in \partial\mathcal{H}} |N(T)\cap S| \le (\ell-r+1)|\partial\mathcal{H}|.
\end{align}
Let $z$ be the largest real number such that $\sigma(R)\le (\ell-r+1)|\partial\mathcal{H}|-(\ell-|R|)z$ for all $R\in \mathcal{K}_{\mathcal{H}}$.
Let $R_0\in  \mathcal{K}_{\mathcal{H}}$ such that
\begin{align}\label{K-sigma-R0}
\sigma(R_0)= (\ell-r+1)|\partial\mathcal{H}|-(\ell-|R_0|)z.
\end{align}
For every $E \in \partial\mathcal{H}$, $E \cup \left(N(E)\cap R_0\right)\in \mathcal{K}_{\mathcal{H}}$,
therefore,
\begin{align}\label{K-sum-sigma-upper-bound}
\sum_{ E\in \partial\mathcal{H} } \sigma(E) & =
\sum_{ E\in \partial\mathcal{H} } \left( \sigma(E\cup (N(E)\cap R_0)) - \sigma(N(E)\cap R_0 ) \right) \notag\\
& \le \sum_{ E\in \partial\mathcal{H} } \left( (\ell-r+1)|\partial\mathcal{H}|
- ( \ell-|E\cup (N(E)\cap R_0)| )z - \sigma( N(E)\cap R_0 ) \right) \notag\\
& \le \sum_{ E\in \partial\mathcal{H} } \left( (\ell-r+1) \left(|\partial\mathcal{H}| - z\right)
+ |N(E)\cap R_0|z - \sigma( N(E)\cap R_0 ) \right) \notag\\
& =(\ell-r+1)(|\partial\mathcal{H}|-z)|\partial\mathcal{H}|
+ z\sum_{E\in \partial\mathcal{H}}|N(E)\cap R_0|-\sum_{E\in \partial\mathcal{H}}\sigma(N(E)\cap R_0) \notag\\
&\overset{(\ref{K-sigma-upper-bound}),(\ref{K-sigma-lower-bound})}{\le}
(\ell-r+1)(|\partial\mathcal{H}|-z)|\partial\mathcal{H}| + z\sigma(R_0)-\frac{\left(\sigma(R_0)\right)^2}{|R_0|} \notag\\
&\overset{(\ref{K-sigma-R0})}{=} (\ell-r+1)(|\partial\mathcal{H}|-2z)|\partial\mathcal{H}| + z^2 \ell
 -\frac{((\ell-r+1)|\partial\mathcal{H}| - z\ell)^2}{|R_0|}.
\end{align}
Since $|R_0|\le \ell$, we may plug $|R_0| = \ell$ into $(\ref{K-sum-sigma-upper-bound})$ and $z$ will be cancelled in the calculation and hence
\begin{align}
\sum_{E \in \partial\mathcal{H}}\sigma(E)\le \frac{(\ell-r+1)(r-1)}{\ell}|\partial\mathcal{H}|^2. \notag
\end{align}
\end{proof}

Now we are ready to prove Theorem \ref{fisher-ryan-K-middle}.
\begin{proof}[Proof of Theorem \ref{fisher-ryan-K-middle}]
We proceed by induction on $r$.
The case $r=2$ is just Mantel's theorem, so we may assume that $r\ge 3$.

For every $v\in V(\mathcal{H})$ the link $L(v)$ is a $\mathcal{K}_{\ell}^{r-1}$-free $(r-1)$-graph,
therefore, by the induction hypothesis,
\begin{align}\label{K-induction-upper}
d(v) \le \binom{\ell-1}{r-1}\left( \frac{|\partial L(v)|}{\binom{\ell-1}{r-2}} \right)^{\frac{r-1}{r-2}}.
\end{align}
It follows that
\begin{align}\label{K-H-upper}
|\mathcal{H}|
=\frac{1}{r}\sum_{v\in V(\mathcal{H})}d(v)
& = \frac{1}{r}\sum_{v\in V(\mathcal{H})}\left(d(v)\right)^{\frac{1}{r-1}} \left(d(v)\right)^{\frac{r-2}{r-1}} \notag \\
& \overset{(\ref{K-induction-upper})}{\le}
\frac{\binom{\ell-1}{r-1}^{\frac{r-2}{r-1}}}{r\binom{\ell-1}{r-2}}\sum_{v\in V(\mathcal{H})} \left(d(v)\right)^{\frac{1}{r-1}} |\partial L(v)|.
\end{align}
Similar to $(\ref{double-count-inequality})$ in Section 4, we have
\begin{align}\label{K-double-count}
\sum_{v\in V(\mathcal{H})} \left(d(v)\right)^{\frac{1}{r-1}}|\partial L(v)|
& = \sum_{E\in \partial\mathcal{H}}\sum_{v\in E}\left(d(v)\right)^{\frac{1}{r-1}} \notag\\
&\overset{(\ref{Jensen-1})}{\le}
\left((r-1)|\partial\mathcal{H}| \right)^{\frac{r-2}{r-1}}\left(\sum_{E\in \partial\mathcal{H}}\sum_{v\in E}d(v) \right)^{\frac{1}{r-1}} \notag\\
& =\left((r-1)|\partial\mathcal{H}| \right)^{\frac{r-2}{r-1}}\left(\sum_{E\in \partial\mathcal{H}}\sigma(E) \right)^{\frac{1}{r-1}} \notag \\
&\overset{\text{Lemma \ref{fisher0ryan-K-main-lemma}}}{\le}
(r-1)\left(\frac{\ell-r+1}{\ell} \right)^{\frac{1}{r-1}}|\partial\mathcal{H}|^{\frac{r}{r-1}}.
\end{align}
It follows from $(\ref{K-H-upper})$ and $(\ref{K-double-count})$ that
\[
|\mathcal{H}|\le \binom{\ell}{r}\left(\frac{|\partial\mathcal{H}|}{\binom{\ell}{r-1}} \right)^{\frac{r}{r-1}}.
\]
\end{proof}

Now we show how to prove Corollary \ref{feasible-region-K} using Theorem \ref{fisher-ryan-hygp}.

\begin{proof}[Proof of Corollary \ref{feasible-region-K}]
Let $\left(\mathcal{H}_{k} \right)_{k=1}^{\infty}$ be a good sequence of $\mathcal{K}_{\ell + 1}^{r}$-free $r$-graphs that realizes $(x,y)$.
Let $x_{k} = (r-1)! |\partial \mathcal{H}_{k}| / \left( v(\mathcal{H}_{k}) \right)^{r-1}$ and
$y_{k} = r! |\mathcal{H}_{k}| / \left( v(\mathcal{H}_{k}) \right)^{r}$.
First, we show that ${\rm proj}\Omega(\mathcal{K}_{\ell+1}^{r}) = [0,(\ell)_{r-1} / \ell^{r-1}]$.

It follows from Theorem \ref{fisher-ryan-hygp} that
\[
\frac{x_k\left( v(\mathcal{H}_{k}) \right)^{r-1}}{(r-1)!} \le \binom{\ell}{r-1}\left(\frac{v(\mathcal{H}_k)}{\ell} \right)^{r-1},
\]
which implies $x_k \le (\ell)_{r-1} / \ell^{r-1}$.
Letting $k \to \infty$, we obtain $x \le (\ell)_{r-1} / \ell^{r-1}$.
Therefore, ${\rm proj}\Omega(\mathcal{K}_{\ell + 1}^{r}) \subset [0,(\ell)_{r-1} / \ell^{r-1}]$.
On the other hand, $\left( T_{r}(k,\ell) \right)_{k=1}^{\infty}$ shows that $(\ell)_{r-1} / \ell^{r-1} \in {\rm proj}\Omega(\mathcal{K}_{\ell + 1}^{r})$
and it follows from Observation \ref{projection-x-interval} that ${\rm proj}\Omega(\mathcal{K}_{\ell + 1}^{r}) = [0,(\ell)_{r-1} / \ell^{r-1}]$.

Next, we show the upper bound for $g(\mathcal{K}_{\ell + 1}^{r},x)$.
It follows from Theorem \ref{fisher-ryan-hygp} that
\[
\left( \frac{y_k \left( v(\mathcal{H}_{k}) \right)^{r}}{r!\binom{\ell}{r}} \right)^{\frac{1}{r}} \le
\left( \frac{x_k\left( v(\mathcal{H}_{k}) \right)^{r-1}}{(r-1)!\binom{\ell}{r-1}} \right)^{\frac{1}{r-1}},
\]
which implies $y_k \le (\ell - r+1) \left(x_k^r/(\ell)_r \right)^{1/(r-1)}$.
Letting $k \to \infty$, we obtain $y \le (\ell - r+1) \left(x^r/(\ell)_r \right)^{1/(r-1)}$.
Therefore, $g(\mathcal{K}_{\ell + 1}^{r},x) \le (\ell - r+1) \left(x^r/(\ell)_r \right)^{1/(r-1)}$
for all $x \in {\rm proj}\Omega(\mathcal{K}_{\ell + 1}^{r})$.

The construction for the lower bound is exactly the same as the construction for Theorem \ref{fisher-ryan-hygp},
and it shows that $g(\mathcal{K}_{\ell + 1}^{r},x) \ge (\ell - r+1) \left(x^r/(\ell)_r \right)^{1/(r-1)}$
for all $x \in {\rm proj}\Omega(\mathcal{K}_{\ell + 1}^{r})$.
Therefore, $g(\mathcal{K}_{\ell + 1}^{r},x) = (\ell - r+1) \left(x^r/(\ell)_r \right)^{1/(r-1)}$
for all $x \in {\rm proj}\Omega(\mathcal{K}_{\ell + 1}^{r})$.
\end{proof}

Let us present a lemma before proving Theorem \ref{feasible-region-H}.

\begin{lemma}\label{feasible-region-removal-lemma}
Let $r \ge 3$ and $\mathcal{F}_1, \mathcal{F}_2$ be two families of $r$-graphs with $\mathcal{F}_1 \subset \mathcal{F}_2$.
Suppose that every $n$-vertex $\mathcal{F}_1$-free $r$-graph can be made $\mathcal{F}_2$-free by removing at most $o(n^r)$ edges,
and $g(\mathcal{F}_2,x)$ is increasing on $[0,c]$ for some $c > 0$.
Then $g(\mathcal{F}_1,x) = g(\mathcal{F}_2,x)$ on $[0,c]$.
\end{lemma}
\begin{proof}
Since $\mathcal{F}_{1} \subset \mathcal{F}_2$, it follows from Observation \ref{subset-inequality}
that $g(\mathcal{F}_{2},x) \le g(\mathcal{F}_1,x)$ for all $x \in {\rm proj}\Omega(\mathcal{F}_2)$.
So it suffices to show that $g(\mathcal{F}_{2},x) \ge g(\mathcal{F}_1,x)$ for all $x \in [0,c]$.
Let $(x_0,y_0) \in \Omega(\mathcal{F}_1)$ with $x_0 \in [0,c]$ and $y_0 = g(\mathcal{F}_1,x_0)$.
By definition, there exists a sequence of $\mathcal{F}_1$-free $r$-graphs $\left( \mathcal{H}_{k}\right)_{k = 1}^{\infty}$
with $\lim_{k \to \infty}d(\partial\mathcal{H}_{k}) = x_0$ and $\lim_{k \to \infty}d(\mathcal{H}_{k}) = y_0$.

For every $k \ge 1$
let $\mathcal{H}'_{k}$ be a subgraph of $\mathcal{H}_k$ that is $\mathcal{F}_2$-free and of maximum size,
and let $x'_{k} = d(\partial \mathcal{H}'_{k})$ and $y'_{k} = d(\mathcal{H}'_{k})$.
By the Bolzano-Weierstrass theorem, $\left( x'_{k}, y'_{k} \right)_{k = 1}^{\infty}$ contains a convergent subsequence
$\left( x'_{t_k}, y'_{t_k} \right)_{k = 1}^{\infty}$.
Let $x_0' = \lim_{k \to \infty}  x'_{t_k} $ and $y_0' = \lim_{k \to \infty}  y'_{t_k}$,
and it is easy to see from the definition of $\mathcal{H}'_{k}$ that $x_0' \le x_0$ and $y_0' \le y_0$.
Since $\left( \mathcal{H}'_{t_k}\right)_{k = 1}^{\infty}$ is a good sequence of $\mathcal{F}_2$-free $r$-graphs that realizes $(x'_0, y'_0)$,
we obtain $(x'_0, y'_0) \in \Omega(\mathcal{F}_2)$.

By assumption, for every $\epsilon > 0$ there exists $n(\epsilon)$ such that
$\mathcal{H}_{k}$ can be made $\mathcal{F}_2$-free by removing at most $\epsilon \left(v(\mathcal{H}_{k})\right)^r$ edges
whenever $v(\mathcal{H}_{k}) \ge n(\epsilon)$.
Since $\lim_{k \to \infty} v(\mathcal{H}_{k}) = \infty$, there exists $k(\epsilon)$ such that $v(\mathcal{H}_{k}) \ge n(\epsilon)$
for all $k \ge k(\epsilon)$, and hence $|\mathcal{H}'_{k}| \ge |\mathcal{H}_k| - \epsilon \left(v(\mathcal{H}_{k})\right)^r$ for all $k \ge k(\epsilon)$.
Therefore, $y_0' \ge y_0 - r! \epsilon$.
Letting $\epsilon \to 0$, we obtain $y_0' \ge y_0$, and hence $y_0' = y_0$.
Therefore, $(x'_0, y_0) \in \Omega(\mathcal{F}_2)$.
By the assumption that $g(\mathcal{F}_2)$ is increasing on $[0,c]$, we obtain
\[
g(\mathcal{F}_2, x_0) \ge g(\mathcal{F}_2, x'_0) \ge y_0 = g(\mathcal{F}_1, x_0).
\]
Since $x_0$ was chosen arbitrarily from $[0,c]$, $g(\mathcal{F}_2, x) \ge g(\mathcal{F}_1, x)$ for all $x\in[0,c]$,
and this completes the proof.
\end{proof}

Now we prove Theorem \ref{feasible-region-H} using Corollary \ref{feasible-region-K}.
\begin{proof}[Proof of Theorem \ref{feasible-region-H}]
It was shown by Pikhurko (see the proof of Lemma 3 in \cite{PI13}) that every $H_{\ell + 1}^{r}$-free $r$-graph
on $n$-vertices can be made $\mathcal{K}_{\ell + 1}^{r}$-free by removing at most $o(n^r)$ edges.
On the other hand, Corollary \ref{feasible-region-K} shows that $g(\mathcal{K}_{\ell + 1}^{r})$ is increasing on $[0,(\ell)_{r-1} / \ell^{r-1}]$.
So, it follows from Lemma \ref{feasible-region-removal-lemma} that
\[
g(H_{\ell + 1}^{r},x) = g(\mathcal{K}_{\ell + 1}^{r},x) = (\ell - r+1) \left(\frac{x^r}{(\ell)_r} \right)^{\frac{1}{r-1}}
\]
for all $x \in [0,(\ell)_{r-1} / \ell^{r-1}]$.
\end{proof}

\section{Concluding remarks}
In this paper we proved that for any $r \ge 3$ and any family $\mathcal{F}$ of $r$-graphs
the function $g(\mathcal{F})$ has at most countably many discontinuities.
We also constructed a family $\mathcal{D}$ of $3$-graphs such that $g(\mathcal{D})$ is discontinuous at $x = 2/3$.
It seems natural to ask the following question.

\begin{prob}\label{infinity-many-discontinuity}
Can $g(\mathcal{F})$ have infinitely many discontinuities?
\end{prob}

In Section 4 we proved several results about $g(\mathcal{T}_{r})$ for $r \ge 3$.
Even for $r =3$ the function $g(\mathcal{T}_3)$ is already shown to have many intersecting properties,
and is closely related to Steiner triple systems.
The following question seems difficult for $x$ not of the form $(k-1)/k$ with $k \equiv 1$ or $3$ (mod $6$).

\begin{prob}\label{determine-cancellative-3}
Determine $g(\mathcal{T}_{3},x)$ for all $x \in (2/3,1]$.
\end{prob}

Let us show a lower bound for $g(\mathcal{T}_{3},x)$ for $x \in (2/3, 6/7]$.

Let $\mathbb{F}$ denote the Fano Plane, i.e., $\mathbb{F}$ is a $3$-graph on $7$ vertices with edge set
\[
\{123, 345,561,174,275,376,246\}.
\]
Let $\alpha \in [1/7,1/3]$ and $\beta = (1-3\alpha)/4$.
Let $\mathcal{H}_n(\alpha)$ be obtained from $\mathbb{F}$ by
blowing up each vertex in $\{1,2,3\}$ into a set of size of $\alpha n$ and blowing up each vertex in $\{4,5,6,7\}$ into a set of size of $\beta n$.
Let
\begin{equation}\label{const-cancel-3-x}
x = \lim_{n \to \infty} \frac{|\partial \mathcal{H}_n(\alpha) |}{\binom{n}{2}}
  = 6 \alpha^2 + 12 \beta^2 + 24 \alpha \beta = \frac{3}{4}(1+2\alpha-7\alpha^2),
\end{equation}
and
\begin{equation}\label{const-cancel-3-y}
y = \lim_{n \to \infty} \frac{|\mathcal{H}_n(\alpha) |}{\binom{n}{3}}
  = 6 \alpha^3 + 36\alpha \beta^2 = \frac{3}{4}\alpha (3-18\alpha+35\alpha^2).
\end{equation}
Then, $(\ref{const-cancel-3-x})$ and $(\ref{const-cancel-3-y})$ give
\begin{equation}\label{cancel-steiner-lower-bound}
y = \frac{1}{147} \left(-70 \sqrt{18x^2 - 21 x^3} +63 x+60 \sqrt{18-21 x}-36\right),
\end{equation}
which implies
\begin{equation}
g(\mathcal{T}_{3},x) \ge \frac{1}{147} \left(-70 \sqrt{18x^2 - 21 x^3} +63 x+60 \sqrt{18-21 x}-36\right) \notag
\end{equation}
for all $x\in [2/3,6/7]$.

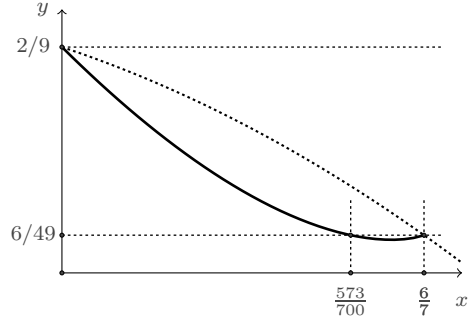
\begin{figure}[htbp]
\centering
\begin{tikzpicture}[xscale=25,yscale=25]
\draw [line width=0.5pt,dash pattern=on 1pt off 1.2pt,domain=2/3:6/7+0.01] plot(\x,2/9);
\draw [line width=0.5pt,dash pattern=on 1pt off 1.2pt,domain=2/3:6/7+0.01] plot(\x,6/49);
\draw [line width=0.5pt,dash pattern=on 1pt off 1.2pt] (6/7,6/49-0.02)--(6/7,6/49+0.02);
\draw [line width=0.5pt,dash pattern=on 1pt off 1.2pt] (573/700,6/49-0.02)--(573/700,6/49+0.02);
\draw [->] (2/3,6/49-0.02)--(6/7+0.02,6/49-0.02);
\draw [->] (2/3,6/49-0.02)--(2/3,2/9+0.02);
\draw[line width=1pt]
(0.666667,0.222222)--(0.671667,0.217269)--(0.676667,0.212411)--(0.681667,0.20765)--(0.686667,0.202986)--
(0.691667,0.198421)--(0.696667,0.193957)--(0.701667,0.189595)--(0.706667,0.185336)--(0.711667,0.181184)--
(0.716667,0.177138)--(0.721667,0.173202)--(0.726667,0.169377)--(0.731667,0.165665)--(0.736667,0.162068)--
(0.741667,0.15859)--(0.746667,0.155232)--(0.751667,0.151997)--(0.756667,0.148888)--(0.761667,0.145908)--
(0.766667,0.143061)--(0.771667,0.140349)--(0.776667,0.137778)--(0.781667,0.135351)--(0.786667,0.133073)--
(0.791667,0.130949)--(0.796667,0.128985)--(0.801667,0.127187)--(0.806667,0.125563)--(0.811667,0.124122)--
(0.816667,0.122872)--(0.821667,0.121826)--(0.826667,0.120998)--(0.831667,0.120404)--(0.836667,0.120067)--
(0.841667,0.120018)--(0.846667,0.120299)--(0.851667,0.120986)--(0.856667,0.122268);
\draw[line width=0.8pt,dash pattern=on 1pt off 1.2pt]
(0.666667,0.222222)--(0.671667,0.220531)--(0.676667,0.218789)--(0.681667,0.216997)--(0.686667,0.215156)--
(0.691667,0.213264)--(0.696667,0.211322)--(0.701667,0.209331)--(0.706667,0.207289)--(0.711667,0.205197)--
(0.716667,0.203056)--(0.721667,0.200864)--(0.726667,0.198622)--(0.731667,0.196331)--(0.736667,0.193989)--
(0.741667,0.191597)--(0.746667,0.189156)--(0.751667,0.186664)--(0.756667,0.184122)--(0.761667,0.181531)--
(0.766667,0.178889)--(0.771667,0.176197)--(0.776667,0.173456)--(0.781667,0.170664)--(0.786667,0.167822)--
(0.791667,0.164931)--(0.796667,0.161989)--(0.801667,0.158997)--(0.806667,0.155956)--(0.811667,0.152864)--
(0.816667,0.149722)--(0.821667,0.146531)--(0.826667,0.143289)--(0.831667,0.139997)--(0.836667,0.136656)--
(0.841667,0.133264)--(0.846667,0.129822)--(0.851667,0.126331)--(0.856667,0.122789)--(0.861667,0.119197)--
(0.866667,0.115556)--(0.871667,0.111864)--(0.876667,0.108122);
\begin{scriptsize}
\draw [fill=uuuuuu] (2/3,6/49-0.02) circle (0.03pt);
\draw [fill=uuuuuu] (573/700,6/49) circle (0.03pt);
\draw [fill=uuuuuu] (2/3,6/49) circle (0.03pt);
\draw[color=uuuuuu] (2/3-0.015,6/49) node {$6/49$};
\draw [fill=uuuuuu] (2/3,2/9) circle (0.03pt);
\draw[color=uuuuuu] (2/3-0.015,2/9) node {$2/9$};
\draw [fill=uuuuuu] (6/7,6/49-0.02) circle (0.03pt);
\draw[color=uuuuuu] (6/7,6/49-0.035) node {$\frac{6}{7}$};
\draw [fill=uuuuuu] (6/7,6/49) circle (0.03pt);
\draw[color=uuuuuu] (6/7,6/49-0.035) node {$\frac{6}{7}$};
\draw [fill=uuuuuu] (573/700,6/49-0.02) circle (0.03pt);
\draw[color=uuuuuu] (573/700,6/49-0.035) node {$\frac{573}{700}$};
\draw[color=uuuuuu] (6/7+0.02,6/49-0.035) node {$x$};
\draw[color=uuuuuu] (2/3-0.01,2/9+0.02) node {$y$};
\end{scriptsize}
\end{tikzpicture}
\caption{The lower bound for $g(\mathcal{T}_3, x)$ given by $(\ref{cancel-steiner-lower-bound})$.}
\end{figure}

The construction above gives an algebraic curve connecting $(2/3,2/9)$ and $(6/7,6/49)$.
Using a similar method, one can construct an algebraic curve defined by
\begin{equation}\label{cancel-steiner-general-lower-bound}
y = \frac{2\sqrt{3}(k+3)(k-1-kx)^{\frac{3}{2}}}{3k^2\sqrt{k-3}}+\frac{3kx-2k+2}{k^2}
\end{equation}
to connect $(2/3,2/9)$ and $((k-1)/k,(k-1)/k^2)$ for all $k \equiv 1$ or $3$ (mod $6$).
However, we do not know how to construct curves to connect $((k-1)/k,(k-1)/k^2)$ and $((k'-1)/k',(k'-1)/k'^2)$
for all $k,k' \ge 7$ and $k,k' \equiv 1$ or $3$ (mod $6$).
Also, there is an interesting phenomenon that
\[
\left\{ ((k-1)/k,(k-1)/k^2): \text{$k \ge 7$ and $k \equiv 1$ or $3$ (mod $6$)} \right\}
\]
are local maximums of the function given by $(\ref{cancel-steiner-general-lower-bound})$.
Therefore, we pose the following question.

\begin{prob}\label{steiner-points-maximum}
For every $k\ge 7$ with $k \equiv 1$ or $3$ (mod $6$),
is the point $((k-1)/k,(k-1)/k^2)$ a local maximum of $g(\mathcal{T}_{3})$?
\end{prob}

In \cite{LM19B}, we prove the following stability theorem about the points $((k-1)/k,(k-1)/k^2)$ in $\Omega(\mathcal{T}_3)$,
which we think might be helpful for Problems \ref{determine-cancellative-3} and \ref{steiner-points-maximum}.
\begin{thm}[Stability, \cite{LM19B}]\label{stability-steiner}
Let $k$ be an integer with $k\equiv 1$ or $3$ (mod $6$) and $\mathcal{H}$ be a cancellative $3$-graph on $n$ vertices.
For every $\delta>0$ there exists an $\epsilon>0$ and $n_0$ such that the following holds for all $n\ge n_0$.
Suppose that $|\partial\mathcal{H}|\ge (1-\epsilon)(k-1)n^2/(2k)$ and $|\mathcal{H}|\ge (1-\epsilon)(k-1)n^3/(6k^2)$.
Then $\mathcal{H}$ can be transformed into a subgraph of a $3$-graph in $\mathcal{S}(n,k)$ by removing at most $\delta n^3$ edges.
\end{thm}

We also have an exact result for the points $((k-1)/k, (k-1)/k^2)$.
Let $s(n,k) = \max\{|\mathcal{H}| : \mathcal{H} \in \mathcal{S}(n,k)\}$.
\begin{thm}[\cite{LM19B}]\label{exact-steiner}
Let $k$ be an integer that satisfies $k\equiv 1$ or $3$ (mod $6$) and
$\mathcal{H}$ be a cancellative $3$-graph on $n$ vertices with $n$ sufficiently large.
Suppose that $|\partial\mathcal{H}| = t_{2}(n,k)$.
Then $|\mathcal{H}|\le s(n,k)$, and equality holds only if $\mathcal{H} \in \mathcal{S}(n,k)$.
\end{thm}

For $r\ge 4$. There is very little known about upper and lower bounds for $g(\mathcal{T}_{r},x)$ for $x > (r-1)! / r^{r-2}$.
We pose the following question.
\begin{prob}\label{cancellative-r}
Let $r \ge 4$ and $x > (r-1)! / r^{r-2}$.
Improve the upper bound for $g(\mathcal{T}_{r},x)$,
and construct cancellative $r$-graphs to give good lower bounds for $g(\mathcal{T}_{r},x)$.
\end{prob}

Given our poor understanding of hypergraph Tur\'{a}n problems,
determining the feasible region of other families of hypergraphs would also be of interest.
In particular, we pose the following two questions.
\begin{prob}\label{region-H}
Determine the feasible region of $H_{\ell+1}^{r}$ for $r \ge 3$ and $\ell \ge r$.
\end{prob}

\begin{prob}\label{region-F}
Determine the feasible region of the Fano Plane.
\end{prob}

In $\cite{LM19}$, we give an example of a (finite) family $\mathcal{F}$, for which $g(\mathcal{F})$ has two global maximums.
In particular, our example shows that $g(\mathcal{F})$ can be non-unimodal.

\begin{thm}[\cite{LM19}]\label{nonunimodal}
There exists a (finite) family $\mathcal{M}$ of $3$-graphs such that $g(\mathcal{M},x) \le 4/9$ for all $x \in {\rm proj}\Omega(\mathcal{M})$,
and equality holds iff $x \in \{5/6,8/9\}$.
\end{thm}

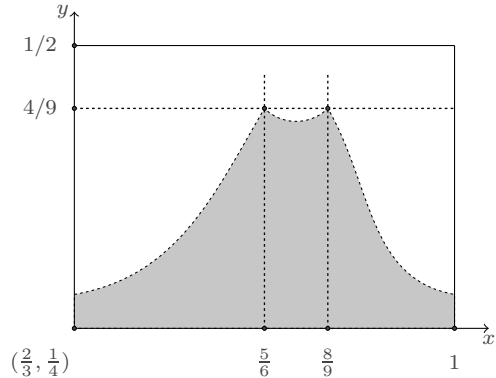
\begin{figure}[htbp]
\centering
\begin{tikzpicture}[xscale=15,yscale=15]
\draw [line width=0.5pt,dash pattern=on 1pt off 1.2pt,domain=2/3:1] plot(\x,4/9);
\draw [line width=0.5pt,dash pattern=on 1pt off 1.2pt] (5/6,1/4)--(5/6,4/9+0.03);
\draw [line width=0.5pt,dash pattern=on 1pt off 1.2pt] (8/9,1/4)--(8/9,4/9+0.03);
\draw [->] (2/3,1/4)--(1+0.03,1/4);
\draw [->] (2/3,1/4)--(2/3,1/2+0.03);
\draw (1,1/4)--(1,1/2);
\draw (2/3,1/2)--(1,1/2);
\draw[fill=sqsqsq,fill opacity=0.25, dash pattern=on 1pt off 1.2pt]
(2/3,1/4+0.03) to [out = 10, in = 240] (5/6,4/9) to [out = 315, in = 225] (8/9,4/9) to [out = 300, in = 170] (1,1/4+0.03)
-- (1,1/4) -- (2/3,1/4) -- (2/3,1/4+ 0.03);
\begin{scriptsize}
\draw [fill=uuuuuu] (5/6,4/9) circle (0.05pt);
\draw [fill=uuuuuu] (8/9,4/9) circle (0.05pt);
\draw [fill=uuuuuu] (2/3,1/4) circle (0.05pt);
\draw[color=uuuuuu] (2/3-0.03,1/4-0.03) node {$(\frac{2}{3}, \frac{1}{4})$};
\draw [fill=uuuuuu] (2/3,1/2) circle (0.05pt);
\draw[color=uuuuuu] (2/3-0.03,1/2) node {$1/2$};
\draw [fill=uuuuuu] (2/3,4/9) circle (0.05pt);
\draw[color=uuuuuu] (2/3-0.03,4/9) node {$4/9$};
\draw [fill=uuuuuu] (5/6,1/4) circle (0.05pt);
\draw[color=uuuuuu] (5/6,1/4-0.03) node {$\frac{5}{6}$};
\draw [fill=uuuuuu] (8/9,1/4) circle (0.05pt);
\draw[color=uuuuuu] (8/9,1/4-0.03) node {$\frac{8}{9}$};
\draw [fill=uuuuuu] (1,1/4) circle (0.05pt);
\draw[color=uuuuuu] (1,1/4-0.03) node {$1$};
\draw[color=uuuuuu] (1+0.03,1/4-0.01) node {$x$};
\draw[color=uuuuuu] (2/3-0.01,1/2+0.03) node {$y$};
\end{scriptsize}
\end{tikzpicture}
\caption{$g(\mathcal{M})$ has two global maximums by Theorem \ref{nonunimodal}.}
\end{figure}

Theorem \ref{nonunimodal} suggests the following natural problem which we hope to address in the future.

\begin{prob}\label{many-global-maximum}
Fix $r\ge 3$ and $t>0$.
Does there exists a (finite) family $\mathcal{F}$ of $r$-graphs and reals
$0 < x_1 < x'_1 < x_2 < \cdots < x'_{t-1} <x_t$ such that $g(\mathcal{F},x_i) = \pi(\mathcal{F})$
for all $i \in [t]$ and $g(\mathcal{F},x_i') < \pi(\mathcal{F})$ for all $i \in [t-1]$.
\end{prob}

\begin{figure}[htbp]
\centering
\begin{tikzpicture}[xscale=8,yscale=8]
\draw [->] (0,0)--(1+0.05,0);
\draw [->] (0,0)--(0,0.5+0.05);
\draw (0,0.5)--(1,0.5);
\draw (1,0)--(1,0.5);
\draw [line width=0.8pt,dash pattern=on 1pt off 1.2pt,domain=0:1] plot(\x,1/3);
\draw [line width=0.5pt,dash pattern=on 1pt off 1.2pt] (0.2,0) -- (0.2,1/3);
\draw [line width=0.5pt,dash pattern=on 1pt off 1.2pt] (0.3,0) -- (0.3,1/3);
\draw [line width=0.5pt,dash pattern=on 1pt off 1.2pt] (0.4,0) -- (0.4,1/3);
\draw [line width=0.5pt,dash pattern=on 1pt off 1.2pt] (0.5,0) -- (0.5,1/3);
\draw [line width=0.5pt,dash pattern=on 1pt off 1.2pt] (0.6,0) -- (0.6,1/3);
\draw [line width=0.5pt,dash pattern=on 1pt off 1.2pt] (0.7,0) -- (0.7,1/3);
\draw [line width=0.5pt,dash pattern=on 1pt off 1.2pt] (0.8,0) -- (0.8,1/3);
\draw [line width=0.5pt,dash pattern=on 1pt off 2.5pt] (0.25,0) -- (0.25,1/3-0.04);
\draw [line width=0.5pt,dash pattern=on 1pt off 2.5pt] (0.35,0) -- (0.35,1/3-0.04);
\draw [line width=0.5pt,dash pattern=on 1pt off 2.5pt] (0.45,0) -- (0.45,1/3-0.04);
\draw [line width=0.5pt,dash pattern=on 1pt off 2.5pt] (0.55,0) -- (0.55,1/3-0.04);
\draw [line width=0.5pt,dash pattern=on 1pt off 2.5pt] (0.65,0) -- (0.65,1/3-0.04);
\draw [line width=0.5pt,dash pattern=on 1pt off 2.5pt] (0.75,0) -- (0.75,1/3-0.04);

\draw[fill=sqsqsq,fill opacity=0.25, dash pattern=on 1pt off 1.2pt]
(0,0)  to [out = 45, in = 250] (0.2, 1/3) to [out = 290, in = 180] (0.25,1/3-0.04) to [out = 0, in = 250] (0.3, 1/3)
to [out = 290, in = 180] (0.35,1/3-0.04) to [out = 0, in = 250] (0.4, 1/3)
to [out = 290, in = 180] (0.45,1/3-0.04) to [out = 0, in = 250] (0.5, 1/3)
to [out = 290, in = 180] (0.55,1/3-0.04) to [out = 0, in = 250] (0.6, 1/3)
to [out = 290, in = 180] (0.65,1/3-0.04) to [out = 0, in = 250] (0.7, 1/3)
to [out = 290, in = 180] (0.75,1/3-0.04) to [out = 0, in = 250] (0.8, 1/3)
to [out = 290, in = 135] (1,0.1) to (1,0);

\begin{scriptsize}
\draw [fill=uuuuuu] (1,0) circle (0.1pt);
\draw[color=uuuuuu] (1,0-0.05) node {$1$};
\draw [fill=uuuuuu] (0,0) circle (0.1pt);
\draw[color=uuuuuu] (0-0.03,0-0.03) node {$0$};
\draw [fill=uuuuuu] (0,1/2) circle (0.1pt);
\draw[color=uuuuuu] (0-0.05,1/2+0.05) node {$y$};
\draw[color=uuuuuu] (1+0.05,0-0.05) node {$x$};

\draw [fill=uuuuuu] (0.2,1/3) circle (0.1pt);
\draw [fill=uuuuuu] (0.2,0) circle (0.1pt);
\draw[color=uuuuuu] (0.2,0-0.03) node {$x_1$};

\draw [fill=uuuuuu] (0.3,1/3) circle (0.1pt);
\draw [fill=uuuuuu] (0.3,0) circle (0.1pt);
\draw[color=uuuuuu] (0.3,0-0.03) node {$x_2$};

\draw [fill=uuuuuu] (0.4,1/3) circle (0.1pt);
\draw [fill=uuuuuu] (0.4,0) circle (0.1pt);
%
\draw [fill=uuuuuu] (0.5,1/3) circle (0.1pt);
\draw [fill=uuuuuu] (0.5,0) circle (0.1pt);
%
\draw [fill=uuuuuu] (0.6,1/3) circle (0.1pt);
\draw [fill=uuuuuu] (0.6,0) circle (0.1pt);
%
\draw [fill=uuuuuu] (0.7,1/3) circle (0.1pt);
\draw [fill=uuuuuu] (0.7,0) circle (0.1pt);

\draw [fill=uuuuuu] (0.8,1/3) circle (0.1pt);
\draw [fill=uuuuuu] (0.8,0) circle (0.1pt);
\draw[color=uuuuuu] (0.8,0-0.03) node {$x_t$};


\draw [fill=uuuuuu] (0.25,1/3-0.04) circle (0.1pt);
\draw [fill=uuuuuu] (0.25,0) circle (0.1pt);
\draw[color=uuuuuu] (0.25,0-0.03) node {$x'_1$};

\draw [fill=uuuuuu] (0.35,1/3-0.04) circle (0.1pt);
\draw [fill=uuuuuu] (0.35,0) circle (0.1pt);
\draw[color=uuuuuu] (0.35,0-0.03) node {$x'_2$};

\draw [fill=uuuuuu] (0.45,1/3-0.04) circle (0.1pt);
\draw [fill=uuuuuu] (0.45,0) circle (0.1pt);
%
\draw [fill=uuuuuu] (0.55,1/3-0.04) circle (0.1pt);
\draw [fill=uuuuuu] (0.55,0) circle (0.1pt);
%
\draw [fill=uuuuuu] (0.65,1/3-0.04) circle (0.1pt);
\draw [fill=uuuuuu] (0.65,0) circle (0.1pt);

\draw [fill=uuuuuu] (0.75,1/3-0.04) circle (0.1pt);
\draw [fill=uuuuuu] (0.75,0) circle (0.1pt);
\end{scriptsize}
\end{tikzpicture}
\caption{Can $g(\mathcal{F})$ has many global maximums?}
\end{figure}
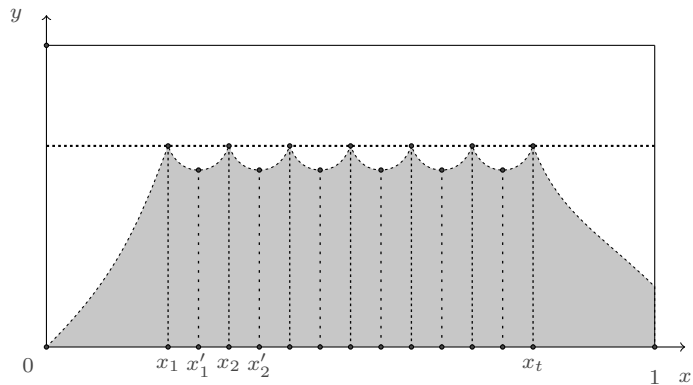
\section{Acknowledgement}
We would like to thank Minghui Ma for discussions on Theorem \ref{left-cont-and-diff} and for informing us about \cite{SS05l}.

\section{Appendix}
Here we prove the following result, which extends Theorem \ref{example-discont} to $r$-graphs with $r \ge 4$.

\begin{thm}\label{discontinity-D-r-graph}
For every $r \ge 3$ there exists a family $\mathcal{D}^{r}$ of $r$-graphs with ${\rm proj}\Omega(\mathcal{D}^{r}) = [0,1]$
and $g(\mathcal{D}^{r},(r-1)!/r^{r-2}) = r!/r^{r}$, but there exists an absolute constant $\delta_0 > 0$
such that $g(\mathcal{D}^{r},(r-1)!/r^{r-2}+\epsilon) < r!/r^{r} - \delta_0$ for all $0 < \epsilon \le 1- (r-1)!/r^{r-2}$.
\end{thm}

\begin{dfn}\label{dfn-D-r}
Let $\mathcal{D}^{r}$ be the collection of all $r$-graphs $F \in \mathcal{K}_{r+1}^{r}$
such that $F \not\subset \mathcal{S}_{n}$ for all $n \ge r$.
\end{dfn}

Since $H_{r+1}^{r} \in \mathcal{D}^{r} \subset \mathcal{K}_{r+1}^{r}$, by results in \cite{MU06} and \cite{PI13} we obtain the following results.
\begin{thm}\label{Turan-number-D-r}
Let $n$ be sufficiently large.
Then, $ex(n,\mathcal{D}^{r}) = t_{r}(n,r)$ and
$T_{r}(n,r)$ is the unique $\mathcal{D}^{r}$-free $r$-graph with $n$ vertices and $t_{r}(n,r)$ edges.
\end{thm}

\begin{thm}[Stability]\label{stability-D-r}
For every $\xi >0$ there exists $\delta >0$ and $n_0$ such that the following holds for all $n\ge n_0$.
Any $\mathcal{D}^{r}$-free $r$-graph $\mathcal{H}$ with $n$ vertices and at least $n^r/r^r - \delta n^r$ edges has a partition
$V(\mathcal{H}) = V_1 \cup \cdots \cup V_r$ such that all but at most $\xi n^r$ edges in $\mathcal{H}$
have exactly one vertex in each $V_i$.
\end{thm}

Since $\mathcal{S}_n$ is $\mathcal{D}^r$-free,
it follows from Observation \ref{projection-x-interval} that ${\rm proj}\Omega(\mathcal{D}^r) = [0,1]$.
Theorem \ref{Turan-number-D-r} implies that $g(\mathcal{D}^{r},x) \le r!/r^{r}$ for all $x \in [0,1]$ and equality holds for $x = (r-1)!/r^{r-2}$.
Therefore, in order to prove Theorem \ref{discontinity-D-r-graph}, it suffices to prove the following theorem.

\begin{thm}\label{discon-D-r-main-thm}
There exists an absolute constant $\delta_0 > 0$ such that the following is true
for all $0 < \epsilon \le 1- (r-1)!/r^{r-2}$ and sufficiently large $n$.
Suppose that $\mathcal{H}$ is a $\mathcal{D}^r$-free $r$-graph on $n$ vertices with $|\partial \mathcal{H}| = (1/r^{r-2} + \epsilon) n^{r-1}$.
Then $|\mathcal{H}| \le (1/r^{r} - \delta_0) n^r$.
\end{thm}
\begin{proof}
Suppose not.
Then for every $\delta > 0$ there exists $\epsilon > 0$ and sufficiently large $n$ so that
there is a $\mathcal{D}^{r}$-free $r$-graph on $n$ vertices with $|\partial \mathcal{H}| = (1/r^{r-2} + \epsilon) n^{r-1}$
and $|\mathcal{H}| > (1/r^{r} - \delta) n^r$.

Let $\xi > 0$ be sufficiently small and let $\delta > 0$ (we may assume that $\delta \le \xi$) and $n_0$ be given by Theorem \ref{stability-D-r}.
By assumption there exists $\epsilon > 0$ and
a $\mathcal{D}^{r}$-free $r$-graph on $n$ vertices with
\begin{align}\label{discon-D-r-size-shadow-H}
|\partial \mathcal{H}| = \left( \frac{1}{r^{r-2}} + \epsilon \right) n^{r-1}
\end{align}
and
\begin{align}\label{discon-D-r-size-H}
|\mathcal{H}| > \left( \frac{1}{r^{r}} - \delta \right) n^r.
\end{align}

By Theorem \ref{stability-D-r}, $\mathcal{H}$ has a partition
$V(\mathcal{H}) = V_1 \cup \cdots \cup V_r$ such that all but at most $\xi n^r$ edges in $\mathcal{H}$ have exactly one vertex in each $V_i$.
Let $\mathcal{H}'$ denote the induced $r$-partite subgraph of $\mathcal{H}$ with parts $V_1,\ldots, V_r$.
Let $G = \partial_{r-2}\mathcal{H}$ and $G' = \partial_{r-2}\mathcal{H}'$.
Notice that $G'$ is an $r$-partite subgraph of $G$ and
\begin{align}\label{discon-D-r-size-H'}
|\mathcal{H}'| > \frac{n^r}{r^r} - (\delta + \xi)n^{r}.
\end{align}

\begin{claim}\label{discon-r-size-of-Vi}
$\left| |V_i| - \frac{n}{r}\right| < 2{r^{\frac{r-2}{2}}}\left( \delta + \xi \right)^{\frac{1}{2}} n$ for all $i \in [r]$.
\end{claim}
\begin{proof}
Fix $i \in [r]$ and let $\alpha = |V_i|$.
Then,
\begin{align}
\alpha \left( \frac{n-\alpha}{r-1}\right)^{r-1} \ge \alpha \prod_{j \in [r]\setminus \{i\}}|V_j|
\ge |\mathcal{H}'| \overset{(\ref{discon-D-r-size-H'})}{\ge} \frac{n^r}{r^r} - \left( \delta + \xi \right)n^r,  \notag
\end{align}
which implies that
\begin{align}
\frac{n}{r} - 2{r^{\frac{r-2}{2}}}\left( \delta + \xi \right)^{\frac{1}{2}} n
< \alpha <
\frac{n}{r} + 2{r^{\frac{r-2}{2}}}\left( \delta + \xi \right)^{\frac{1}{2}} n. \notag
\end{align}
\end{proof}

For $\{u,v\} \subset V(\mathcal{H})$ the link of $\{u,v\}$ in $\mathcal{H}$ is
\[
L_{\mathcal{H}}(uv) := \left\{ A \in \binom{V(\mathcal{H})}{r-2}: \{u,v\}\cup A \in \mathcal{H} \right\},
\]
and the degree of $\{u,v\}$ in $\mathcal{H}$ is  $d_{\mathcal{H}}(uv) := |L_{\mathcal{H}}(uv)|$.
When it is clear from the text we will omit the subscript $\mathcal{H}$.

Let $K$ be an $r$-partite subgraph of $G$ with the maximum number of edges and let $V'_1,\ldots,V'_r$ denote the $r$ parts of $K$.

\begin{claim}\label{discon-r-size-of-K-G'}
$|K| \ge |G'| > \frac{r-1}{2r}n^{2} - 2r^{\frac{r+2}{2}}\left( \delta + \xi \right)^{\frac{1}{2}} n^2$.
\end{claim}
\begin{proof}
By Claim \ref{discon-r-size-of-Vi},
\begin{align}
|G'| \left( \frac{n}{r} + 2{r^{\frac{r-2}{2}}}\left( \delta + \xi \right)^{\frac{1}{2}} n \right)^{r-2}
 \ge \sum_{uv \in G'}d_{\mathcal{H}'}(u,v) & = \binom{r}{2} |\mathcal{H}'| \notag \\
& \overset{(\ref{discon-D-r-size-H'})}{\ge} \binom{r}{2}\left( \frac{n^r}{r^r} - \left( \delta + \xi \right)n^r \right),  \notag
\end{align}
which gives
\begin{align}
|G'| & \ge \frac{\binom{r}{2}\left( \frac{n^r}{r^r} - \left( \delta + \xi \right)n^r \right)}{\left( \frac{n}{r} + 2{r^{\frac{r-2}{2}}}\left( \delta + \xi \right)^{\frac{1}{2}} n \right)^{r-2} }  \notag \\
& > \frac{\binom{r}{2}\left( \frac{n^r}{r^r} - \left( \delta + \xi \right)n^r \right)\left( 1 - 2{r^{\frac{r}{2}}}\left( \delta + \xi \right)^{\frac{1}{2}} \right)^{r-2}}{\left( n/r \right)^{r-2} } \notag \\
& > \frac{\binom{r}{2}\left( \frac{n^r}{r^r} - \left( \delta + \xi \right)n^r \right)\left( 1 - 3(r-2){r^{\frac{r}{2}}}\left( \delta + \xi \right)^{\frac{1}{2}} \right)}{\left( n/r \right)^{r-2} } \notag \\
& > \frac{\binom{r}{2} \frac{n^r}{r^r}- 2{r^{\frac{r}{2}}}\left( \delta + \xi \right)^{\frac{1}{2}} \frac{n^r}{r^{r-3}}}{\left( n/r \right)^{r-2} } \notag \\
& > \frac{r-1}{2r}n^{2} - 2r^{\frac{r+2}{2}}\left( \delta + \xi \right)^{\frac{1}{2}} n^2. \notag
\end{align}
It follows from the maximality of $K$ that $|K| \ge |G'|$.
\end{proof}

\begin{claim}\label{discon-r-size-of-Vi'}
$\left| |V'_i| - \frac{n}{r}\right| < 4r^{\frac{r+2}{4}}\left( \delta + \xi \right)^{\frac{1}{4}}n$ for all $i \in [r]$.
\end{claim}
\begin{proof}
Fix $ i\in [r]$ and let $\beta = |V_i'|$.
Then,
\begin{align}
\beta \left( n-\beta \right) + \frac{r-2}{2(r-1)}\left( n-\beta \right)^2 \ge |K|
\overset{\text{Claim \ref{discon-r-size-of-K-G'}}}{>}
\frac{r-1}{2r}n^{2} - 2r^{\frac{r+2}{2}}\left( \delta + \xi \right)^{\frac{1}{2}} n^2, \notag
\end{align}
which implies
\begin{align}
\frac{n}{r} - 4r^{\frac{r+2}{4}}\left( \delta + \xi \right)^{\frac{1}{4}} n
< \beta <
\frac{n}{r} + 4r^{\frac{r+2}{4}}\left( \delta + \xi \right)^{\frac{1}{4}} n. \notag
\end{align}
\end{proof}

Define
\begin{align}
B_{K} = \left\{\{u,v\} \in K: \{u,v\} \subset V'_i \text{ for some } i\in [r] \right\}, \notag
\end{align}
\begin{align}
B_{G'} = \left\{\{u,v\} \in G': \{u,v\} \subset V_i \text{ for some } i\in [r] \right\}, \notag
\end{align}
\begin{align}
M_{K} = \left\{\{u,v\} \in \binom{V(\mathcal{H})}{2} \setminus K: u \in V'_i, v \in V'_{j} \text{ for some } i, j \in [r]\text{ and } i\neq j \right\}, \notag
\end{align}
and
\begin{align}
M_{G'} = \left\{\{u,v\} \in \binom{V(\mathcal{H})}{2} \setminus G': u \in V_i, v \in V_{j} \text{ for some } i, j \in [r]\text{ and } i\neq j \right\}. \notag
\end{align}
Sets in $B_K$ (resp. $B_G'$) are called bad edges in $K$ (resp. $G'$) and sets in $M_{K}$ (resp. $M_{G'}$) are called missing edges
in $K$ (resp. $G'$).
It follows from Claim \ref{discon-r-size-of-K-G'} and Tur\'{a}n's theorem that
\begin{align}\label{discon-r-size-of-MK}
|M_{K}| < 2r^{\frac{r+2}{2}}\left( \delta + \xi \right)^{\frac{1}{2}} n^2,
\end{align}
and
\begin{align}\label{discon-r-size-of-MG'}
|M_{G'}| < 2r^{\frac{r+2}{2}}\left( \delta + \xi \right)^{\frac{1}{2}} n^2.
\end{align}

\begin{claim}\label{discon-r-B-M}
$|B_{K}| \ge \frac{(r-3)!}{r^{r-2}} |M_{K}| + \epsilon (r-3)! n^{2}$.
\end{claim}
\begin{proof}
Let $\mathcal{K}$ be the complete $r$-partite $(r-1)$-graph with parts $V'_1 \cup \cdots \cup V'_{r}$.
Then,
\begin{align}
|\mathcal{K} \setminus \partial \mathcal{H}|
&\overset{\text{Claim \ref{discon-r-size-of-Vi'}}}{\ge}
\frac{1}{\binom{r-1}{2}}|M_{K}| \binom{r-2}{r-3}\left( \frac{n}{r} - 4r^{\frac{r+2}{4}}\left( \delta + \xi \right)^{\frac{1}{4}}n \right)^{r-3}  \notag \\
&  = \frac{2|M_{K}|}{r-1}\left( \frac{n}{r} - 4r^{\frac{r+2}{4}}\left( \delta + \xi \right)^{\frac{1}{4}}n \right)^{r-3}, \notag
\end{align}
which implies
\begin{align}
& |B_{K}|\binom{n}{r-3} + \frac{n^{r-1}}{r^{r-2}} - \frac{2|M_{K}|}{r-1}\left( \frac{n}{r} - 4r^{\frac{r+2}{4}}\left( \delta + \xi \right)^{\frac{1}{4}}n \right)^{r-3} \notag \\
 \ge & |\partial \mathcal{H} \setminus \mathcal{K}| + |\partial \mathcal{H} \cap \mathcal{K}|
 = |\partial \mathcal{H}|
 \overset{(\ref{discon-D-r-size-shadow-H})}{=} \frac{n^{r-1}}{r^{r-2}} + \epsilon n^{r-1}, \notag
\end{align}
and it follows that
\begin{align}
|B_{K}| & \ge (r-3)! \left( \frac{2|M_{K}|}{r-1}\left( \frac{1}{r} - 4r^{\frac{r+2}{4}}\left( \delta + \xi \right)^{\frac{1}{4}} \right)^{r-3}+ \epsilon n^{2} \right) \notag \\
& > (r-3)! \left( \frac{2}{r-1}\left( \frac{1}{r^{r-3}} - 8(r-3)r^{\frac{r+2}{4}}\left( \delta + \xi \right)^{\frac{1}{4}} \right)|M_K|+ \epsilon n^{2} \right) \notag \\
& > \frac{(r-3)!}{r^{r-2}} |M_K| + \epsilon (r-3)! n^{2}. \notag
\end{align}
\end{proof}

Let
\[
S_{K} = \left\{ \{u,v\}: u \in V'_i, v \in V'_j \text{ for some }i, j \in [r], i\neq j, d_{\mathcal{H}}(uv) \le r^3 \binom{n}{r-3} \right\},
\]
and
\[
S_{G'} = \left\{ \{u,v\}: u \in V_i, v \in V_j \text{ for some }i, j \in [r], i\neq j, d_{\mathcal{H}}(uv) \le r^3 \binom{n}{r-3} \right\}.
\]

\begin{claim}\label{discon-r-upper-bound-SK}
$|S_K| <  \frac{r^r}{(r-2)!} |B_K| + \delta r^r n^2$.
\end{claim}
\begin{proof}
Similar to the proof of Claim \ref{discon-r-B-M},
\begin{align}\label{discon-r-upper-bound-H}
|\mathcal{H}| & = \frac{1}{\binom{r}{2}} \sum_{uv \in G}d(u,v) \notag\\
& \le \sum_{uv \in B_K}d(u,v) + \frac{n^r}{r^r} - \frac{1}{\binom{r}{2}}|S_{K}|\left( \left(\frac{n}{r} - 4r^{\frac{r+2}{4}}\left( \delta + \xi \right)^{\frac{1}{4}} n \right)^{r-2} - r^3\binom{n}{3}\right)\notag \\
& <  |B_K|\binom{n}{r-2} + \frac{n^r}{r^r} - \frac{1}{\binom{r}{2}} |S_{K}| \frac{n^{r-2}}{2r^{r-2}} \notag \\
& <  \frac{|B_K|}{(r-2)!}n^{r-2} + \frac{n^r}{r^r} -\frac{|S_K|}{r^{r}} n^{r-2} \notag
\end{align}
which implies
\begin{align}
\frac{|B_K|}{(r-2)!}n^{r-2} + \frac{n^r}{r^r} -\frac{|S_K|}{r^{r}} n^{r-2}
> |\mathcal{H}|
\overset{(\ref{discon-D-r-size-H})}{>} \frac{n^r}{r^r} - \delta n^r, \notag
\end{align}
and it follows that
\begin{align}
|S_{K}| < \frac{r^r}{(r-2)!} |B_K| + \delta r^r n^2. \notag
\end{align}
\end{proof}

\begin{claim}\label{discon-r-upper-bound-SG'}
$|S_{G'}| < r^{r}(\delta + \xi)n^{2}$.
\end{claim}
\begin{proof}
Similar to the proof of Claim \ref{discon-r-B-M},
\begin{align}
\frac{n^{r}}{r^r} - \frac{1}{\binom{r}{2}}|S_{G'}|\left( \left( \frac{n}{r} - 4r^{\frac{r+2}{4}}\left( \delta + \xi \right)^{\frac{1}{4}} n \right)^{r-2} - r^{3}\binom{n}{r-3} \right) \ge |\mathcal{H}'|
\overset{(\ref{discon-D-r-size-H'})}{>} \frac{n^r}{r^r} - (\delta + \xi)n^{r}, \notag
\end{align}
which implies
\begin{align}
|S_{G'}| & \le \frac{ (\delta + \xi)n^{r}}{\frac{1}{\binom{r}{2}}\left( \left( \frac{n}{r} - 4r^{\frac{r+2}{4}}\left( \delta + \xi \right)^{\frac{1}{4}} n \right)^{r-2} - r^{3}\binom{n}{r-3} \right)}
< \frac{ (\delta + \xi)n^{r}}{\frac{1}{\binom{r}{2}}\frac{n^{r-2}}{2r^{r-2}}} < r^{r}(\delta + \xi)n^{2}. \notag
\end{align}
\end{proof}

\begin{claim}\label{discon-r+1-partite-subgp-of-K}
Suppose that $K$ (resp. $G'$) contains a subgraph on $\{u\} \cup \{v\} \cup U_2 \cup \cdots \cup U_r$
with $u,v \in V'_1$ (resp. $u,v \in V_1$) and $U_i \subset V_i'$ (resp. $U_i \subset V_i$) and $|U_i| \ge 2(\delta + \xi)^{\frac{1}{8}} n$
for $2 \le i \le r$, such that $uv \in G$ and $u,v$ are adjacent (in $G$) to all vertices in $U_2 \cup \cdots \cup U_r$.
Then all but at most $8r^{\frac{r+2}{2}}(\delta + \xi)^{\frac{3}{8}} n^2$ pairs in $U_i \times U_j$ has degree (in $\mathcal{H}$)
at most $r^3 \binom{n}{r-3}$ for all $\{i,j\} \subset \{2,\ldots, r\}$.
\end{claim}
\begin{proof}
Without loss of generality we may assume that $i = r-1, j = r$.
First, let us show how to obtain a subgraph of $K$ on $\{u\}\cup \{v\}\cup \{v_2\} \cup \cdots \cup \{v_{r-2}\} \cup U_{r-1}^{r-1} \cup U_{r}^{r-1}$
with $v_i \in V'_{i}$ for $2 \le i \le r-2$ and $U_{j}^{r-1} \subset U_{j}$ for $j = r-1,r$, and moreover,
\[
|U_{i}^{r-1}| \ge |U_{i}| -  2r^{\frac{r+2}{2}}(\delta + \xi)^{\frac{3}{8}} n > (\delta + \xi)^{\frac{1}{8}}n, \text{ for } i \in \{r-1,r\}.
\]

Let $U_{i}^{2} = U_{i}$ for $2 \le i \le r$.
We claim that there exists $v_2 \in V'_2$ such that
\[
|N_{K}(v_2) \cap U_{i}^{2}| \ge |U_{i}^{2}| - 2r^{\frac{r+2}{2}}(\delta + \xi)^{\frac{3}{8}} n.
\]
Otherwise,
\begin{align}
|M_{K}| \ge \sum_{u \in U_{2}^{2}}d_{M_{K}}(u)
\ge (\delta + \xi)^{\frac{1}{8}}n \times 2r^{\frac{r+2}{2}}(\delta + \xi)^{\frac{3}{8}} n
>  2r^{\frac{r+2}{2}}(\delta + \xi)^{\frac{1}{2}} n , \notag
\end{align}
which contradicts $(\ref{discon-r-size-of-MK})$.

Now suppose that we have chosen $\{v_2, \ldots, v_{j-1}\}$ for some $2 \le j \le r-3$.
Then, let $U_{i}^{j} = N_{K}(v_{j-1}) \cap U_{i}^{j-1}$ and choose $v_j \in U_j^{j}$ so that
\begin{align}
|N_{K}(v_{j}) \cap U_{i}^{j}|  \ge |U_{i}^{j}| - 2r^{\frac{r+2}{2}}(\delta + \xi)^{\frac{3}{8}} n, \text{ for all } j \le i \le r. \notag
\end{align}
Similar to the argument above, there exists such a vertex $v_j \in U_j^{j}$.

Repeating this process until $j = r-2$.
Then, we obtain a subgraph of $K$ on $\{u\}\cup \{v\}\cup \{v_2\} \cup \cdots \cup \{v_{r-2}\} \cup U_{r-1}^{r-1} \cup U_{r}^{r-1}$
so that $u,v,v_2,\ldots, v_{r-2}$ are adjacent to each other and all vertices in $U_{r-1}^{r-1} \cup U_{r}^{r-1}$,
and
\[
|U_{i}^{r-1}| > |U_{i}| - 2r^{\frac{r+4}{2}}(\delta + \xi)^{\frac{3}{8}} n, \text{ for } i \in \{r-1,r\}.
\]

Let
\[
L = \left\{\{w,w'\}\in U_{r-1}^{r-1} \times U_{r}^{r-1} : d_{\mathcal{H}}(ww') \ge r^3 \binom{n}{r-3}   \right\}.
\]
Next, we show that the matching number of $L$, denoted by $\nu(L)$, is at most $r-2$.

Let $ww' \in L$ and
\[
\mathcal{E}_{uv} = \left\{E \in \mathcal{H}: \{u,v\} \in E \right\}.
\]
We claim that every $E \in \mathcal{E}_{uv}$ satisfies $E \cap \{u,v\} \neq \emptyset$.
Indeed, suppose that there exists  $E_{uv} \in \mathcal{E}_{uv}$ with $E \cap \{u,v\} = \emptyset$.
Since $d_{\mathcal{H}}(ww') \ge r^3 \binom{n}{r-3} \ge r^3 \binom{n}{r-3}$ and $n$ is sufficiently large,
by result in \cite{FR13}, $L_{\mathcal{H}}(ww')$ contains at least $r^3$ pairwise disjoint set.
So, we can choose $E_{ww'} \in \mathcal{H}$ such that $\{w,w'\} \in E_{ww'}$ and $E_{ww'} \cap E_{uv} = \emptyset$.
For every $\{a,b\} \subset \{u,v,v_2,\ldots,v_{r-2},w,w'\}$ and $\{a,b\} \not\in \{\{u,v\},\{w,w'\}\}$,
choose $E_{ab} \in \mathcal{H}$ such that $\{a,b\} \subset E_{ab}$.
Let $F_1$ be the $r$-graph with edges set
\[
\{E_{uv},E_{ww'}\} \cup \left\{E_{ab}: \text{$\{a,b\} \subset \{u,v,v_2,\ldots,v_{r-2},w,w'\}$ and $\{a,b\} \not\in \{\{u,v\},\{w,w'\}\}$} \right\}.
\]
Notice that $F_1 \subset \mathcal{H}$ and $F_1 \in \mathcal{K}_{r+1}^{r}$.
Since $E_{uv} \cap E_{ww'} = \emptyset$, $F_1 \in \mathcal{D}^{r}$, which is a contradiction.
Therefore, every $E \in \mathcal{E}_{uv}$ satisfies $E \cap \{u,v\} \neq \emptyset$.

Suppose that $\nu(L) \ge r-1$ and let $\{w_i w_i' \in L: i \in [r-1] \}$ be a set of pairwise disjoint edges.
The argument above implies that $E_{uv} \cap \{w_i,w_i'\} \neq \emptyset$ for all $i \in [r-1]$,
which is impossible since $E_{uv}$ is an $r$-set.
Therefore, $\nu(L) \le r-2$ and it follows that $|L| \le (r-2)n$, and hence
there are at most
\begin{align}
(r-2)n + 2r^{\frac{r+4}{2}}(\delta + \xi)^{\frac{3}{8}} n
\times 2 \left( \frac{n}{r} + 4r^{\frac{r+2}{4}}\left( \delta + \xi \right)^{\frac{1}{4}}n \right)
<  8r^{\frac{r+2}{2}}(\delta + \xi)^{\frac{3}{8}} n^2 \notag
\end{align}
pairs in $U_{r-1} \times U_{r}$ have degree greater than $r^{3}\binom{n}{r-3}$.
\end{proof}

\begin{claim}\label{discon-r-upper-bound-BK}
$|B_{K}| \le 128r^{\frac{r+4}{2}}\left( \delta + \xi \right)^{\frac{1}{2}} n^2$.
\end{claim}
\begin{proof}
Suppose that $|B_{K}| > 128r^{\frac{r+4}{2}}\left( \delta + \xi \right)^{\frac{1}{2}} n^2$.
Then, by the maximality of $K$,
\begin{align}
|B_{G'}| \ge |B_K| > 128r^{\frac{r+4}{2}}\left( \delta + \xi \right)^{\frac{1}{2}} n^2. \notag
\end{align}
Let $B_{G'}^{1} = B_{G'} \cap \binom{V_1}{2}$ and without loss of generality we may assume that
\begin{align}
|B_{G'}^{1}| \ge \frac{|B_{G'}|}{r} > 128r^{\frac{r+2}{2}}\left( \delta + \xi \right)^{\frac{1}{2}} n^2. \notag
\end{align}
Then, by Claim \ref{discon-r-size-of-Vi}, $B_{G'}^{1}$ contains at least
\begin{align}
\frac{|B_{G'}^{1}|}{2 \left( \frac{n}{r} + 2 r^{\frac{r-2}{2}}\left( \delta + \xi \right)^{\frac{1}{2}} \right)}
> \frac{128r^{\frac{r+2}{2}}\left( \delta + \xi \right)^{\frac{1}{2}} n^2}{4 n /r}
= 32r^{\frac{r+4}{2}}\left( \delta + \xi \right)^{\frac{1}{2}} n \notag
\end{align}
pairwise disjoint edges.

Let $uv \in B_{G'}^{1}$ and $N_i = N_{G'}(u) \cap N_{G'}(v) \cap V_i$ for $2 \le i \le r$.
Suppose that $|N_{G'}(u) \cap N_{G'}(v)| \ge \frac{2r-3}{2r}n$, then $\min\{|N_i| : 2\le i \le r\} \ge n/3r$.
Applying Claim \ref{discon-r+1-partite-subgp-of-K} to the induced subgraph of $G'$ on $\{u\}\cup \{v\} \cup N_2 \cup \cdots N_r$ gives
\begin{align}
|S_{G'}| \ge \left(\frac{n}{3r} \right)^2 - 8r^{\frac{r+2}{2}}(\delta + \xi)^{\frac{3}{8}} n^2
> \frac{n^2}{18r^2}, \notag
\end{align}
which contradicts Claim \ref{discon-r-upper-bound-SG'}.
Therefore, $|N_{G'}(u) \cap N_{G'}(v)| < \frac{2r-3}{2r}n$, and hence
\[
d_{G'}(u) + d_{G'}(v) = |N_{G'}(u) \cup N_{G'}(v)| + |N_{G'}(u) \cap N_{G'}(v)| < \sum_{i=2}^{r}|V_{i}| + \frac{2r-3}{2r}n.
\]
It follows that
\begin{align}
|M_{G'}| & \ge \sum_{uv \in B_{G'}^{1}} \left(d_{M_{G'}}(u) + d_{M_{G'}}(v) \right)  \notag \\
& \ge 32r^{\frac{r+4}{2}}\left( \delta + \xi \right)^{\frac{1}{2}} n \left( (r-1)\left(\frac{n}{r} - 4r^{\frac{r+2}{4}}\left( \delta + \xi \right)^{\frac{1}{4}}n \right) - \frac{2r-3}{2r}n \right) \notag \\
& > 32r^{\frac{r+4}{2}}\left( \delta + \xi \right)^{\frac{1}{2}} n \times \frac{n}{4r} \notag \\
& >  8r^{\frac{r+2}{2}}\left( \delta + \xi \right)^{\frac{1}{2}} n^2
\overset{(\ref{discon-r-size-of-MG'})}{>} |M_{G'}|, \notag
\end{align}
a contradiction.
\end{proof}

By Claims \ref{discon-r-upper-bound-SK} and \ref{discon-r-upper-bound-BK},
\begin{align}\label{discon-r-upper-bound-SK-2}
|S_K| & < \frac{r^r}{(r-2)!} |B_K| + \delta r^r n^2
< \frac{256 r^{\frac{3r+4}{2}}}{(r-2)!} \left( \delta + \xi \right)^{\frac{1}{2}} n^2.
\end{align}

Let $B_{K}^{1} = B_{K} \cap \binom{V_1'}{2}$ and without of generality we may assume that
\begin{align}\label{discon-r-upper-bound-BK1}
|B_{K}^1| \ge \frac{|B_{K}|}{r}.
\end{align}
Let $\Delta = \max\{d_{B_{K}^1}(v): v\in V_1'\}$.

\medskip

\noindent \textbf{Case 1:} $\Delta > 4 r^{\frac{r+10}{4}} (\delta + \xi)^{\frac{1}{8}} n$. \\
Let $v \in V_1'$ with $d_{B_{K}^1}(v) = \Delta$ and let $N_{i} = N_{K}(v) \cap V_i'$.
Since $K$ is a maximum $r$-partite subgraph of $G$, $|N_{i}| \ge \Delta$ for all $i \in [t]$.
Fix $u \in N_1$ and let $U_{i} = N_{K}(u) \cap N_i$ for $2 \le i \le r$.
If $|U_i| \ge \Delta /2 > 2 r^{\frac{r+10}{4}} (\delta + \xi)^{\frac{1}{8}} n$ for all $2 \le i \le r$,
then by Claim \ref{discon-r+1-partite-subgp-of-K},
\begin{align}
|S_{K}| \ge \left( 2 r^{\frac{r+10}{4}} (\delta + \xi)^{\frac{1}{8}} n \right)^2 - 8r^{\frac{r+2}{2}}(\delta + \xi)^{\frac{3}{8}} n^2
> 2 r^{\frac{r+10}{2}} (\delta + \xi)^{\frac{1}{4}} n^2, \notag
\end{align}
which contradicts $(\ref{discon-r-upper-bound-SK-2})$.
Therefore, $|U_i| <  \Delta /2$ for some $2 \le i \le r$, and hence
\begin{align}
|M_{K}| & = \sum_{u \in N_1}d_{M_{K}}(u)
\ge \Delta \left( \Delta - \frac{\Delta}{2} \right)
> 2 \left( 2 r^{\frac{r+10}{4}} (\delta + \xi)^{\frac{1}{8}} n \right)^2
= 8 r^{\frac{r+10}{2}} (\delta + \xi)^{\frac{1}{4}} n^2 , \notag
\end{align}
which contradicts $(\ref{discon-r-size-of-MK})$.

\bigskip

\noindent \textbf{Case 2:} $\Delta \le 4 r^{\frac{r+10}{4}} (\delta + \xi)^{\frac{1}{8}} n$. \\
Then, $B_{K}^{1}$ contains at least $\frac{|B_{K}^{1}|}{8 r^{\frac{r+10}{4}} (\delta + \xi)^{\frac{1}{8}} n}$ pairwise disjoint sets.
Let $uv \in B_{K}^{1}$.
If $|N_{K}(u) \cap N_{K}(v)| \ge \frac{2r-3}{2r}n$, then similar to the proof of Claim \ref{discon-r-upper-bound-BK},
\begin{align}
|S_{K}| \ge \left( \frac{n}{3r} \right)^2 - 8r^{\frac{r+2}{2}}(\delta + \xi)^{\frac{3}{8}} n^2
> \frac{n^2}{18r^2}, \notag
\end{align}
which contradicts $(\ref{discon-r-upper-bound-SK-2})$.
Therefore, $|N_{K}(u) \cap N_{K}(v)| < \frac{2r-3}{2r}n$, and hence
\[
d_{K}(u) + d_{K}(v) = |N_{K}(u) \cup N_{K}(v)| + |N_{K}(u) \cap N_{K}(v)| < \sum_{i=2}^{r}|V_{i}'| + \frac{2r-3}{2r}n.
\]
It follows that
\begin{align}
|M_{K}| & \ge \sum_{uv \in B_{K}^{1}} \left(d_{M_K}(u) + d_{M_K}(v) \right)  \notag \\
& \ge \frac{|B_{K}^{1}|}{8 r^{\frac{r+10}{4}} (\delta + \xi)^{\frac{1}{8}} n} \left( (r-1)\left(\frac{n}{r} - 4r^{\frac{r+2}{4}}\left( \delta + \xi \right)^{\frac{1}{4}}n \right) - \frac{2r-3}{2r}n \right) \notag \\
& > \frac{|B_{K}^{1}|}{8 r^{\frac{r+10}{4}} (\delta + \xi)^{\frac{1}{8}} n} \times \frac{n}{4r} \notag \\
& >  \frac{|B_{K}^{1}|}{32 r^{\frac{r+14}{4}} (\delta + \xi)^{\frac{1}{8}}}
\overset{\text{Claim \ref{discon-r-B-M}}}{>} |M_K|, \notag
\end{align}
a contradiction.
\end{proof}

\bibliographystyle{abbrv}
\bibliography{feasible_region}
\end{document}